\RequirePackage{fix-cm}
\documentclass[smallextended]{svjour3}       
\usepackage{graphicx}
\usepackage{multirow}

\usepackage[utf8]{inputenc} 
\usepackage[T1]{fontenc}    
\usepackage{url}            
\usepackage{booktabs}       
\usepackage{amsfonts}       
\usepackage{nicefrac}       
\usepackage{microtype}      
\usepackage{xcolor}         
\usepackage{algpseudocode}
\usepackage{algorithm}

\usepackage{booktabs}
\usepackage{tabularx}
\usepackage{array}
\usepackage{makecell}
\usepackage{pifont}
\usepackage{xcolor}
\usepackage{hyperref}
\usepackage{enumitem,amsfonts,amssymb,mathtools,bm,pifont,listings,xcolor,threeparttable,xcolor}

\usepackage{algorithm}
\usepackage{algpseudocode}

\usepackage{amsmath}
\usepackage{comment}
\usepackage{subcaption}

\newtheorem{assumption}{Assumption}

\DeclareMathOperator{\qf}{qf}

\DeclareMathOperator{\dist}{dist}

\DeclareMathOperator{\Retr}{Retr}

\def\x{\mathbf{x}}\def\y{\mathbf{y}}

\def\A{\mathbf{A}}\def\B{\mathbf{B}}\def\C{\mathbf{C}}\def\G{\mathbf{G}}\def\H{\mathbf{H}}\def\K{\mathbf{K}}\def\M{\mathbf{M}}\def\N{\mathbf{N}}\def\Q{\mathbf{Q}}\def\R{\mathbf{R}}\def\S{\mathbf{S}}\def\V{\mathbf{V}}\def\W{\mathbf{W}}\def\X{\mathbf{X}}\def\Y{\mathbf{Y}}\def\Z{\mathbf{Z}}

\def\AA{\mathcal{A}}\def\BB{\mathcal{B}}\def\CC{\mathcal{C}}\def\DD{\mathcal{D}}\def\EE{\mathcal{E}}\def\FF{\mathcal{F}}\def\OO{\mathcal{O}}\def\SS{\mathcal{S}}\def\XX{\mathcal{X}}\def\YY{\mathcal{Y}}\def\ZZ{\mathcal{Z}}

\def\EEE{\mathbb{E}}\def\GGG{\mathbb{G}}\def\KKK{\mathbb{K}}\def\LLL{\mathbb{L}}

\def\trans{^\mathsf{T}}\def\Rn{\mathbb{R}}  \def\zero{\mathbf{0}}  
 \def\trans {^\mathsf{T}} \def\fro {\mathsf{F}}

\newcommand{\bfit}[1]{\textit{\textbf{#1}}}
\newcommand{\beq}{\begin{eqnarray}}
\newcommand{\eeq}{\end{eqnarray}}
\newcommand{\la}{\langle}
\newcommand{\ra}{\rangle}
\newcommand{\noi}{\noindent}
\newcommand{\nn}{\nonumber}
\def\noi{\noindent}
\def\nn{\nonumber}
\def\dist{{\rm{dist}}}
\newcommand{\step}[1]{\text{\ding{\numexpr#1+171\relax}}}

\def\Deltas{\boldsymbol{\Delta}}

\def\etaup{\overline{\eta}}

\def\ts{\textstyle}

\begin{document}

\title{Recursive Gradient Retraction Methods for Stochastic Nonconvex Optimization  on Stiefel Manifolds}

\titlerunning{{\sf RGR} for Stochastic Nonconvex Optimization on Stiefel Manifolds} 
\author{Zhijie Yuan         \and
	Lei Sun$^*$ \thanks {$^*$Corresponding author}  \and
	Ganzhao Yuan
}


\institute{Zhijie Yuan \at
	School of Systems Science and Engineering, Sun Yat-sen University, Guangzhou, China   \\
	\email{yuanzhj5@mail2.sysu.edu.cn}           
	\and
	Lei Sun \at School of Systems Science and Engineering, Sun Yat-sen University, Guangzhou, China      \\
	\email{sunlei8@mail.sysu.edu.cn}
	\and
	Ganzhao Yuan \at Faculty of Computer Science and Artificial Intelligence, Shenzhen University of Advanced Technology, Shenzhen, China        \\
	\email{yuanganzhao@suat-sz.edu.cn}
}

\date{}

\maketitle

\begin{abstract}
	Nonconvex optimization on Stiefel manifolds underlies many problems in machine learning and deep learning. However, existing stochastic Riemannian retraction methods often rely on stepsizes that require prior knowledge of the gradient Lipschitz constant, a prescribed iteration horizon, or restrictive batch-size choices. We develop {\sf RGR}, a recursive gradient retraction method for stochastic nonconvex optimization on Stiefel manifolds. Instead of recursively updating gradient estimators in changing tangent spaces, {\sf RGR} constructs them in the Euclidean space and subsequently maps them to tangent space, thereby avoiding vector transport. The method comprises three variants: {\sf RGR-SPIDER} for finite-sum optimization and {\sf RGR-SRM} and {\sf RGR-EMA} for online stochastic optimization. {\sf RGR-SPIDER} is Lipschitz-free and horizon-free, as its adaptive stepsize requires neither prior knowledge of the gradient Lipschitz constant nor a prescribed iteration horizon. {\sf RGR-SRM} and {\sf RGR-EMA} are mini-batch and horizon-free methods whose convergence guarantees do not rely on prescribed large batch sizes and whose parameters can be chosen without prior knowledge of the total iteration budget. We show that {\sf RGR-SPIDER} reaches an $\varepsilon$-critical point using $\mathcal{O}(\varepsilon^{-2})$ iterations and $\mathcal{O}(N+\sqrt{N}\varepsilon^{-2})$ component-gradient evaluations, where $N$ denotes the number of component functions. For the online setting, {\sf RGR-SRM} and {\sf RGR-EMA} achieve iteration complexities of $\tilde{\mathcal{O}}(\varepsilon^{-3})$ and $\tilde{\mathcal{O}}(\varepsilon^{-4})$, respectively. Experiments on principal component analysis, low-rank matrix completion, and orthogonal convolutional neural networks demonstrate faster empirical convergence than methods that construct recursive gradient estimators directly in Riemannian tangent spaces.
\end{abstract}

\section{Introduction}
\label{intro}

Stochastic nonconvex optimization problems on Stiefel manifolds arise throughout machine learning, deep learning, and scientific computing. The Stiefel manifold $\mathrm{St}(n,r)\triangleq\{\mathbf{X}\in\mathbb{R}^{n\times r}:\mathbf{X}^{\top}\mathbf{X}=\mathbf{I}_r, n \ge r\}$ provides a natural geometric representation of matrices with orthonormal columns, thereby preserving orthogonality throughout the
optimization process~\cite{Absil2008,EdelmanAS98}. In classical subspace learning, principal component analysis (PCA) and its sparse variants seek orthonormal loading vectors that preserve the dominant variation in high-dimensional data~\cite{journee2010generalized,lu2012augmented}. After
whitening, independent component analysis can similarly be reduced to estimating an orthogonal transformation that separates statistically independent latent sources~\cite{comon1994independent,hyvarinen2000independent}. In deep learning, optimization on Stiefel manifolds has been employed to enforce orthogonality in weight matrices and reshaped convolutional filters, thereby controlling signal propagation, reducing feature
redundancy, and improving training stability~\cite{cho2017riemannian,xie2017all,bansal2018can,wang2020orthogonal}. Stiefel manifolds also play a central role in matrix factorization. Factorized formulations of low-rank matrix completion introduce orthonormal subspace factors to represent the row and column spaces of the recovered matrix~\cite{vandereycken2013low,cambier2016robust}, whereas orthogonal nonnegative matrix factorization combines orthogonality with nonnegativity to obtain interpretable, clustering-oriented factors ~\cite{ding2006orthogonal,jiang2022exact}. Fisher discriminant analysis and its kernel extensions impose orthogonality or generalized orthogonality when constructing low-dimensional representations that preserve class-discriminative information; after an appropriate whitening or reparameterization, these constraints take a standard Stiefel form~\cite{fisher1936use,mika1999fisher}. In dictionary learning, spherical and orthogonal-group formulations are used to identify complete dictionaries from sparsely generated observations~\cite{sun2017complete,zhai2020complete}.

Motivated by the above applications, this paper focuses on the following finite-sum and online stochastic optimization problems on the Stiefel manifold:
\beq\label{eq:main}
\min_{\X \in \mathrm{St}(n,r)} f(\X) \triangleq
\begin{cases} 
	\tfrac{1}{N} \sum_{i=1}^N f_i(\X) & (\text{finite-sum}), \\ 
	\EEE_{\xi \sim \mathcal{D}}[f(\X; \xi)] & (\text{online}),
\end{cases}
\eeq
where $\mathrm{St}(n,r)= \{\X \in \mathbb{R}^{n \times r} \mid \X^\top \X = \mathbf{I}_r, n \ge r\}$. In the finite-sum setting, $N$ is the number of samples, while in the online setting, $\xi$ is a random variable, such as a training data sample, drawn from a distribution $\mathcal{D}$. We assume that the objective is differentiable and possibly nonconvex. Moreover, we assume individual smoothness and bounded gradients: for some constants $L>0$ and $C_f>0$, each $f_i(\X)$ and each stochastic function $f(\cdot;\xi)$ are $L$-smooth and their gradients are uniformly bounded by $C_f$ in the Euclidean space.

$\bullet$ \textbf{Recursive Gradient Estimation}. 
In many machine-learning and data-science applications, the objective takes a finite-sum form involving a large number $N$ of component functions. Evaluating the full gradient,
$\nabla f(\x)=\frac{1}{N}\sum_{i=1}^{N}\nabla f_i(\x)$, at every iteration can therefore be computationally prohibitive. Stochastic methods reduce this cost by estimating the gradient from a mini-batch, but the resulting sampling variance may slow convergence and destabilize optimization. Recursive gradient estimation addresses this issue by combining historical gradient estimates with newly sampled stochastic gradients. Variance-reduction ({\sf VR}) methods form a prominent class of such techniques~\cite{NIPS2013TongZ,schmidt2013minimizing}; they construct corrected stochastic gradient estimators with reduced variance, thereby improving convergence while retaining a low per-iteration cost. Representative methods include {\sf SVRG}~\cite{NIPS2013TongZ,li2018simple}, which employs a double-loop structure: each outer loop computes the full gradient at a reference point, whereas each inner iteration uses this snapshot gradient to correct a newly sampled stochastic gradient.
{\sf SARAH} \cite{nguyen2017sarah} introduces a recursive variance reduction scheme that updates the estimator iteratively without requiring full snapshots, achieving optimal rates for nonconvex finite-sum problems and naturally extending to minimax optimization. {\sf SPIDER} \cite{fang2018spider} can be viewed as a normalized variant of {\sf SARAH} designed for nonconvex settings, incorporating stepsize normalization to ensure bounded variance and near-optimal sample complexity. {\sf PAGE} \cite{li2021page} is a probabilistic method that adaptively switches between full-gradient and stochastic updates based on a coin flip. Additionally, momentum-infused {\sf VR} methods have attracted considerable attention due to their ability to accelerate convergence in noisy environments. For instance, {\sf STORM} \cite{cutkosky2019momentum} and its parameter-free variant {\sf STORM+} \cite{levy2021storm+} integrate momentum-based variance reduction by dynamically adjusting stepsizes using historical gradients, attaining optimal convergence rates for stochastic nonconvex optimization without requiring large batch sizes. The exponential moving average ({\sf EMA}) is another widely used recursive estimation technique and forms a central component of stochastic optimization methods such as {\sf Adam}~\cite{KingmaB14} and {\sf RMSProp}~\cite{hinton2012neural}. It provides a memory-efficient mechanism for aggregating historical gradient information while assigning greater weight to recent observations. In this work, we develop {\sf RGR}, a recursive gradient retraction method for solving Problem~\eqref{eq:main}. In the finite-sum setting, we introduce {\sf RGR-SPIDER} by incorporating {\sf SPIDER}. In the online stochastic setting, we design two recursive estimators based on stochastic recursive momentum ({\sf SRM}) and {\sf EMA}, which lead to the variants {\sf RGR-SRM} and {\sf RGR-EMA}, respectively.

$\bullet$ \textbf{Stochastic Riemannian Retraction Methods}. By exploiting the intrinsic geometry on Stiefel manifolds, Problem \eqref{eq:main} can instead be reformulated as an unconstrained optimization problem over Stiefel manifold spaces. Numerous methods are developed for manifold optimization. A classical approach is  Riemannian steepest descent ({\sf R-SD})  \cite{udriste2013convex,pmlr-v49-zhang16b} and Riemannian stochastic gradient descent ({\sf R-SGD}) \cite{bonnabel2013stochastic}, which computes gradients using a mini-batch, substantially reducing the per-iteration computational cost. Nonetheless, {\sf R-SD} costs unstanding computation and {\sf R-SGD} often suffers from slow convergence due to its high variance. Hence, subsequent studies have extended variance reduction techniques from Euclidean spaces to Riemannian manifolds, including  Riemannian stochastic variance reduced gradient ({\sf R-SVRG}) \cite{doi:10.1137/17M1116787,zhang2016riemannian}, Riemannian {\sf SPIDER} ({\sf R-SPIDER}) \cite{zhou2019faster}, Riemannian {\sf PAGE} ({\sf R-PAGE}) \cite{demidovich2024streamlining}, and Riemannian stochastic recursive gradient ({\sf R-SRG}) \cite{kasai2018riemannian}. Recently, a stochastic recursive variance reduced gradient ({\sf SRVRG}) estimator is proposed and extended it to Riemannian manifolds as {\sf R-SRVRG} \cite{liuserena}. They also develop {\sf R-SVRRM} and {\sf R-Hybrid-SGD} by extending the stochastic variance reduced recursive momentum ({\sf SVRRM}) and {\sf Hybrid-SGD} estimators to the Riemannian setting. Moreover, a number of studies have explored adaptive optimization methods on Stiefel manifolds. The Riemannian adaptive stochastic gradient method ({\sf RASA}) \cite{kasai2019riemannian} was proposed for matrix manifolds. {\sf RAMSGrad} \cite{becigneul2018riemannian} and its modified version \cite{sakai2021riemannian} were developed as direct extensions of {\sf AMSGrad} to Riemannian manifolds. Similarly, {\sf cRAMSProp} \cite{roy2018geometry} was introduced and applied to various Riemannian stochastic optimization problems. Tables \ref{tab:finite_sum_methods} and \ref{tab:online_methods} summarize the iteration complexities of existing stochastic Riemannian methods. For finite-sum optimization, most methods use prespecified stepsizes. Although {\sf R-SPIDER}~\cite{zhou2019faster} is adaptive, it constructs estimators in changing tangent spaces via vector transport, and its theoretical stepsize depends on  $\varepsilon$, whereas ours requires neither. For online optimization, many existing methods are not mini-batch methods. {\sf R-SGD}~\cite{hosseini2020alternative}, {\sf R-AMSGrad}~\cite{becigneul2018riemannian}  and {\sf RASA} \cite{kasai2019riemannian} do not employ stochastic gradient estimation, {\sf R-AMSGrad}~\cite{becigneul2018riemannian} is analyzed primarily for geodesically convex objectives, and {\sf  SRM}~\cite{ijcai2021p345} constructs recursive estimators in changing tangent spaces using vector transport. Inspired by the vector-transport-free construction of {\sf S-SVRG}~\cite{jiang2017vector}, {\sf RGR} constructs the recursive estimator in the Euclidean space and then maps it to a tangent direction for the retraction update.  Although {\sf RGR} may not apply to all Riemannian optimization problems, it avoids vector transport between successive tangent spaces, thereby reducing computational overhead and mitigating potential transport-induced approximation errors.

$\bullet$ \textbf{Adaptive Stepsize}. Adaptive stepsize techniques have emerged as powerful tools in stochastic optimization, particularly for addressing the challenges of nonconvex, high-dimensional, and large-scale learning problems. Classic algorithms include {\sf AdaGrad} \cite{duchi2011adaptive,McMahanS10}, {\sf AdaGrad-Norm} \cite{ward2020adagrad}, and {\sf RMSProp} \cite{hinton2012neural}, which adapt the stepsizes according to accumulated historical gradients, thereby favoring sparse updates in high-dimensional spaces. Building upon these methods, {\sf Adam} \cite{KingmaB14} has become one of the most widely used optimizers for training deep neural networks by combining momentum with adaptive learning rates. Numerous extensions have further enhanced the capabilities of {\sf Adam}, including {\sf AdamW} \cite{loshchilov2017decoupled}, {\sf AdaBelief} \cite{zhuang2020adabelief}, {\sf AMSGrad} \cite{reddi2019convergence}, and {\sf Padam} \cite{chen2018closing}. Despite their practical success, many adaptive optimizers are designed for Euclidean spaces and therefore cannot be applied directly to Stiefel manifolds without accounting for their geometry. In this work, we develop an adaptive stepsize strategy for finite-sum optimization on Stiefel manifolds using Riemannian gradient estimates.

$\bullet$ \textbf{Contributions}. Our main contributions are summarized as follows. \textbf{(}\bfit{i}\textbf{)} We propose {\sf RGR}, a recursive gradient retraction method for Problem~\eqref{eq:main} on Stiefel manifolds. {\sf RGR} perform recursive stochastic gradient estimation in the Euclidean space and then map the estimators to tangent directions for retraction updates, avoiding recursive estimation across different tangent spaces, eliminating explicit vector transport, and reducing computational overhead. The method yields three variants: {\sf RGR-SPIDER} for finite-sum problems, and {\sf RGR-SRM} and {\sf RGR-EMA} for online settings (see Section~\ref{sec:prop}). Specifically, {\sf RGR-SPIDER} employs an adaptive {\sf AdaGrad-Norm}-type stepsize, making it both Lipschitz-free and horizon-free: its stepsize does not require prior knowledge of the gradient Lipschitz constant, nor does it depend on a prescribed total number of iterations. For the online stochastic setting, {\sf RGR-SRM} and {\sf RGR-EMA} are mini-batch and horizon-free methods whose convergence guarantees hold for any fixed positive mini-batch size, without requiring it to be sufficiently large, and whose parameter choices do not require prior knowledge of the total number of iterations $T$. In contrast, many existing stochastic Riemannian methods require the batch size to be sufficiently large to establish their iteration complexity guarantees.  \textbf{(}\bfit{ii}\textbf{)} We establish convergence guarantees for all proposed variants under standard assumptions. In particular, to obtain an $\varepsilon$-approximate critical point, {\sf RGR-SPIDER} requires $\mathcal{O}(N+\sqrt{N}/\varepsilon^{2})$ component-gradient complexity, whereas {\sf RGR-SRM} and {\sf RGR-EMA} have iteration complexities of $\tilde{\mathcal{O}}(1/\varepsilon^{3})$ and $\tilde{\mathcal{O}}(1/\varepsilon^{4})$, respectively (see Section~\ref{sec:com}).  \textbf{(}\bfit{iii}\textbf{)} We validate the practical efficiency and robustness of the proposed methods through extensive numerical experiments on benchmark tasks, including PCA, low-rank matrix completion, and orthogonal convolutional neural networks (see Section~\ref{sec:exp}).

\newcolumntype{L}{>{\raggedright\arraybackslash}X}
\newcolumntype{C}{>{\centering\arraybackslash}c}
\newcolumntype{R}{>{\centering\arraybackslash}p{0.46\textwidth}}

\begin{table*}[thb]
	\centering
	
	\label{tab:finite_sum_methods}
	\renewcommand{\arraystretch}{1.5}
	\setlength{\tabcolsep}{5pt}
	\scriptsize
	
	\begin{tabularx}{\textwidth}{@{}|L @{\hspace{-30pt}} C @{\hspace{15pt}}C @{\hspace{10pt}} L|@{}}
		\hline
		Method & Adaptive stepsize & Vector-transport-free & Component-gradient complexity \\
		\hline
		
		{\sf R-SGD} \cite{hosseini2020alternative}
		& \ding{56} & \ding{52}
		& $\OO(\varepsilon^{-4})$ \\

		{\sf R-SVRG} \cite{han2021improved}
		& \ding{56} & \ding{56}
		& $\OO(N+N^{2/3}\varepsilon^{-2})$ \\
		
		{\sf R-AbaSVRG} \cite{han2021improved}
		& \ding{56} & \ding{56}
		& $\OO(N^{2/3}\varepsilon^{-2})$ \\
		
		{\sf R-SRG} \cite{han2021improved}
		& \ding{56} & \ding{56} 
		& $\OO(N+N^{1/2}\varepsilon^{-2})$ \\
		
		{\sf R-AbaSRG} \cite{han2021improved}
		& \ding{56} & \ding{56}
		& $\OO(N^{1/2}\varepsilon^{-2})$ \\
		
		{\sf R-SPIDER} \cite{zhou2019faster}
		& \ding{52} & \ding{56}
		& $\OO(
		\min\{N+N^{1/2}\varepsilon^{-2},\varepsilon^{-3}\}
		)$ \\
		
		{\sf R-PAGE} \cite{demidovich2024streamlining}
		& \ding{56} & \ding{56}
		& $\OO(N+N^{1/2}\varepsilon^{-2})$ \\
		
		{\sf R-SRM} \cite{liuserena}
		& \ding{56} & \ding{56}
		& $\OO(\min\{N,\varepsilon^{-1}\}+\varepsilon^{-3})$ \\
		
		{\sf R-SRVRG} \cite{liuserena}
		& \ding{56} & \ding{56}
		& $\OO(\min\{N^{1/2}\varepsilon^{-2},\varepsilon^{-3}\})$ \\
		
		{\sf R-SVRRM} \cite{liuserena}
		& \ding{56} & \ding{56}
		& $\OO(\min\{N^{1/2}\varepsilon^{-2},\varepsilon^{-3}\})$ \\

		{\sf RGR-SPIDER} \textcolor{red}{[ours]}
		& \ding{52} & \ding{52}
		& $\OO(N+N^{1/2}\varepsilon^{-2})$ \\
		
		\hline
	\end{tabularx}
	
	\caption{Summary of stochastic algorithms and their complexities for the finite-sum setting of Problem~\eqref{eq:main}. \textbf{Note:} The notation $\OO(\cdot)$ suppresses constant factors in the iteration complexity required to obtain an $\varepsilon$-approximate critical point.}\label{tab:finite_sum_methods}
\end{table*}

\begin{table*}[thb]
	\centering
	
	\renewcommand{\arraystretch}{1.5}
	\setlength{\tabcolsep}{5pt}
	\scriptsize
	
	\begin{tabularx}{\textwidth}{@{}|L @{\hspace{-10pt}} C @{\hspace{20pt}}C @{\hspace{20pt}} C @{\hspace{30pt}}L|@{}}
		\hline
		Method & Mini-batch & Vector-transport-free & SGE & Iteration complexity \\
		\hline
		
		{\sf R-SGD} \cite{hosseini2020alternative}
		& \ding{52} & \ding{52} & \ding{56}
		& $\OO(\varepsilon^{-4})$ \\
		
		{\sf RAMSGrad} \cite{becigneul2018riemannian}
		& \ding{52}& \ding{52} & \ding{56}
		& $\OO^*(\varepsilon^{-2})$ \\
		
		{\sf RASA} \cite{kasai2019riemannian}
		& \ding{52} & \ding{52} & \ding{56}
		& $\tilde{\OO}(\varepsilon^{-4})$ \\

		{\sf R-SVRG} \cite{han2021improved}
		& \ding{56} & \ding{56} & \ding{52}
		& $\OO(\varepsilon^{-10/3})$ \\
		
		{\sf R-AbaSVRG} \cite{han2021improved}
		& \ding{56} & \ding{56} & \ding{52}
		& $\OO(\varepsilon^{-10/3})$ \\
		
		{\sf R-SRG} \cite{han2021improved}
		& \ding{56} & \ding{56} & \ding{52}
		& $\OO(\varepsilon^{-3})$ \\
		
		{\sf R-AbaSRG} \cite{han2021improved}
		& \ding{56} & \ding{56} & \ding{52}
		& $\OO(\varepsilon^{-3})$ \\
		
		{\sf R-SPIDER} \cite{zhou2019faster}
		& \ding{56} & \ding{56} & \ding{52}
		& $\OO(\varepsilon^{-3})$ \\
		
		{\sf R-PAGE} \cite{demidovich2024streamlining}
		& \ding{56} & \ding{56} & \ding{52}
		& $\OO(\varepsilon^{-3})$ \\
		
		{\sf RSRM} \cite{ijcai2021p345}
		& \ding{52} & \ding{56} & \ding{52}
		& $\tilde{\OO}(\varepsilon^{-3})$ \\
		
		{\sf R-SRM} \cite{liuserena}
		& \ding{56} & \ding{56} & \ding{52}
		& $\OO(\varepsilon^{-3})$ \\
		
		{\sf R-SRVRG} \cite{liuserena}
		& \ding{56} & \ding{56} & \ding{52}
		& $\OO(\varepsilon^{-3})$ \\
		
		{\sf R-SVRRM} \cite{liuserena}
		& \ding{56} & \ding{56} & \ding{52}
		& $\OO(\varepsilon^{-3})$ \\

		{\sf RGR-SRM} \textcolor{red}{[ours]}
		& \ding{52} & \ding{52} & \ding{52}
		& $\tilde{\OO}(\varepsilon^{-3})$ \\
		
		{\sf RGR-EMA} \textcolor{red}{[ours]}
		& \ding{52} & \ding{52} & \ding{52}
		& $\tilde{\OO}(\varepsilon^{-4})$ \\
		
		\hline
	\end{tabularx}
	
	\caption{Summary of stochastic algorithms and their complexities for the online setting of Problem~\eqref{eq:main}.
	 \textbf{Note:} SGE denotes stochastic gradient estimation. The notation $\OO(\cdot)$ suppresses constant factors, whereas $\tilde{\OO}(\cdot)$ additionally hides polylogarithmic factors. The notation $\OO^*(\cdot)$ indicates a complexity guarantee established for convex problems.} \label{tab:online_methods}
\end{table*}

\section{Applications}\label{sec:app}

This section illustrates three representative applications of stochastic optimization on Stiefel manifolds: principal component analysis, low-rank matrix completion, and orthogonally constrained convolutional neural networks. These applications arise in machine learning and deep learning and cover both finite-sum and stochastic settings.

$\bullet$ \textbf{Principal Component Analysis (PCA).}
PCA is a fundamental technique for dimensionality reduction and feature
extraction. Its sparse and structured variants are also widely studied in
machine learning and data analysis
\cite{journee2010generalized,lu2012augmented}.
Let $\A\in\mathbb{R}^{N\times n}$ be a centered data matrix, where each
row $\A_i\in\mathbb{R}^{1\times n}$ represents one sample. PCA seeks an
$r$-dimensional subspace that minimizes the reconstruction error, or
equivalently, maximizes the variance explained by the projected data.
The PCA problem can then be formulated as the following finite-sum optimization problem:

\begin{equation}
	\min_{\X \in\mathrm{St}(n,r)} 
	\tfrac{1}{2N}\ts 	\sum_{i=1}^{N}
	\left\|
	\A_i-\A_i\X\X^\top
	\right\|_{\fro}^2. \nn
\end{equation}

\noi For a mini-batch $\xi\subseteq\{1,\ldots,N\}$ sampled uniformly at random, the stochastic PCA problem is formulated as follows:
\begin{equation}
	\min_{\X\in\mathrm{St}(n,r)}
	\EEE_{\xi}\left[
	\tfrac{1}{2|\xi|}
	\ts\sum_{i\in\xi}
	\left\|
	\A_i-\A_i\X\X^\top
	\right\|_{\fro}^2
	\right]. \nn
\end{equation}

$\bullet$ \textbf{Low-Rank Matrix Completion (LRMC).}
LRMC aims to recover a low-rank matrix from a subset of its observed
entries and has applications in recommender systems, computer vision,
and sensor-network reconstruction
\cite{cambier2016robust,vandereycken2013low}.
Let $\A\in\mathbb{R}^{N\times n}$ be a partially observed matrix, and
let $\Omega$ denote the set of observed indices. Define
$\mathcal{P}_{\Omega}$ by $[\mathcal{P}_{\Omega}(\Z)]_{ij} = \begin{cases}
	\Z_{ij}, & (i,j)\in\Omega,\\
	0,       & (i,j)\notin\Omega.
\end{cases}
$ Any matrix of rank at most $r$ can be represented as $\Y\X^\top$, where
$\Y\in\mathbb{R}^{N\times r}$ and
$\X\in\mathrm{St}(n,r)$. A normalized factorized LRMC model is therefore $\min_{\substack{\X\in\mathrm{St}(n,r),
		\Y}}
\frac{1}{2N}
\left\|
\mathcal{P}_{\Omega}
\bigl(\A-\Y\X^\top\bigr)
\right\|_{\fro}^2.$ When $\Y$ is fixed, above problem reduces to a finite-sum
optimization problem over the variable $\X$ on the Stiefel manifold:
\begin{equation}\label{eq:app:lrmc:finite}
	\min_{\X\in\mathrm{St}(n,r)} \tfrac{1}{2N} 
	\ts \sum_{i=1}^{N} \left\| \mathcal{P}_{\Omega_i} \bigl(\A_i-\Y_i\X^\top\bigr) \right\|_{\fro}^2, \nn
\end{equation}
where $\Omega_i=\{j:(i,j)\in\Omega\}$ and $\Y_i$ denotes the $i$-th row
of $\Y$. Uniformly sampling a mini-batch $\xi$ of rows gives the equivalent stochastic representation
\begin{equation}\label{eq:app:lrmc:stoc}
	\min_{\X\in\mathrm{St}(n,r)}
	\EEE_{\xi}\left[
	\tfrac{1}{2|\xi|}
	\ts \sum_{i\in\xi}
	\left\|
	\mathcal{P}_{\Omega_i}
	\bigl(\A_i-\Y_i\X^\top\bigr)
	\right\|_{\fro}^2
	\right]. \nn
\end{equation}
Thus, an alternating LRMC method can update $\X$ using a stochastic Stiefel-manifold algorithm while updating $\Y$ separately.

$\bullet$ \textbf{Orthogonal Convolutional Neural Networks (OCNNs).}
Orthogonality constraints and regularizers have been used in convolutional neural networks to reduce correlations among filters and improve the stability of forward and backward propagation \cite{wang2020orthogonal,xie2017all}. Consider a convolutional kernel with spatial dimensions $k_h\times k_w$,
$c_{\mathrm{in}}$ input channels, and $c_{\mathrm{out}}$ output channels.
Let $ d\triangleq k_hk_wc_{\mathrm{in}},
\ p\triangleq c_{\mathrm{out}},$ and let $\W\in\mathbb{R}^{d\times p}$ denote the reshaped kernel matrix.
To obtain a standard column-orthogonal representation, we define $\X\triangleq
\begin{cases}
	\W, & d\geq p,\\
	\W^\top, & d<p,
\end{cases}
\
n\triangleq\max\{d,p\},
\
r\triangleq\min\{d,p\}.$ Then $\X\in\mathbb{R}^{n\times r}$ with $n\geq r$, and orthogonality can be imposed through $\X^\top\X=\mathbf{I}_r,$ so that $\X\in\mathrm{St}(n,r)$. Let $\{(\A_i,\y_i)\}_{i=1}^{N}$ denote the training dataset. When all other network parameters are held fixed, the optimization of this orthogonally constrained convolutional layer can be formulated as
\begin{equation}\label{eq:app:ocnn}
	\min_{\X\in\mathrm{St}(n,r)}
	\EEE_{\xi}\left[
	\ts\frac{1}{|\xi|} \ts\sum_{i\in\xi}
	\ell(\X;\A_i,\y_i)
	\right],\nn
\end{equation}
where $\ell$ denotes the per-sample loss, such as the cross-entropy loss,
and $\xi\subseteq\{1,\ldots,N\}$ is a uniformly sampled mini-batch.

\section{The Proposed Methods}\label{sec:prop}

In this section, we introduce {\sf RGR}, a recursive gradient retraction method for solving Problem~\eqref{eq:main}. Its central idea is to construct recursive stochastic gradient estimators in the Euclidean space and subsequently map them to tangent directions for retraction updates on the Stiefel manifold. Because historical and current gradient information is combined in the same Euclidean space, this construction avoids the additional cost of vector transport between successive tangent spaces. The {\sf RGR} method comprises three variants: {\sf RGR-SPIDER} for finite-sum optimization and {\sf RGR-SRM} and {\sf RGR-EMA} for stochastic optimization. {\sf RGR-SPIDER} is a Lipschitz-free and horizon-free method that combines the {\sf SPIDER} gradient estimator with an adaptive {\sf AdaGrad-Norm}-type stepsize rule, requiring neither prior knowledge of the gradient Lipschitz constant nor a prescribed iteration horizon. Built on {\sf SRM} and {\sf EMA} estimators, respectively, {\sf RGR-SRM} and {\sf RGR-EMA} are mini-batch and horizon-free methods, employing open-loop diminishing stepsize schedules.

\subsection{Geometric Preliminaries}

For any $\X\in\mathrm{St}(n,r)$, the tangent space of the Stiefel manifold
at $\X$ is given by $T_{\X}\mathrm{St}(n,r)
=
\left\{
\Deltas\in\mathbb{R}^{n\times r}
:
\X^\top\Deltas+\Deltas^\top\X=\zero_{r\times r}
\right\}.$ It provides a first-order linear approximation of
$\mathrm{St}(n,r)$ around $\X$. Although
$\Deltas\in T_{\X}\mathrm{St}(n,r)$ specifies a feasible first-order
direction, the Euclidean displacement $\X+\Deltas$ generally does not
belong to $\mathrm{St}(n,r)$. A retraction maps such a tangent displacement
onto the manifold while preserving its first-order behavior.  We next recall the definition of a retraction and present two auxiliary lemmas used in the subsequent analysis.

\begin{definition}[{\cite[Riemannian retraction]{Absil2008}}]
	The tangent bundle of the Stiefel manifold is $	T\mathrm{St}(n,r)
	\triangleq
	\left\{
	(\X,\Deltas):
	\X\in\mathrm{St}(n,r),\
	\Deltas\in T_{\X}\mathrm{St}(n,r)
	\right\}.$ A retraction is a smooth mapping $	\Retr:T\mathrm{St}(n,r)\rightarrow\mathrm{St}(n,r)$
	whose restriction at any $\X\in\mathrm{St}(n,r)$,
	$
	\Retr_{\X}:
	T_{\X}\mathrm{St}(n,r)
	\rightarrow
	\mathrm{St}(n,r),
	$
	satisfies
	$
	\Retr_{\X}(\zero_{n\times r})=\X,
	$  and $\lim_{ T_{\X}\mathrm{St}(n,r) \ni \Deltas \rightarrow \zero_{n\times r}} \tfrac{ \| \Retr_{\X}(\Deltas) - \X - \Deltas  \|_{\fro}}{\|\Deltas\|_{\fro}} = 0$.
\end{definition}

\begin{lemma} [{\cite[Properties of the retraction]{boumal2019global}}]\label{lemma:M12}
	Suppose that $\Retr$ is a smooth retraction defined on the tangent bundle of the Stiefel manifold. Then there exist constants $\dot{k},\ddot{k}>0$, independent of $\X$ and $\Deltas$, such that, for every $\X\in\mathrm{St}(n,r)$ and $\Deltas\in T_{\X}\mathrm{St}(n,r)$, we have the following results: \textbf{(}\bfit{i}\textbf{)} $\|\Retr_{\X}(\Deltas)-\X\|_{\fro}\leq \dot{k} \|\Deltas\|_{\fro}$. \textbf{(}\bfit{ii}\textbf{)} $\|\Retr_{\X}(\Deltas) - \X - \Deltas\|_{\fro} \leq \tfrac{1}{2}\ddot{k} \|\Deltas\|_{\fro}^2$.
\end{lemma}

\begin{lemma}[{\cite[Lemma~2.12]{yuan2024OADMM}}]\label{lemma:GD:bound}
	For any $\G\in\Rn^{n\times r}$, and $\X\in\mathrm{St}(n,r)$, we define $\GGG \triangleq \G - \X\G\trans\X$. It follows that: $2 \cdot \la \G,\GGG\ra \geq \|\GGG\|_{\fro}^2$.
\end{lemma}

\begin{remark}\rm{}
	Several standard retractions are available on the Stiefel manifold
	\cite{Absil2008,absil2012projection}. Two commonly used
	examples are the polar-decomposition retraction
	$ 
	\Retr_{\X}^{\mathrm{polar}}(\Deltas)
	=
	(\X+\Deltas)
	\left(
	\mathbf{I}_r+\Deltas^\top\Deltas
	\right)^{-1/2},
	$ 
	and the QR-decomposition retraction
	$ 
	\Retr_{\X}^{\mathrm{QR}}(\Deltas)
	=
	\qf(\X+\Deltas),
	$ 
	where $\qf(\X+\Deltas)$ denotes the $Q$ factor of the thin QR
	decomposition of $\X+\Deltas$. The diagonal entries of the corresponding
	$R$ factor are chosen to be positive to ensure a locally unique and
	smooth definition.
\end{remark}

\subsection{{\sf RGR-SPIDER} Method}

We first employ the {\sf SPIDER} estimator~\cite{fang2018spider} to construct a stochastic estimator $\G_t$ of the full gradient $\nabla f(\X_t)$ in the Euclidean space. At $t$-th iteration, a mini-batch $\xi_t\subseteq\{1,\ldots,N\}$ of fixed size $b=|\xi_t|$ is sampled uniformly at random. The corresponding mini-batch gradient evaluated at $\X$ is defined as
$
\nabla f(\X;\xi_t)
\triangleq
\frac{1}{b}\sum_{i\in\xi_t}\nabla f_i(\X).
$
The {\sf SPIDER}  estimator is then updated recursively according to
\begin{equation}\label{eq:spider}
	\G_t=
	\begin{cases}
		\nabla f(\X_t),
		& t\equiv0\pmod q,\\[1mm]
		\G_{t-1}+\tilde{\Deltas}_t,
		& \text{otherwise},
	\end{cases} \nn
\end{equation}
where $q\geq1$ is an update frequency parameter and
$
\tilde{\Deltas}_t
\triangleq
\nabla f(\X_t;\xi_t)
-
\nabla f(\X_{t-1};\xi_t)
$
is the same-sample gradient correction term. The correction term controls variance accumulation while retaining a low per-iteration cost. {\sf SPIDER}  computes the mini-batch gradients $\nabla f(\X_t; \xi_t)$ and $\nabla f(\X_{t-1}; \xi_t)$ in most cases while computing the exact full gradient $\nabla f(\X_t)$ every $q$ iterations to prevent error accumulation.

However, the Euclidean estimator $\G_t$ does not generally belong to
$T_{\X_t}\mathrm{St}(n,r)$. Under the canonical metric of the Stiefel manifold, we
therefore construct the tangent gradient estimator
\begin{equation}\label{eq:Qt}
	\Q_t
	\triangleq
	\G_t-\X_t\G_t^\top\X_t
	\in T_{\X_t}\mathrm{St}(n,r). \nn
\end{equation}

\noi Inspired by {\sf AdaGrad-Norm}~\cite{ward2020adagrad} and
{\sf ADA-SPIDER}~\cite{kavis2022adaptive}, {\sf RGR-SPIDER} employs an
adaptive stepsize based on the cumulative squared norms of the tangent
gradient estimators. Given an initial stepsize parameter $\etaup>0$, we set
\begin{equation}
	\eta_0=\etaup,
	\qquad
	\eta_{t+1}
	=
	\tfrac{1}{
		\sqrt{\eta_{t}^{-2}+\|\Q_{t}\|_{\fro}^{2}}
	}, 
	\quad t\geq1. \nn
\end{equation}
Let $\{\mathcal{H}_t\}_{t\geq0}$ denote the natural filtration
generated by the initial iterate and the sampled mini-batches.
The stepsize $\eta_t$ used in the $t$-th update is recursively
determined by $\eta_{t-1}$ and the lagged quantity
$\|\Q_{t-1}\|_{\fro}^2$. Therefore, $\eta_t$ is
$\mathcal{H}_{t-1}$-measurable. Thus, $\eta_t$ is determined before the
current mini-batch $\xi_t$ is sampled and is predictable with respect
to the current stochastic gradient. For $t \geq 1$, this update rule can be equivalently expressed in a cumulative form as:
$
\eta_t
=
\left(
\frac{1}{\etaup^2}
+
\sum_{i=0}^{t-1}\|\Q_i\|_{\fro}^{2}
\right)^{-1/2}.
$
The iterate is then updated through the retraction $\X_{t+1}=\Retr_{\X_t}(-\eta_t\Q_t),$ which ensures that $\X_{t+1}\in\mathrm{St}(n,r)$.

\begin{remark}\rm{}
\textbf{(}\bfit{i}\textbf{)} To the best of our knowledge, {\sf RGR-SPIDER} is the first method to integrate an {\sf AdaGrad-Norm}-type stepsize rule, and the Euclidean-space {\sf SPIDER} estimator for the finite-sum setting of Problem~\eqref{eq:main} on the Stiefel manifold. Existing Riemannian adaptive methods do not combine their adaptive stepsize rules with recursive variance reduction \cite{becigneul2018riemannian,bento2025riemannian,kasai2019riemannian}. {\sf S-SVRG}~\cite{jiang2017vector} construct a variance-reduced estimator in the ambient Euclidean space, but it uses an {\sf SVRG} estimator and a Barzilai--Borwein stepsize rather than an {\sf AdaGrad-Norm}-type rule.    \textbf{(}\bfit{ii}\textbf{)}  Although the convergence analysis assumes Lipschitz continuity of the gradient, the implementation of {\sf RGR-SPIDER} does not require the corresponding Lipschitz constant as an input. In this algorithmic sense, the method is Lipschitz-free. Moreover, it is also horizon-free because its stepsize rule does not depend on the total number of iterations.
\end{remark}

\subsection{{\sf RGR-SRM} Method} \label{sec:prop:srm}

For online stochastic optimization, we employ the stochastic recursive momentum ({\sf SRM}) estimator~\cite{cutkosky2019momentum,levy2021storm+} to approximate the full Euclidean gradient $\nabla f(\X_t)$.  Unlike standard momentum methods, {\sf SRM} constructs a recursive estimator by incorporating a gradient correction term to control the accumulation of estimation error. This technique allows the algorithm to achieve better iteration complexity, typically improving from $\mathcal{O}(\epsilon^{-4})$ to $\tilde{\mathcal{O}}(\epsilon^{-3})$. Let $b=|\xi_t|$ and $\tilde{\Deltas}_t \triangleq \nabla f(\X_t; \xi_t) - \nabla f(\X_{t-1}; \xi_t)$, where $\xi_t$ denotes the randomly sampled mini-batch at iteration $t$. The {\sf SRM} estimator is then recursively constructed  in the Euclidean space as:
\begin{align}\label{eq:storm}
	\M_{t}	~& = \alpha_t\nabla f(\X_{t};\xi_t) + (1-\alpha_t)\big(\M_{t-1} + \tilde{\Deltas}_t\big) \nn\\
	~&  = \nabla f(\X_{t};\xi_t) 	+ (1-\alpha_t)\big(\M_{t-1} - \nabla f(\X_{t-1};\xi_t )\big). \nn
\end{align}
Although it incorporates historical gradient information to reduce the variance of the estimator, this comes at the cost of additional computational overhead for updating the historical gradients.

Since $\M_t$ is constructed in the Euclidean space, it does not generally lie in the tangent space at $\X_t$. Under the canonical metric of the Stiefel manifold, we define the tangent gradient estimator  \[
\R_t
\triangleq
\M_t-\X_t\M_t^\top\X_t
\in T_{\X_t}\mathrm{St}(n,r).
\]

\noi We update $\alpha_t$ and the stepsize $\gamma_t$ solely based on the open-loop strategy. The open-loop schedules are developed as:
\[
\alpha_t=t^{-2/3},
\qquad
\gamma_t=\theta\sqrt{\alpha_t}=\theta t^{-1/3},
\qquad \theta>0,
\]
which depend only on the current iteration index. Finally, the next iterate $\X_{t+1}$ is obtained by a Riemannian retraction step: $\X_{t+1} = \Retr_{\X_t}\!\left(-\gamma_t \R_t\right)\in\mathrm{St}(n,r),$ which ensures that the updated iterate $\X_{t+1}$ remains on the Stiefel manifold.

\subsection{{\sf RGR-EMA} Method}

Alternatively, we employ the {\sf EMA} technique to approximate the Euclidean gradient $\nabla f(\X_t)$ evaluated at the current iterate $\X_t$. {\sf EMA} is a commonly used stabilization technique that smooths stochastic gradient noise by assigning exponentially decaying weights to historical estimators. Following the same strategy in Section \ref{sec:prop:srm}, {\sf RGR-EMA} replaces the {\sf SRM} estimator with an {\sf EMA} estimator:
\begin{align}\label{eq:ema}
	\N_{t} = \beta_t \nabla  f(\X_{t};\xi_t) + (1-\beta_t)\N_{t-1},\nn
\end{align}
where $\xi_t$ denotes the randomly sampled mini-batch at iteration $t$. Unlike {\sf SRM}, this estimator only uses the current stochastic gradient and the previous stochastic estimator without requiring the previous stochastic gradient. Its per-iteration gradient-estimation cost is therefore lower.

Analogously, we map $\N_t$ to the tangent space and update the iterate according to
\[
\S_t
\triangleq
\N_t-\X_t\N_t^\top\X_t
\in T_{\X_t}\mathrm{St}(n,r),
\qquad
\X_{t+1}
=
\Retr_{\X_t}(-\mu_t\S_t)\in \mathrm{St}(n,r),
\]
where the open-loop schedules are
\[
\beta_t=t^{-1/2},
\qquad
\mu_t=\theta\beta_t=\theta t^{-1/2},
\qquad \theta>0.
\]
Algorithm~\ref{alg:rgr-spider} summarizes {\sf RGR-SPIDER}, whereas Algorithm~\ref{alg:RGR-SRM-EMA} presents {\sf RGR-SRM} and {\sf RGR-EMA}.

\begin{remark}\normalfont
    \textbf{(}\bfit{i}\textbf{)}
    For stochastic nonconvex optimization on the Stiefel manifold, {\sf RGR-SRM} and {\sf RGR-EMA} construct their recursive gradient estimators entirely in the ambient Euclidean space instead of tangent space, avoiding explicit vector transport between  successive tangent spaces. Moreover, neither method requires the  total iteration horizon $T$ to be specified in advance, making both methods horizon-free. Although a related Euclidean-space  construction was used in the vector-transport-free {\sf S-SVRG} method~\cite{jiang2017vector}, their method is based on the {\sf SVRG} estimator for finite-sum optimization rather than a recursive estimator for the online setting problem. \textbf{(}\bfit{ii}\textbf{)} The {\sf SRM} estimator uses the correction term $\nabla f(\X_t;\xi_t)-\nabla f(\X_{t-1};\xi_t)$ and thus requires gradient evaluations at both $\X_t$ and $\X_{t-1}$. This technique improves the convergence guarantee at the expense of higher computational cost. By contrast, the {\sf EMA} estimator uses a weighted average of the current stochastic gradient $\nabla f(\X_t;\xi_t)$ and the previous estimator $\N_{t-1}$, which reduces the computational burden but leads to a weaker convergence rate.
\end{remark}

\begin{algorithm}[htbp]
	\caption{{\sf RGR-SPIDER}}
	\label{alg:rgr-spider}
	\begin{algorithmic}[1]
		
		\State \textbf{Input:}
		Initial point $\X_0\in\mathrm{St}(n,r)$,
		initial stepsize $\etaup>0$,
		refresh frequency $q\geq1$,
		mini-batch size $b$, and iteration budget $T$.
		
		\State \textbf{Initialization:} Set $\eta_0=\etaup$.
		
		\For{$t=0$ \textbf{to} $T-1$}
		
		\State (S1) Update the Euclidean-space {\sf SPIDER} estimator:
		\If{$t\equiv0\pmod q$}
		\Statex \hspace{3em}
		$\displaystyle \G_t=\nabla f(\X_t).$
		\Else
		\Statex \hspace{3em}
		Draw a mini-batch
		$\xi_t$ of size $b$ and set
		\Statex \hspace{3em}
		$
		\G_t
		=
		\G_{t-1}
		+\nabla f(\X_t;\xi_t)
		-\nabla f(\X_{t-1};\xi_t).
		$
		\EndIf
		
		\State (S2) Map $\G_t$ to the tangent space:
		\Statex \hspace{3em}
		$\displaystyle
		\Q_t
		=
		\G_t-\X_t\G_t^\top\X_t
		\in T_{\X_t}\mathrm{St}(n,r).
		$
		
		\State (S3) Update the iterate via retraction:
		\Statex \hspace{3em}
		$\displaystyle
		\X_{t+1}
		=
		\Retr_{\X_t}(-\eta_t\Q_t).
		$
		
		\State (S4) Update the adaptive stepsize:
		\Statex \hspace{3em}
		$\displaystyle
		\eta_{t+1}
		=
		\left(
		\eta_t^{-2}+\|\Q_t\|_{\fro}^2
		\right)^{-1/2}.
		$
		
		\EndFor
		
		\State Draw $\tau$ uniformly from $\{0,\ldots,T-1\}$.
		\State \textbf{Output:} $\X_\tau$.
		
	\end{algorithmic}
\end{algorithm}

\begin{algorithm}[htbp]
	\caption{{\sf RGR-SRM} and {\sf RGR-EMA}}
	\label{alg:RGR-SRM-EMA}
	\begin{algorithmic}[1]
		
		\State \textbf{Input:}
		Method flag
		$\mathsf{method}\in
		\{{\sf RGR-SRM},{\sf RGR-EMA}\}$,
		initial point $\X_0\in\mathrm{St}(n,r)$,
		stepsize parameter $\theta>0$,
		mini-batch size $b$, and iteration budget $T$.
		
		\State \textbf{Initialization:}
		Set $\X_1=\X_0$,
		$\M_0=\N_0=\zero_{n\times r}$, and
		$\gamma_0=\mu_0=\theta$.
		
		\For{$t=1$ \textbf{to} $T$}
		
		\State (S1) Draw a random mini-batch $\xi_t$ of size $b$.
		
		\If{$\mathsf{method}={\sf RGR-SRM}$}
		
		\State (S2) Set
		$\alpha_t=t^{-2/3}$ and
		$\gamma_t=\theta\sqrt{\alpha_t}
		=\theta t^{-1/3}$.
		
		\State (S3) Update the Euclidean-space {\sf SRM} estimator:
		\Statex \hspace{4.5em}
		$\displaystyle
		\M_t
		=
		\nabla f(\X_t;\xi_t)
		+
		(1-\alpha_t)
		\left(
		\M_{t-1}
		-\nabla f(\X_{t-1};\xi_t)
		\right).
		$
		
		\State (S4) Map $\M_t$ to a tangent direction:
		\Statex \hspace{4.5em}
		$\displaystyle
		\R_t
		=
		\M_t-\X_t\M_t^\top\X_t
		\in T_{\X_t}\mathrm{St}(n,r).
		$
		
		\State (S5) Update the iterate via retraction:
		\Statex \hspace{4.5em}
		$\displaystyle
		\X_{t+1}
		=
		\Retr_{\X_t}(-\gamma_t\R_t).
		$
		
		\Else
		
		\State (S2) Set
		$\beta_t=t^{-1/2}$ and
		$\mu_t=\theta\beta_t
		=\theta t^{-1/2}$.
		
		\State (S3) Update the Euclidean-space {\sf EMA} estimator:
		\Statex \hspace{4.5em}
		$\displaystyle
		\N_t
		=
		\beta_t\nabla f(\X_t;\xi_t)
		+
		(1-\beta_t)\N_{t-1}.
		$
		
		\State (S4) Map $\N_t$ to a tangent direction:
		\Statex \hspace{4.5em}
		$\displaystyle
		\S_t
		=
		\N_t-\X_t\N_t^\top\X_t
		\in T_{\X_t}\mathrm{St}(n,r).
		$
		
		\State (S5) Update the iterate via retraction:
		\Statex \hspace{4.5em}
		$\displaystyle
		\X_{t+1}
		=
		\Retr_{\X_t}(-\mu_t\S_t).
		$
		
		\EndIf
		
		\EndFor
		
		\State Draw $\tau$ uniformly from $\{1,\ldots,T\}$.
		\State \textbf{Output:} $\X_\tau$.
		
	\end{algorithmic}
\end{algorithm}

\section{Convergence Analysis for {\sf RGR} method}\label{sec:com}
This section presents the oracle complexity results for the proposed {\sf RGR-SPIDER}, {\sf RGR-SRM}, and {\sf RGR-EMA}.  The natural filtration is denoted by
$\{\mathcal{H}_t\}_{t\geq0}$, where
$
\mathcal{H}_t
\triangleq
\sigma(\X_0,\xi_0,\ldots,\xi_t).
$ We use $\mathbb{E}[\cdot]$ to denote the total expectation and define $ \EEE_t [\cdot]
\triangleq
\mathbb{E}[\cdot\mid\mathcal{H}_{t-1}]
$ as the conditional expectation given the history up to iteration $t-1$.

\subsection{Initial Theoretical Analysis}

To establish the convergence properties of the proposed algorithms, we begin by introducing the following standard assumptions regarding the objective function and the stochastic gradients.

\begin{assumption}\label{ass:fx}
	The stochastic function $f(\X;\xi)$ is sample-wise $L$-smooth and $C_f$-Lipschitz continuous with respect to $\X$. There exists two constants $L>0$ and $C_f>0$, such that $\|\nabla f(\X;\xi)-\nabla f(\X';\xi)\|_{\fro}
	\leq L\|\X-\X'\|_{\fro}$ and $\|\nabla f(\X;\xi)\|_{\fro}\leq C_f$ almost surely for all $\X,\X',\xi$.  Consequently, Jensen’s inequality gives $\|\nabla f(\X)\|_{\fro}\leq C_f$ for all $\X$.
\end{assumption}

\begin{assumption}\label{ass:sigma}
	The stochastic gradient $\nabla f(\x;\xi)$ is unbiased and has uniformly bounded variance. That is, for all $\X$, $\EEE_{\xi}[\nabla f(\X;\xi)] = \nabla f(\X)$, and $\EEE_{\xi}\!\left[\|\nabla f(\X;\xi)-\nabla f(\X)\|_{\fro}^{2}\right]\leq \sigma^2$.

\end{assumption}
\begin{remark}\rm{}
Assumptions~\ref{ass:fx} and~\ref{ass:sigma} impose standard sample-wise smoothness, unbiasedness, and bounded-variance conditions. Closely related conditions are widely used in analyses of nonconvex stochastic optimization and variance-reduced methods \cite{fang2018spider,li2023convergence,cutkosky2019momentum}. At iteration $t$, suppose that $\X_t$ is $\mathcal H_{t-1}$-measurable and that $\xi_t$ is sampled independently of $\mathcal H_{t-1}$ from the prescribed distribution.
Assumption~\ref{ass:sigma} then implies $\EEE_t[\nabla f(\X_t;\xi_t)]=\EEE_t[\nabla f(\X_t)]$
and $\EEE_t[\|\nabla f(\X_t;\xi_t)-\nabla f(\X_t)\|_{\fro}^2 ]\leq\sigma^2$. Applying the tower property yields
$\EEE[\nabla f(\X_t;\xi_t)]=\EEE[\nabla f(\X_t)]$ and $\EEE[\|\nabla f(\X_t;\xi_t)-\nabla f(\X_t)\|_{\fro}^2 ] \leq\sigma^2$.
\end{remark}

\noi  Under the canonical metric, the Riemannian gradient is $ \operatorname{grad}f(\X) \triangleq \nabla f(\X)-\X\nabla f(\X)^\top\X.$ We next define the critical point criterion used in the convergence analysis.

\begin{definition}[Critical point and $\varepsilon$-critical point]
	\label{def:c:point}
	A point $\dot{\X}\in\mathrm{St}(n,r)$ is a critical point of
	Problem~\eqref{eq:main} if
	$
	\operatorname{grad}f(\dot{\X})
	=
	\zero_{n\times r}.
	$
	A random output $\X_\tau\in\mathrm{St}(n,r)$ is said to be
	$\varepsilon$-critical in expectation if
	$
	\EEE\!\left[
	\|\operatorname{grad}f(\X_\tau)\|_{\fro}
	\right]
	\leq \varepsilon.
	$
\end{definition}

\subsection{Analysis for {\sf RGR-SPIDER}}\label{sec:com:srm}
This subsection establishes the iteration complexity and component-gradient  
complexity guarantees for {\sf RGR-SPIDER}. Since the Stiefel manifold $\mathrm{St}(n,r)$
is compact and $f(\X)$ is continuous, we let $\bar{\X}\in\arg\min_{\X\in\mathrm{St}(n,r)} f(\X)$ denote a global optimal solution. We introduce several key quantities to facilitate our analysis: $\XX_t \triangleq f(\X_t)- f(\bar{\X})$, $\AA_t \triangleq \eta_t\|\Q_t\|_{\fro}^2$, $\BB_t \triangleq \eta_t^2\|\Q_t\|_{\fro}^2$, and $\SS_t \triangleq \|\G_t - \nabla f(\X_t)\|_{\fro}^2$. Our convergence proof proceeds in three main stages: \textbf{(}\bfit{i}\textbf{)} We first establish a one-step bound for the sequence $\{\XX_t\}$. \textbf{(}\bfit{ii}\textbf{)} We derive cumulative bounds for $\AA_t$ and $\BB_t$ in terms of $\tfrac{1}{\eta_T}$ and $\tfrac{1}{\sqrt{\eta_T}}$, respectively. \textbf{(}\bfit{iii}\textbf{)} We prove that $\EEE[\tfrac{1}{\eta_T}]$ and $\EEE[\XX_T]$ are bounded above by two positive constants, respectively, both independent of $T$. \textbf{(}\bfit{iv}\textbf{)} Finally, we derive the iteration complexity and component-gradient  complexity of {\sf RGR-SPIDER}. To describe the periodic full-gradient updates, let $q\ge 1$ be
the update frequency. For every $t\geq0$, define the following indexing scheme: $r_{t}\triangleq \lfloor\tfrac{t}{q}\rfloor+1$, $ \underline{r_t} \triangleq (r_{t} - 1)q$, $ \overline{r_t} \triangleq r_{t} q  - 1$  Thus, $\underline{r}_t\leq t\leq\overline{r}_t$, and
$[\underline{r}_t,\overline{r}_t]$ is the epoch containing iteration $t$. For example, when $q=3$, the indices $t=0,1,2$ belong to the first epoch,
$t=3,4,5$ belong to the second epoch, and so forth. For subsequent use, define $
\hat{\eta} \triangleq \sqrt{1+4q\etaup^2(C_f^2+4(q-1)L^2r)},\ q'\triangleq q\hat{\eta}^{q-1}.$ We first derive the one-step bound for the sequence $\XX_t$.
\begin{lemma}[One-step bound for $\XX_t$, proof in \ref{app:prop:VR}]\label{prop:VR}
	Let $c_1 \triangleq \tfrac{(C_f\ddot{k} + \dot{k}^2L)}{2}$. Then, for every $t\geq0$, we have the following result almost surely: $$\XX_{t+1} - \XX_t  \leq c_1\BB_t - \tfrac{1}{4} \AA_t + \etaup\SS_t.$$
\end{lemma}

\noi Then, we establish upper bounds on the cumulative sums of \(\AA_t\) and \(\BB_t\).

\begin{lemma} [Cumulative bounds for $\AA_t$ and $\BB_t$, proof in \ref{app:lemma:AB}]\label{lemma:AB}
 For every integer $T\geq1$, the following bounds hold almost surely:
	
	\begin{enumerate}[label=\textbf{(\alph*)}, leftmargin=18pt, itemsep=2pt, topsep=1pt, parsep=0pt, partopsep=0pt]
		
		\item $\ts\sum_{t=0}^{T-1} \AA_t \leq  \frac{2\hat{\eta}}{\eta_T}$.
		
		\item $\sum_{t=0}^{T-1} \BB_t \leq \tfrac{4\hat{\eta}^2\sqrt{\bar\eta}}{\sqrt{\eta_T}}$.
	\end{enumerate}
\end{lemma}

\noi Subsequently, the remaining analysis focuses on establishing uniform upper bounds for $\EEE[\tfrac{1}{\eta_{T}}]$ and $\EEE[\XX_T]$. 


\begin{lemma} [Expected reciprocal-stepsize bound, proof in  \ref{app:prop:V}] \label{prop:V}
	We use the definitions of $\dot{k}$ and $\hat{\eta}$ as in Lemma \ref{lemma:M12} and Lemma \ref{lemma:AB}, respectively. Then we define $c_1' \triangleq 4\hat{\eta}^2\sqrt{\bar{\eta}}(\tfrac{\bar{\eta}\dot{k}^2L^2(q-1)}{b} + c_1)$, $w_1 \triangleq 2\EEE[\XX_0]+\frac1{\etaup}$ and $w_2 \triangleq 2c_1'$. We have the following results for all $0 \leq t \leq T-1$:  
	\begin{enumerate}[label=\textbf{(\alph*)}, leftmargin=18pt, itemsep=2pt, topsep=1pt, parsep=0pt, partopsep=0pt]
		
		\item $\EEE[\tfrac{1}{\eta_{T}}] \leq 2w_1 + w_2^2 \triangleq \tfrac{1}{\underline{\eta}}$.
		
		\item $\EEE[\XX_T] \leq \tfrac{w_1}{2} +  \tfrac{w_2}{2\sqrt{\underline{\eta}}} \triangleq \overline{\XX}$.
	\end{enumerate}
\end{lemma}

\noi Finally, we present the iteration complexity and component-gradient   complexity for {\sf RGR-SPIDER}.

\begin{theorem} [Iteration complexity and component-gradient   complexity, proof in \ref{app:the:it:VR}]\label{the:it:VR}
	Let the sequence $\{\X_t\}_{t=0}^{T-1}$ be generated by {\sf RGR-SPIDER}, and define the Riemannian gradient as $\operatorname{grad}f(\X_t) \triangleq \nabla f(\X_t) - \X_t\nabla f(\X_t)\trans\X_t$.
	
	\begin{enumerate}[label=\textbf{(\alph*)}, leftmargin=18pt, itemsep=2pt, topsep=1pt, parsep=0pt, partopsep=0pt]
		\item Let $\tau$ be sampled uniformly from $\{0,\ldots,T-1\}$, independently of all other algorithmic randomness.  We established that $\EEE\!\left[
		\|\operatorname{grad}f(\X_\tau)\|_{\fro}
		\right]
		= \tfrac{1}{T} \ts \sum_{t=0}^{T-1}  \EEE[\|\operatorname{grad}f(\X_t)\|_{\fro} ] = \OO(T^{-1/2})$. Therefore, $\X_\tau$ is an $\varepsilon$-critical point in expectation
		after $T=\OO(\varepsilon^{-2})$ iterations.
		
		\item  Set both mini-batch size $b$ and update frequency parameter $q$ as $\lceil \sqrt N\rceil$. Then, {\sf RGR-SPIDER} achieves a component-gradient   complexity of $\mathcal{O}(N + \sqrt{N}\varepsilon^{-2})$ to find an $\varepsilon$-critical point $\X_\tau$ satisfying $\EEE[\|\operatorname{grad}f(\X_\tau)\|_{\fro}] \leq \varepsilon$.
	\end{enumerate}
	
\end{theorem}

\begin{remark}\rm{}
\textbf{(}\bfit{i}\textbf{)} Lemma \ref{prop:V} shows that the adaptive stepsize $\EEE[\tfrac{1}{\eta_{t}}]$ is bounded below by a strictly positive constant $\tfrac{1}{\underline{\eta}}$. This  property implies that $\eta_t$ does not decay indefinitely and, in particular, prevents the stepsize from vanishing. Such a property is crucial for establishing the iteration complexity result in Theorem \ref{the:it:VR}. \textbf{(}\bfit{ii}\textbf{)} Theorem \ref{the:it:VR} provides the  iteration complexity result for Lipschitz-free methods within the adaptive Riemannian retraction applied to finite-sum optimization problems. Part \textbf{(b)} of Theorem \ref{the:it:VR} further refines this result by specifying suitable choices of $b$ and $q$. In particular, by setting $b = q = \lceil \sqrt N\rceil $, {\sf RGR-SPIDER} attains a component-gradient   complexity of $\mathcal{O}(N+\sqrt{N}\varepsilon^{-2})$, matching the optimal complexity order of {\sf SPIDER}-type methods for nonconvex finite-sum optimization.
\end{remark}

\subsection{Analysis for {\sf RGR-SRM}}\label{sec:it:srm}
This subsection establishes the convergence and iteration-complexity guarantees for {\sf RGR-SRM}. Let $\X_t$ denote the iterate generated at iteration $t$. For our analysis, we define  $\YY_t = \EEE[f(\X_t)- f(\bar{\X})] $, and introduce two auxiliary quantities: $\CC_t \triangleq \EEE[ \|\M_t- \nabla f(\X_t)\|^2_{\fro}]$ and $\DD_t \triangleq \EEE[\|\M_t - \X_t \M_t \trans \X_t\|^2_{\fro}]$. Here, $\CC_t$ measures the estimation error of the recursive stochastic estimator $\M_t$, whereas $\DD_t$ measures the squared norm of its tangent component. Our convergence analysis unfolds in three steps: \textbf{(\bfit{i})} Using the smoothness of the objective, we establish a one-step bound for the sequence $\{\YY_t\}$.  \textbf{(}\bfit{ii}\textbf{)}  Then we derive a recursive bound of $\CC_t$ associated with $\CC_{t-1}$ and $\DD_{t-1}$. \textbf{(}\bfit{iii}\textbf{)}  Finally, combining above results, we establish the iteration complexity of {\sf RGR-SRM}. The following two lemmas provide the one-step bound for $\YY_t$ and the recursive bound for $\CC_t$ for {\sf RGR-SRM}.

\begin{lemma}[One-step bound for $\YY_t$, proof in Appendix~\ref{app:prop:C}]\label{prop:C}   
	We choose $\theta$ sufficiently small so that $c_2 \triangleq  \tfrac{C_f\ddot{k}\theta + \dot{k}^2L\theta}{2} - \tfrac{1}{4} <0$. The sequence $\YY_t$ satisfies the following inequality for $t\geq1$: $$	\YY_{t+1} - \YY_{t}  \leq c_2 \gamma_t \DD_t + \gamma_t  \CC_t.$$
\end{lemma}

\begin{lemma}[Recursive bound  of $\CC_t$, proof in \ref{app:prop:D}]\label{prop:D}
 	 We first define $\LLL_t \triangleq 6L^2\gamma_{t-1}^2\dot{k}^2 \DD_{t-1} + 3 \alpha_t^2\sigma^2$. Then the sequence $\CC_t$ satisfies the following recursive bound for all $t\geq 1$: $$\CC_t \leq (1-\alpha_t)^2\CC_{t-1} + \LLL_t.$$
\end{lemma}

\noi Combining Lemmas~\ref{prop:C} and~\ref{prop:D}, we establish the following iteration complexity of {\sf RGR-SRM}.
\begin{theorem}[Iteration complexity, proof in \ref{app:the:E}]\label{the:E}
	Suppose Assumptions \ref{ass:fx} and \ref{ass:sigma} hold. Let the sequence $\{\X_t\}$ be generated by {\sf RGR-SRM}, and define the full Riemannian gradient at $\X_t$ as $\operatorname{grad}f(\X_{t}) \triangleq \nabla f(\X_t) - \X_t\nabla f(\X_t)\trans\X_t$. Let $\delta>0$ and define $c_3 \triangleq -c_2 - 24 \cdot 2^{1/3}\theta^2 L^2\dot{k}^2(1+\delta)$. If there exist a constant $\delta$ and a sufficiently small  $\theta$ such that $c_3 > 0$, then {\sf RGR-SRM} achieves an iteration complexity of $\tilde{\OO}(\varepsilon^{-3})$ to find an $\varepsilon$-critical point. Let $\tau$ be sampled uniformly from $\{1,\ldots,T\}$, independently
	of all other algorithmic randomness. Then
	$$\ts\EEE\!\left[
	\|\operatorname{grad}f(\X_\tau)\|_{\fro}
	\right]
	= \tfrac{1}{T}\sum_{t=1}^{T} \EEE[\|\operatorname{grad}f(\X_{t})\|_{\fro}]  = \tilde{\OO}(T^{-1/3}),$$
\end{theorem}
where $\tilde{\OO}(\cdot)$ hides polylogarithmic factors.

\subsection{Analysis for {\sf RGR-EMA}}
This subsection provides the convergence analysis of {\sf RGR-EMA}. Let $\X_t$ denote the solution generated by {\sf RGR-EMA} at the $t$-th iteration. Following Section \ref{sec:it:srm}, we define $\ZZ_t = \EEE[f(\X_t)- f(\bar{\X})]$, $\EE_t \triangleq \EEE[ \|\N_t- \nabla f(\X_t)\|^2_{\fro}]$ and $\FF_t \triangleq \EEE[\|\N_t - \X_t \N_t \trans \X_t\|^2_{\fro}]$. Here, \(\EE_t\) measures the estimation error of the EMA estimator \(\N_t\), whereas \(\FF_t\) measures the squared norm of its tangent component. The analysis follows the same three-stage structure as that of {\sf RGR-SRM} in Section~\ref{sec:it:srm}. We first establish a one-step bound for \(\{\ZZ_t\}\) and a recursive bound for \(\EE_t\), and then combine these results to derive the iteration complexity of {\sf RGR-EMA}.

\begin{lemma}[One-step bound for $\ZZ_t$, proof in \ref{app:prop:F}]\label{prop:F}
	We find a $\theta$ sufficiently small so that $c_2 \triangleq  \tfrac{C_f\ddot{k}\theta + \dot{k}^2L\theta}{2} - \tfrac{1}{4}<0$. The sequence $\ZZ_t$ satisfies: $$\ZZ_{t+1} - \ZZ_{t}  \leq c_2 \mu_t \FF_t + \mu_t  \EE_t.$$
\end{lemma}

\begin{lemma}[Recursive bound  of $\EE_t$, proof in \ref{app:prop:G}]\label{prop:G}
	Let $\KKK_t \triangleq L^2\mu_{t-1}^2\dot{k}^2\tfrac{\FF_{t-1}}{\beta_t} + \beta_t^2\sigma^2$. We have the following results for all $t\geq 1$: $$\EE_t \leq 	(1 - \beta_t)\EE_{t-1} + \KKK_t.$$
\end{lemma}

\noi Finally, we provide the iteration complexity of {\sf RGR-EMA}.

\begin{theorem}[Iteration complexity, proof in \ref{app:the:H}]\label{the:H}
	Letting the sequence $\{\X_t\}$ be generated by {\sf RGR-EMA} and $\delta > 0$, we define $c_4 \triangleq - c_2 - 4 \theta^2 L^2\dot{k}^2(1+\delta)$. If there exist a constant $\delta$ and a sufficiently small $\theta$ such that $c_4>0$, then {\sf RGR-EMA} attains an iteration complexity of $\tilde{\OO}(\varepsilon^{-4})$ to find an $\varepsilon$-critical point. Let $\tau$ be sampled uniformly from $\{1,\ldots,T\}$, independently
	of the algorithmic randomness. Then
	$$\ts\EEE[\|\operatorname{grad}f(\X_\tau)\|_{\fro}]	= \ts\tfrac{1}{T}\sum_{t=1}^{T} \EEE\left[\|\operatorname{grad}f(\X_{t})\|_{\fro}\right]  = \tilde{\OO}(T^{-1/4}).$$
\end{theorem}
\begin{remark}\rm{}
	
We make the following remarks. \textbf{(}\bfit{i}\textbf{)} Theorems~\ref{the:E} and~\ref{the:H} establish iteration complexity guarantees for two mini-batch and horizon-free  methods, {\sf RGR-SRM} and {\sf RGR-EMA}, in the online setting on Stiefel manifolds. Their convergence guarantees do not require prior knowledge of the total iteration horizon $T$. \textbf{(}\bfit{ii}\textbf{)} {\sf RGR-SRM} and {\sf RGR-EMA} are mini-batch stochastic Riemannian methods for Problem~\eqref{eq:main}. Theorems~\ref{the:E} and~\ref{the:H} show that their iteration complexity guarantees remain valid with a constant mini-batch size $b=\OO(1)$; hence, $b$ need not scale with either the target tolerance $\varepsilon$ or the iteration horizon $T$. This contrasts with existing analyses that require increasing batch sizes. In~\cite{liuserena}, {\sf R-Hybrid-SGD} and {\sf R-SRM} use $b=\OO(T^{1/3})$, whereas {\sf R-SRVRG}, {\sf R-SVRRM}, {\sf R-SVRG}, and {\sf R-PAGE} use $b=\OO(\varepsilon^{-2})$. {\sf R-SPIDER}~\cite{zhou2019faster} uses a batch size of $\OO(\varepsilon^{-2})$ at periodic refresh iterations and $\OO(\varepsilon^{-1})$ otherwise. In~\cite{han2021improved}, {\sf R-AbaSRG} and {\sf R-SRG} require $b=\OO(\varepsilon^{-1})$, whereas {\sf R-AbaSVRG} and {\sf R-SVRG} require $b=\OO(\varepsilon^{-4/3})$. Moreover, {\sf R-PAGE}~\cite{li2021page} uses a batch size of $\OO(\varepsilon^{-2})$ during its probabilistic refresh steps.
\end{remark}

\section{Experiments}\label{sec:exp}

This section evaluates {\sf RGR-SPIDER}, {\sf RGR-SRM}, and {\sf RGR-EMA} on three representative applications: PCA, LRMC, and OCNNs. To isolate the contribution of recursive gradient estimation, we include an ablation variant, termed {\sf RGR-vanilla}, that replaces the recursive estimator with a standard stochastic gradient estimator. We further compare recursive estimators constructed in the ambient Euclidean space with their counterparts constructed directly in Riemannian tangent spaces. This comparison examines whether Euclidean-space recursive estimation offers practical advantages while avoiding explicit vector transport. For the PCA and LRMC experiments, we use real-world datasets obtained from the LIBSVM repository.\footnote{\url{https://www.csie.ntu.edu.tw/~cjlin/libsvm/}} We adopt the notation \textsf{name}-$N$-$n$ for a dataset subset containing $N$ instances and $n$ features randomly sampled from the original dataset. Specifically, we consider TDT2-$9394$-$500$, 20News-$10000$-$500$, Sector-$6412$-$500$, Mushroom-$8124$-$112$, CIFAR-$10000$-$500$, Gisette-$6000$-$500$, MNIST-$5000$-$500$, and TEXT1-$1946$-$500$. For OCNNs experiments, we use two datasets: CIFAR-10 and MNIST.

\subsection{{\sf RGR-SPIDER} on Finite-Sum PCA}\label{sec:pcafinite}

Let $\A\in\mathbb{R}^{N\times n}$ be a centered data matrix, with
$\A_i\in\mathbb{R}^{1\times n}$ denoting its $i$-th row. PCA seeks an
orthonormal basis $\X\in\mathrm{St}(n,r)$ for an $r$-dimensional subspace
that minimizes the average reconstruction error:
\begin{equation}\label{eq:pca}
	\min_{\X\in\mathrm{St}(n,r)}f(\X) \triangleq
	\ts  \frac{1}{N}\sum_{i=1}^{N}f_i(\X), \qquad f_i(\X) \triangleq \frac{1}{2} \left\|\A_i-\A_i\X\X^\top \right\|_{2}^{2},
\end{equation}
which has the finite-sum structure specified in Problem~\eqref{eq:main}.

$\bullet$ \textbf{Compared Methods}. We compare {\sf RGR-SPIDER} with five representative baseline methods and one ablation variant: \textbf{(\bfit{i})} {\sf R-SGD}~\cite{bonnabel2013stochastic}; \textbf{(\bfit{ii})} {\sf R-SVRG}~\cite{han2021improved}; \textbf{(\bfit{iii})} {\sf R-SPIDER}~\cite{zhou2019faster}; \textbf{(\bfit{iv})} {\sf R-SRVRG}~\cite{liuserena}; \textbf{(\bfit{v})} {\sf R-SVRRM}~\cite{liuserena}; and \textbf{(\bfit{vi})} {\sf RGR-vanilla}.

$\bullet$ \textbf{Experimental Settings}. We set $r=\operatorname{round}(N/2)$ and $b=\operatorname{round}(\sqrt{N})$. For {\sf RGR-SPIDER}, {\sf RGR-vanilla}, and {\sf R-SPIDER}, we set $q=\operatorname{round}(\sqrt{N})$, whereas for {\sf R-SVRG}, {\sf R-SRVRG}, and {\sf R-SVRRM}, we set the number of inner-loop iterations to $100$. We use the QR-decomposition retraction $\Retr_{\X}^{\mathrm{QR}}(\Deltas) = \qf(\X+\Deltas)$  for all methods except {\sf R-SGD}, for which we use the exponential map. All baseline methods except {\sf R-SPIDER} use a fixed stepsize of \(\tau=10^{4}\), with \(\tau=10^{3}\) used for the Gisette dataset. For {\sf R-SPIDER}, we employ the adaptive stepsize $\tau_t = \max(\tau,\ \dist(\X_{t-1},\X_t))$, where $\dist(\X_{t-1},\X_t)$ denotes the distance between two consecutive iterates. For {\sf R-SRVRG} and {\sf R-SVRRM}, the parameter $\beta$ introduced in~\cite{liuserena} is set to $0.5$. For {\sf RGR-SPIDER} and {\sf RGR-vanilla}, we set the initial stepsize bound to $\etaup=2\tau$. All methods are implemented in MATLAB and executed on a computer equipped with a 3.4-GHz Intel CPU and 32 GB of RAM.

$\bullet$ \textbf{Experimental Results}. Figure~\ref{fig:pcafinite} compares {\sf RGR-SPIDER} with the baseline methods. Two main observations can be made. First, {\sf RGR-SPIDER} converges faster than the compared methods that construct recursive gradient estimators directly in Riemannian tangent spaces. Second, {\sf RGR-SPIDER} consistently outperforms its non-recursive ablation, {\sf RGR-vanilla}, demonstrating the practical benefit of incorporating the {\sf SPIDER}.

\subsection{{\sf RGR-SPIDER} on Finite-Sum LRMC}\label{sec:lrmcfinite}

We next consider LRMC with partially observed data. Let
$\A\in\mathbb{R}^{N\times n}$ be a partially observed matrix and let
$\Omega_i$ denote the set of observed column indices in its $i$-th row.
The operator $\mathcal{P}_{\Omega_i}(\cdot)$ retains the entries indexed
by $\Omega_i$ and sets all remaining entries to zero. For a fixed auxiliary
factor $\Y\in\mathbb{R}^{N\times r}$, the LRMC subproblem with respect to
the Stiefel-constrained factor $\X$ is formulated as
\begin{equation}\label{eq:lrmc}
	\min_{\X\in\mathrm{St}(n,r)}
	f(\X)
	\triangleq
	\ts \frac{1}{N}\sum_{i=1}^{N}f_i(\X),
	\qquad
	f_i(\X)
	\triangleq
	\frac{1}{2}
	\left\|
	\mathcal{P}_{\Omega_i}
	\left(\A_i-\Y_i\X^\top\right)
	\right\|_{2}^{2},
\end{equation}
where $\A_i\in\mathbb{R}^{1\times n}$ and
$\Y_i\in\mathbb{R}^{1\times r}$ denote the $i$-th rows of $\A$ and $\Y$,
respectively. For any fixed $\Y$, Problem~\eqref{eq:lrmc} has the
finite-sum structure specified in Problem~\eqref{eq:main}.

$\bullet$ \textbf{Experimental Settings}. For each LRMC instance, we generate an observation mask by sampling entries independently according to a prescribed observation rate $\rho_{\mathrm{obs}}<0.2$. The resulting mask is held fixed throughout each run. We employ an alternating update scheme: the compared Riemannian methods update the Stiefel-constrained factor $\X$, whereas standard Euclidean gradient descent with stepsize $1$ updates the unconstrained factor $\Y$. We otherwise use the same baseline methods and hyperparameter settings as in Section~\ref{sec:pcafinite}.

$\bullet$ \textbf{Experimental Results}. Figure~\ref{fig:lrmcfinite} presents the results for finite-sum LRMC. Consistent with the PCA experiments, {\sf RGR-SPIDER} converges faster than the tangent-space recursive baselines and outperforms {\sf RGR-vanilla}, supporting the effectiveness of Euclidean-space recursive gradient estimation across different Stiefel-constrained problems.

\subsection{{\sf RGR-SRM} and {\sf RGR-EMA} on Stochastic PCA}
\label{sec:pcaonline}

We consider a stochastic mini-batch formulation of PCA. Let
$\xi\subseteq\{1,\ldots,N\}$ denote a mini-batch of data indices sampled
uniformly from the empirical distribution. The stochastic PCA problem is
formulated as
\begin{equation}\label{eq:pca:stoc}
	\min_{\X\in\mathrm{St}(n,r)}
	f(\X)
	\triangleq
	\EEE_{\xi}\!\left[f(\X;\xi)\right],
	\qquad
	f(\X;\xi)
	\triangleq
	\ts\frac{1}{2|\xi|}
	\ts\sum_{i\in\xi}
	\left\|
	\A_i-\A_i\X\X^\top
	\right\|_{2}^{2}.
\end{equation}
Under uniform mini-batch sampling, $f(\X;\xi)$ is an unbiased stochastic
estimator of the finite-sum PCA objective. Therefore,
Problem~\eqref{eq:pca:stoc} has the stochastic optimization structure
specified in Problem~\eqref{eq:main}.

$\bullet$ \textbf{Experimental Settings}. For all baseline methods, we set the initial stepsize to $\tau_1=10^{4}$, except on Gisette, for which we use $\tau_1=10^{3}$. For $t\geq2$, the baseline stepsize is updated according to $\tau_t=\tau_{t-1}-10^{-4}\tau_1.$ For {\sf R-SPIDER}, we employ the adaptive stepsize $\tau_t = \max(\tau_{t-1}-10^{-4}\tau_1,\ \dist(\X_{t-1},\X_t))$. The parameter $\beta$ used by the methods in~\cite{liuserena} is set to $0.5$. For {\sf RGR-SRM}, {\sf RGR-EMA}, and {\sf RGR-vanilla}, we set the stepsize parameter to $\theta=10\tau_1$. Unless otherwise specified, the datasets, initialization procedure, and computational environment are the same as those described in Section~\ref{sec:pcafinite}.

$\bullet$ \textbf{Experimental Results}. Figure~\ref{fig:pcaonline} presents the results for stochastic PCA. Two main observations can be made. First, {\sf RGR-SRM} and {\sf RGR-EMA} converge faster than the compared methods that construct recursive gradient estimators directly in Riemannian tangent spaces. Second, both methods outperform {\sf RGR-vanilla} on most datasets, demonstrating the practical benefits of incorporating the {\sf SRM} and {\sf EMA} estimators into the proposed retraction method.

\subsection{{\sf RGR-SRM} and {\sf RGR-EMA} on Stochastic LRMC}
\label{sec:lrmconline}

We next consider a stochastic mini-batch formulation of LRMC. Let
$\xi\subseteq\{1,\ldots,N\}$ denote a mini-batch sampled
uniformly from the observed data. For a fixed auxiliary factor
$\Y\in\mathbb{R}^{N\times r}$, the stochastic LRMC subproblem with respect
to $\X$ is formulated as
\begin{equation}\label{eq:lrmc:stoc}
	\min_{\X\in\mathrm{St}(n,r)}
	f(\X)
	\triangleq
	\EEE_{\xi}\!\left[f(\X;\xi)\right],
	\qquad
	f(\X;\xi)
	\triangleq
	\ts\frac{1}{2|\xi|}
	\ts\sum_{i\in\xi}
	\left\|
	\mathcal{P}_{\Omega_i}
	\left(\A_i-\Y_i\X^\top\right)
	\right\|_{2}^{2}.
\end{equation}
For any fixed $\Y$, $f(\X;\xi)$ is an unbiased stochastic estimator of
the finite-sum LRMC objective with respect to $\X$. Hence,
Problem~\eqref{eq:lrmc:stoc} has the stochastic optimization structure
specified in Problem~\eqref{eq:main}. In the experiments, the complete
LRMC model is optimized using an alternating scheme in which the
Stiefel-constrained factor $\X$ and the unconstrained factor $\Y$ are
updated successively.

$\bullet$ \textbf{Experimental Settings}. The observation masks are generated by sampling entries independently according to prescribed observation rates below $0.2$ and are held fixed throughout each run. The Stiefel variable $\X$ is updated using {\sf RGR-SRM}, {\sf RGR-EMA}, or one of the compared Riemannian methods, whereas $\Y$ is updated by Euclidean gradient descent with stepsize $1$. We otherwise use the same compared methods and hyperparameter settings as in Section~\ref{sec:pcaonline}.

$\bullet$ \textbf{Experimental Results}. Figure~\ref{fig:lrmconline} reports the results for stochastic LRMC. Consistent with the PCA experiments, {\sf RGR-SRM} and {\sf RGR-EMA} converge faster than the tangent-space recursive baselines and generally outperform {\sf RGR-vanilla}. These results indicate that Euclidean-space recursive gradient estimation remains effective when combined with alternating updates for the LRMC model.

\subsection{{\sf RGR-SRM} and {\sf RGR-EMA} on Stochastic OCNNs}
\label{sec:cnnonline}

We evaluate {\sf RGR-SRM} and {\sf RGR-EMA} on the training of OCNNs. For a mini-batch $\xi$ sampled uniformly from the training set, the stochastic loss is defined as
\[
f(\X;\xi)
=
\ts\frac{1}{|\xi|}\sum_{i\in\xi}
\ell(\X;\A_i,\y_i),
\
f(\X)=\EEE_{\xi}[f(\X;\xi)],
\]
where $\ell$ denotes the cross-entropy loss and $\X$ collectively represents the constrained network weights. Because orthogonality is imposed on multiple layers, the feasible set is a product of Stiefel manifolds. In implementation, the tangent projection and retraction are applied separately to each constrained weight matrix. The resulting problem follows the stochastic optimization structure considered in Problem~\eqref{eq:main}.

$\bullet$ \textbf{Network Architecture}. The network consists of three convolutional blocks followed by two fully connected layers. For CIFAR-10 and MNIST, the numbers of input channels are set to $3$ and $1$, respectively. The three convolutional layers use $3\times3$ kernels and produce $32$, $64$, and $128$ output channels, respectively, with each convolution followed by batch normalization and a ReLU activation. The first two convolutional blocks additionally use max pooling, while adaptive average pooling reduces the final feature maps to a spatial resolution of $4\times4$. The resulting features are flattened into a $2048$-dimensional vector and passed through the first fully connected layer, which maps the feature dimension from $2048$ to $256$ and is followed by a ReLU activation. The second fully connected layer maps the resulting $256$-dimensional representation to $10$ output classes. Orthogonality constraints are imposed on all convolutional and fully connected weight matrices. These constrained weights are updated using the tested Retraction methods, whereas the bias and batch-normalization parameters remain unconstrained and are updated using {\sf Adam}.

$\bullet$ \textbf{Experimental Settings}. We train the networks on CIFAR-10 and MNIST using the cross-entropy loss and compare the same family of baseline methods as in Section~\ref{sec:pcafinite}. For the baseline methods, the initial stepsize is set to $\tau_1=0.05$ and is decreased additively by $10^{-4}$ after each epoch: $
\tau_t=\tau_1-(t-1)\times 10^{-4},
\ t=1,\ldots,200.
$
For {\sf RGR-SRM}, {\sf RGR-EMA}, and {\sf RGR-vanilla}, we set the stepsize parameter to $\theta=10\tau_1$. During the first epoch, each mini-batch contains one quarter of the training set, providing low-variance gradient estimates during the initial stage of optimization. From the second epoch onward, the mini-batch size is fixed at $128$. All methods are trained for $200$ epochs. The experiments are implemented in Python and conducted on an NVIDIA RTX 3090 Ti GPU with 24 GB of memory.

$\bullet$ \textbf{Experimental Results}. Figure~\ref{fig:ocnn} reports the training and test performance of the compared methods. Although {\sf RGR-SRM} and {\sf RGR-EMA} do not always attain the highest training accuracy, the proposed methods generally decrease the training loss faster and attain lower final loss values than the tangent-space recursive baselines.  On the test sets, the proposed methods achieve favorable loss and accuracy values.

\subsection{Comparison of Euclidean-Space and Riemannian Tangent-Space Recursive Estimation}\label{sec:Euclidean-tangent}

We further compare recursive gradient estimators constructed in the Euclidean space with those constructed directly in Riemannian tangent spaces. The comparison is conducted on both finite-sum and stochastic PCA problems to examine how the estimator construction space affects practical performance. The proposed {\sf RGR} variants construct the {\sf SPIDER}, {\sf SRM}, and {\sf EMA} estimators in the Euclidean space $\mathbb{R}^{n\times r}$ and subsequently map them to tangent directions for retraction updates. By contrast, Riemannian tangent-space methods, including {\sf R-SPIDER}~\cite{zhou2019faster} and {\sf RSRM}~\cite{ijcai2021p345}, update their recursive estimators directly in the tangent spaces associated with successive iterates and therefore require vector transport between successive tangent spaces. For comparison with {\sf RGR-EMA}, we additionally implement a tangent-space Riemannian exponential moving average method, denoted by {\sf REMA}, as detailed in Algorithm~\ref{alg:REMA} in Appendix \ref{appendixema}.

$\bullet$ \textbf{Experimental Settings}. We use the same datasets and experimental protocols as in the preceding finite-sum and stochastic PCA experiments. The parameter $\alpha$ for {\sf RSRM} and {\sf REMA} is set to $0.9$, following~\cite{ijcai2021p345}. To isolate the effect of the estimator construction space, {\sf R-SPIDER}, {\sf RSRM}, and {\sf REMA} use the same respective stepsize schedules as {\sf RGR-SPIDER}, {\sf RGR-SRM}, and {\sf RGR-EMA}.

$\bullet$ \textbf{Experimental Results}. Figures~\ref{fig:comparespider}, \ref{fig:comparesrm}, and \ref{fig:compareema} compare the Euclidean-space and tangent-space implementations. The Euclidean-space variants generally converge faster,  and attain lower final objective values in several cases. These results demonstrate the practical effectiveness of constructing recursive estimators in the Euclidean space before mapping them to tangent directions.

\section{Conclusion}\label{sec:conclusion}

In this paper, we proposed {\sf RGR}, a recursive gradient retraction method for stochastic nonconvex optimization on Stiefel manifolds. {\sf RGR} is the method on Stiefel manifolds that performs recursive stochastic gradient estimation in the Euclidean space and then maps the estimators to tangent directions, eliminating explicit vector transport. We developed three variants, {\sf RGR-SPIDER}, {\sf RGR-SRM}, and {\sf RGR-EMA}, and established their iteration complexity guarantees. In particular, {\sf RGR-SPIDER} is Lipschitz-free and horizon-free, while {\sf RGR-SRM} and {\sf RGR-EMA} are mini-batch and horizon-free methods whose convergence guarantees do not require prescribed large or specific batch sizes and whose parameters can be chosen without prior knowledge of the total
iteration budget. We establish that {\sf RGR-SPIDER} finds an $\varepsilon$-critical point in expectation within $\mathcal{O}(\varepsilon^{-2})$ iterations, requiring $\mathcal{O}(N+\sqrt{N}\varepsilon^{-2})$ component-gradient evaluations.  In the online setting, {\sf RGR-SRM} and {\sf RGR-EMA} attain iteration complexities of $\tilde{\mathcal{O}}(\varepsilon^{-3})$ and $\tilde{\mathcal{O}}(\varepsilon^{-4})$, respectively. Experiments on PCA, LRMC and OCNNs problems demonstrate the practical efficiency and scalability of the proposed methods, with faster empirical convergence than competing stochastic Riemannian methods.

\begin{figure}[htbp]
	\centering
	\begin{subfigure}{0.6\textwidth}
		\includegraphics[width=\textwidth]{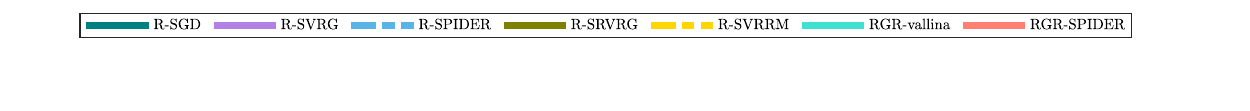}
	\end{subfigure}
	
	\begin{subfigure}{0.24\textwidth}
		\includegraphics[width=\textwidth]{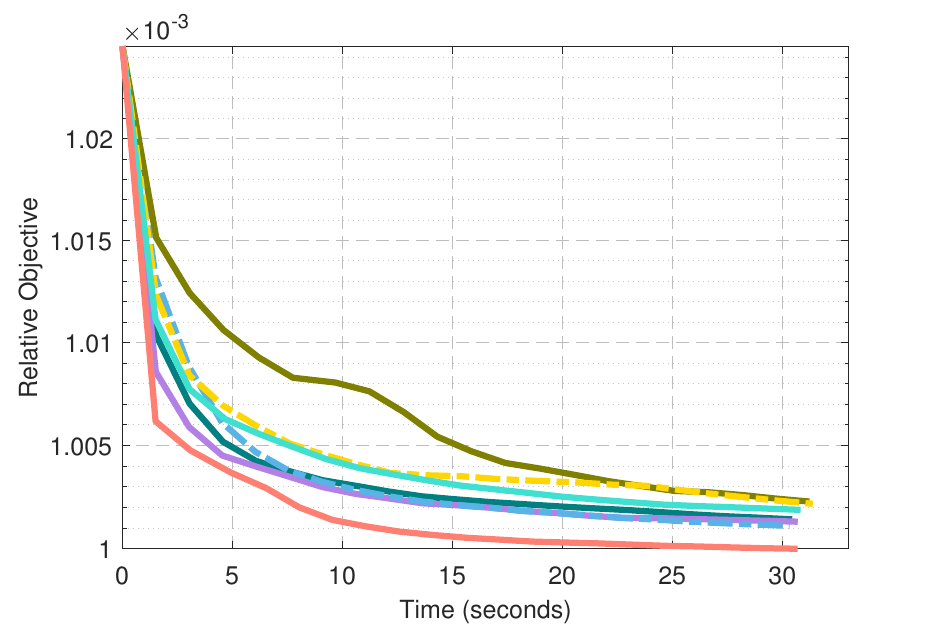}
		\caption{TDT2.}
	\end{subfigure}
	\begin{subfigure}{0.24\textwidth}
		\includegraphics[width=\textwidth]{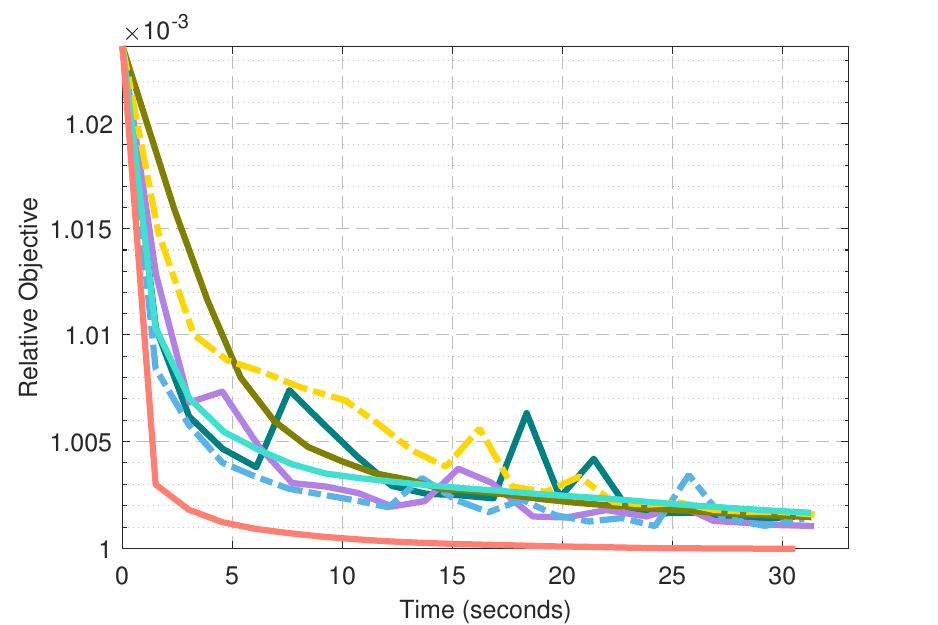}
		\caption{20News.}
	\end{subfigure}
	\begin{subfigure}{0.24\textwidth}
		\includegraphics[width=\textwidth]{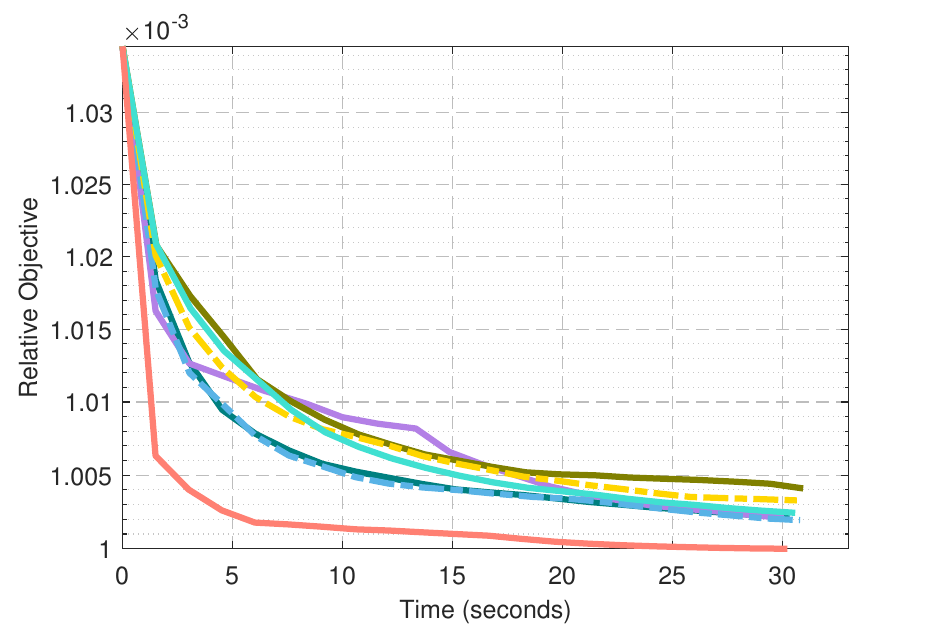}
		\caption{Sector.}
	\end{subfigure}
	\begin{subfigure}{0.24\textwidth}
		\includegraphics[width=\textwidth]{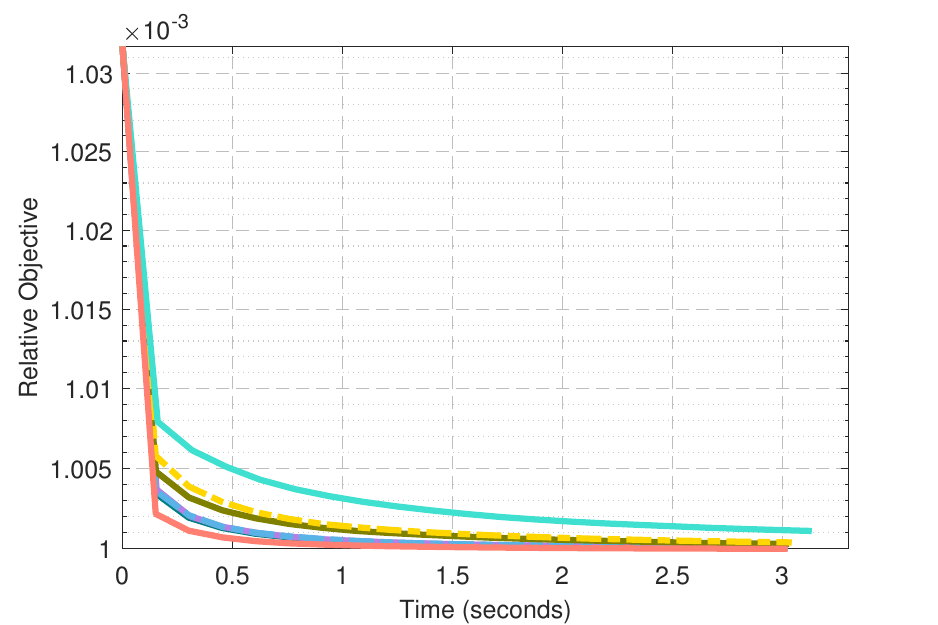}
		\caption{mushroom.}
	\end{subfigure}
	
	\begin{subfigure}{0.24\textwidth}
		\includegraphics[width=\textwidth]{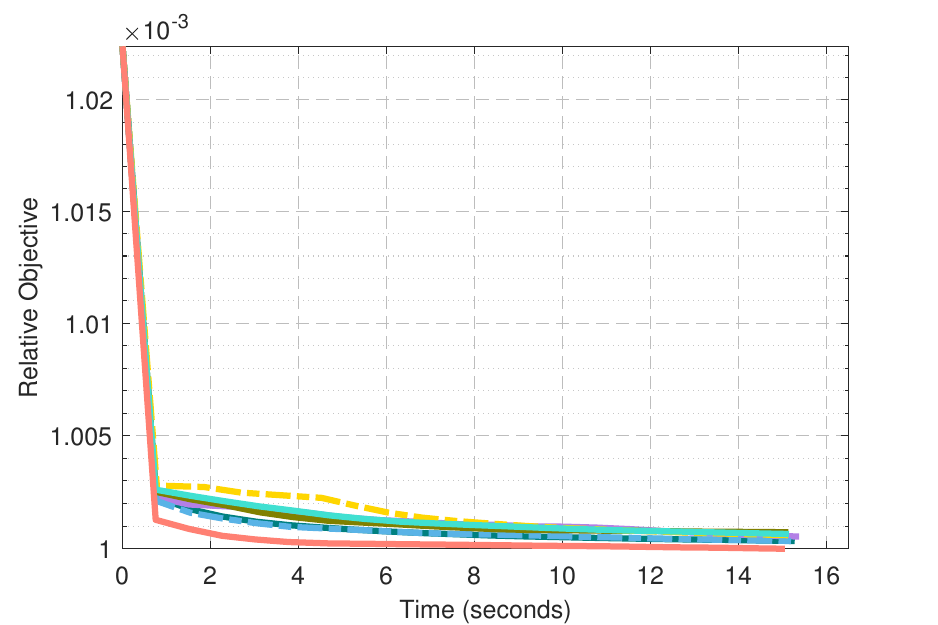}
		\caption{Cifar.}
	\end{subfigure}
	\begin{subfigure}{0.24\textwidth}
		\includegraphics[width=\textwidth]{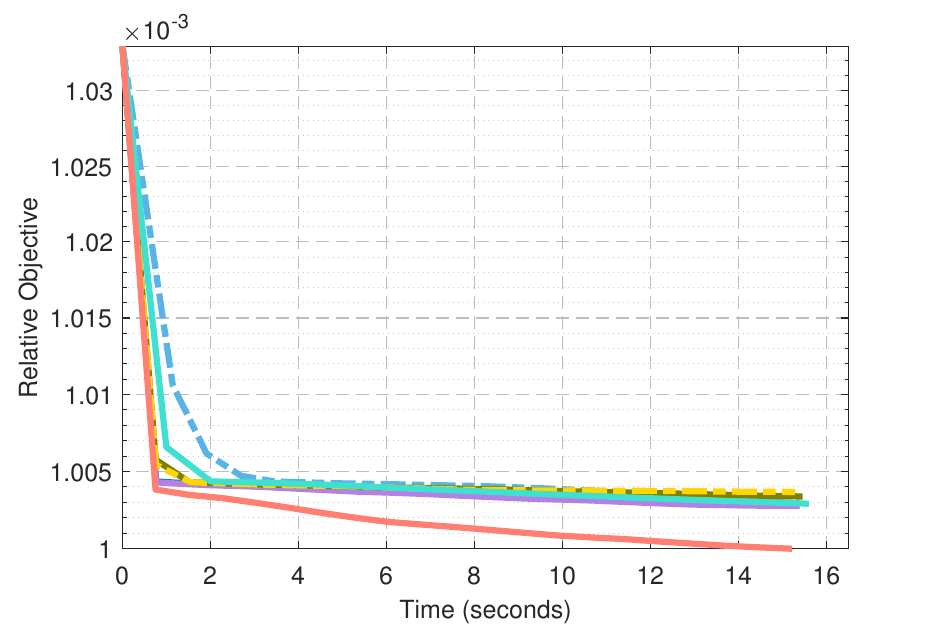}
		\caption{Gisette.}
	\end{subfigure}
	\begin{subfigure}{0.24\textwidth}
		\includegraphics[width=\textwidth]{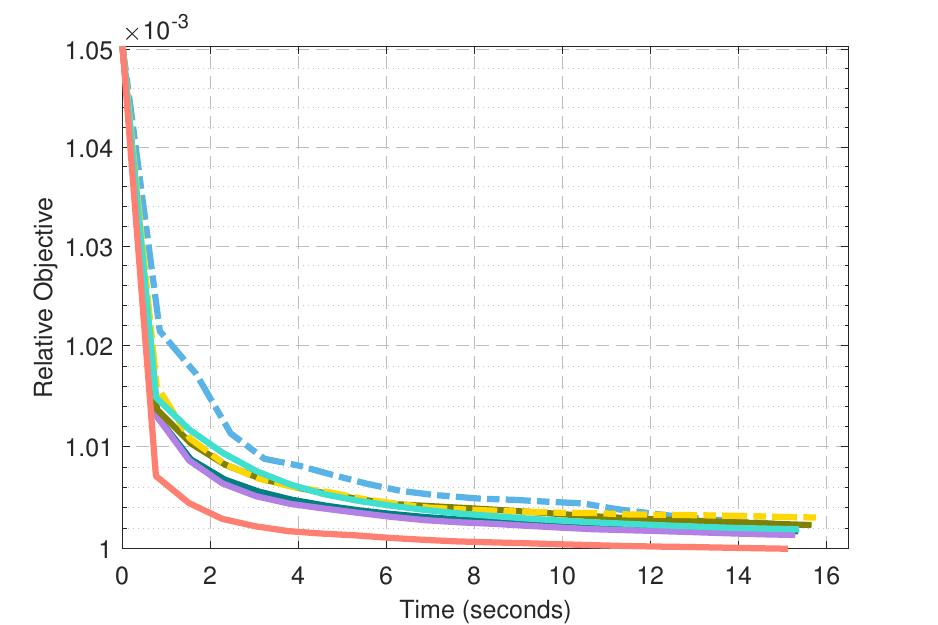}
		\caption{MNIST.}
	\end{subfigure}
	\begin{subfigure}{0.24\textwidth}
		\includegraphics[width=\textwidth]{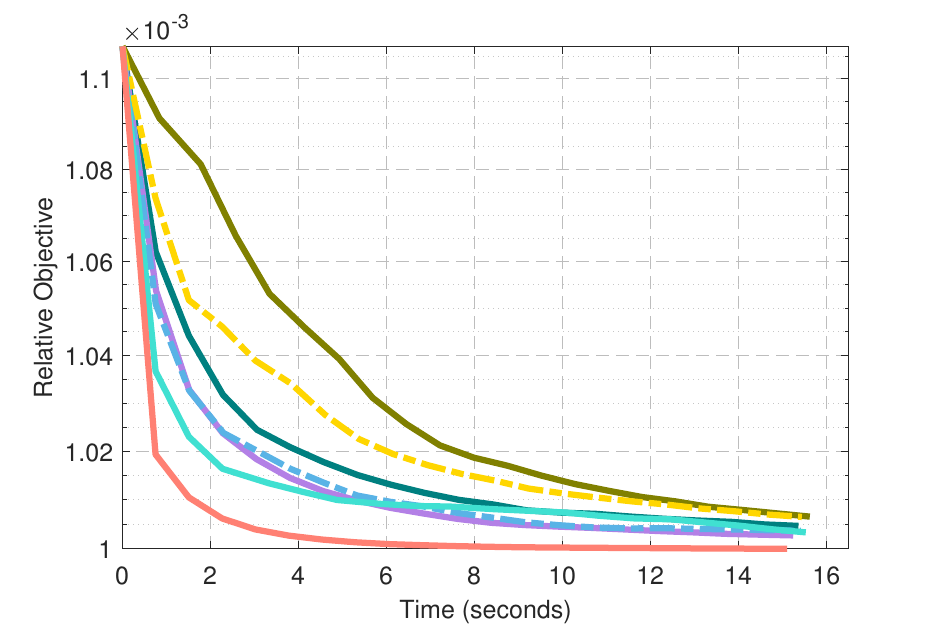}
		\caption{TEXT1.}
	\end{subfigure}
	\caption{Convergence curves of the compared methods for the PCA problem in the finite-sum setting.}
	\label{fig:pcafinite}
\end{figure}

\begin{figure}[htbp]
	\centering
	\begin{subfigure}{0.6\textwidth}
		\includegraphics[width=\textwidth]{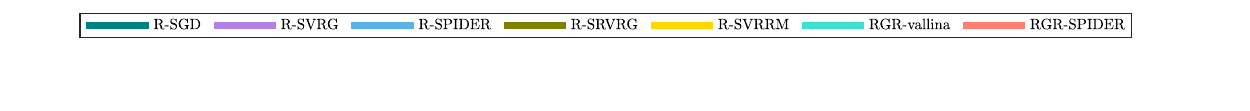}
	\end{subfigure}
	
	\begin{subfigure}{0.24\textwidth}
		\includegraphics[width=\textwidth]{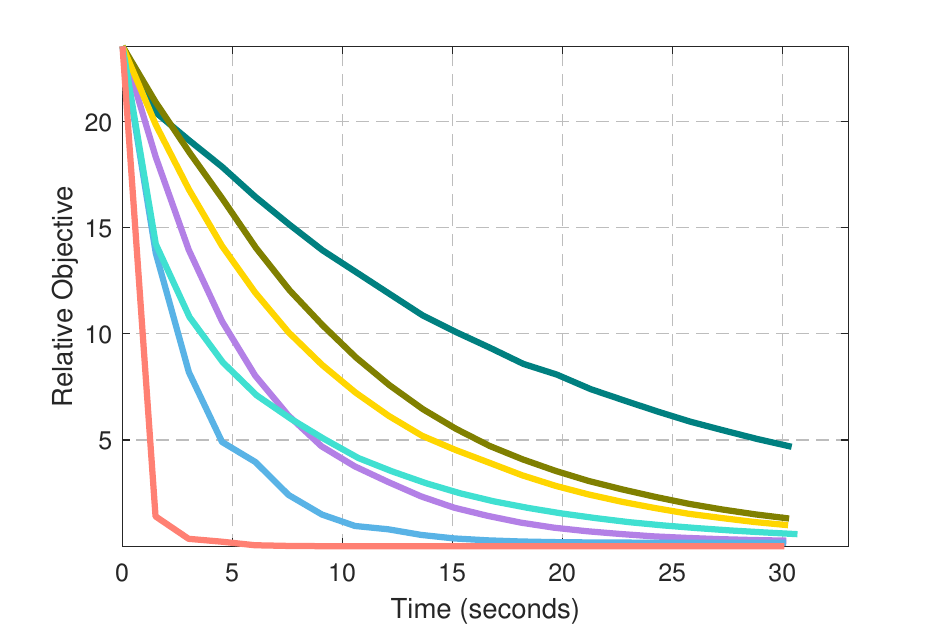}
		\caption{TDT2.}
	\end{subfigure}
	\begin{subfigure}{0.24\textwidth}
		\includegraphics[width=\textwidth]{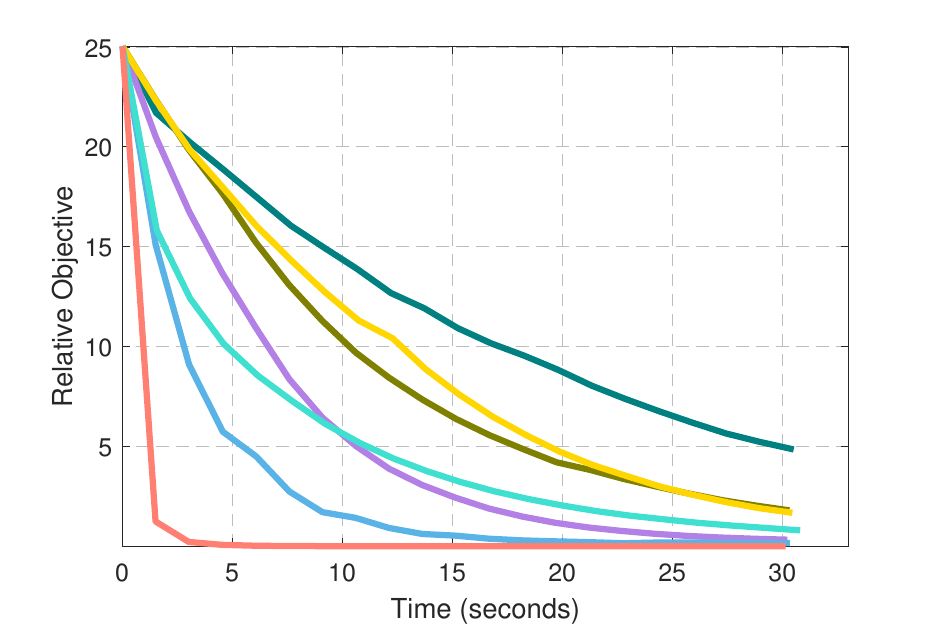}
		\caption{20News.}
	\end{subfigure}
	\begin{subfigure}{0.24\textwidth}
		\includegraphics[width=\textwidth]{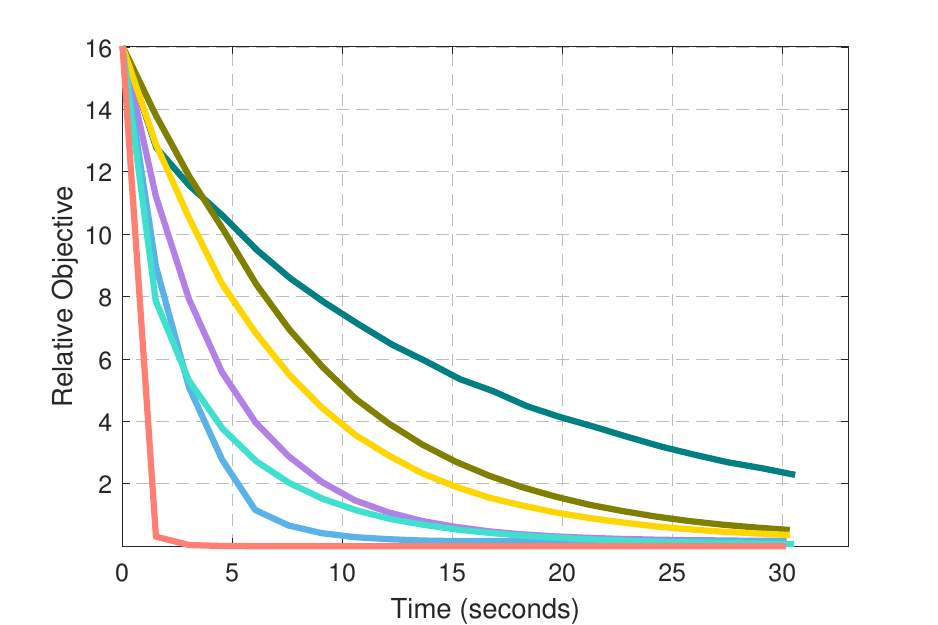}
		\caption{Sector.}
	\end{subfigure}
	\begin{subfigure}{0.24\textwidth}
		\includegraphics[width=\textwidth]{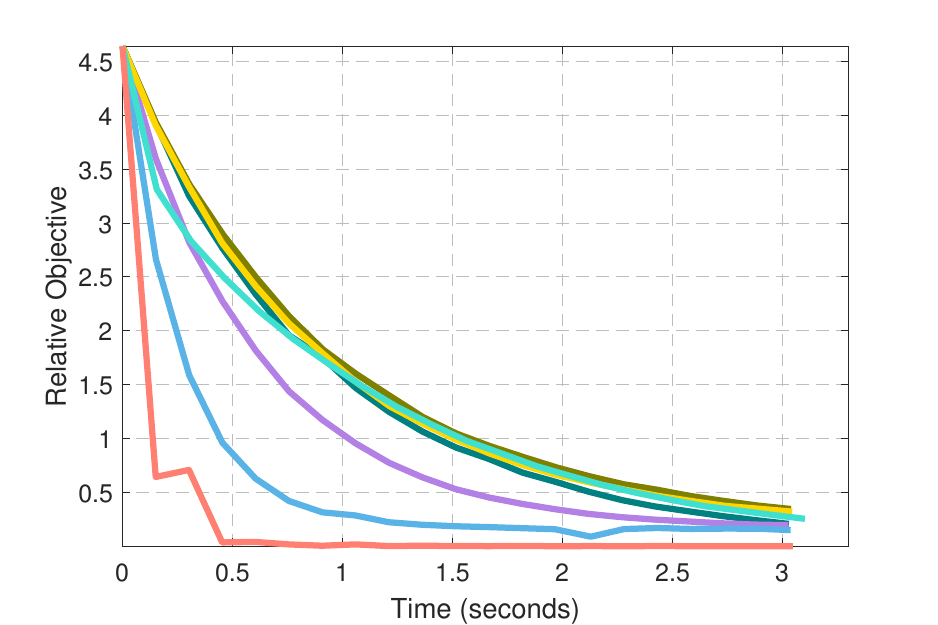}
		\caption{mushroom.}
	\end{subfigure}
	
	\begin{subfigure}{0.24\textwidth}
		\includegraphics[width=\textwidth]{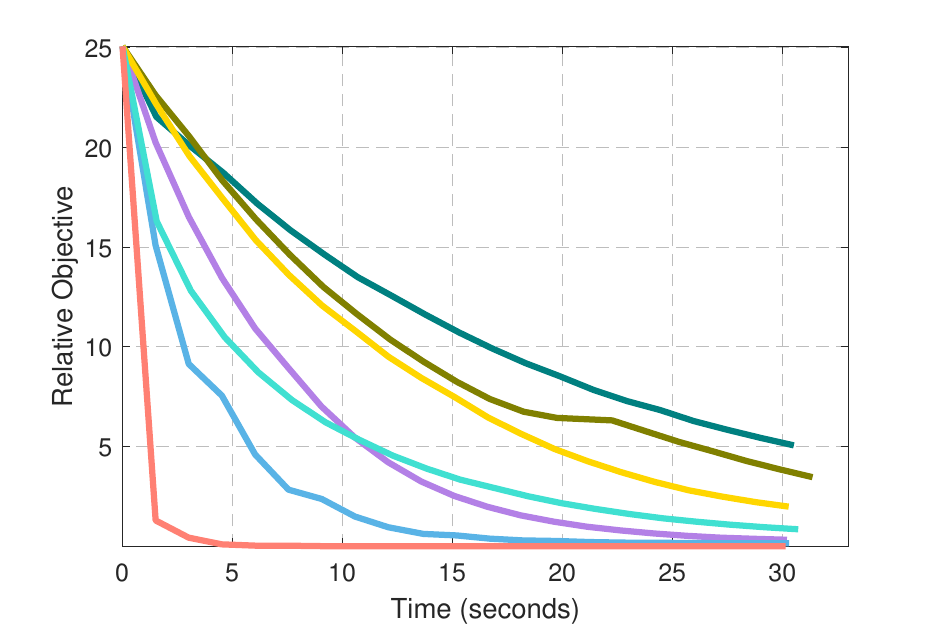}
		\caption{Cifar.}
	\end{subfigure}
	\begin{subfigure}{0.24\textwidth}
		\includegraphics[width=\textwidth]{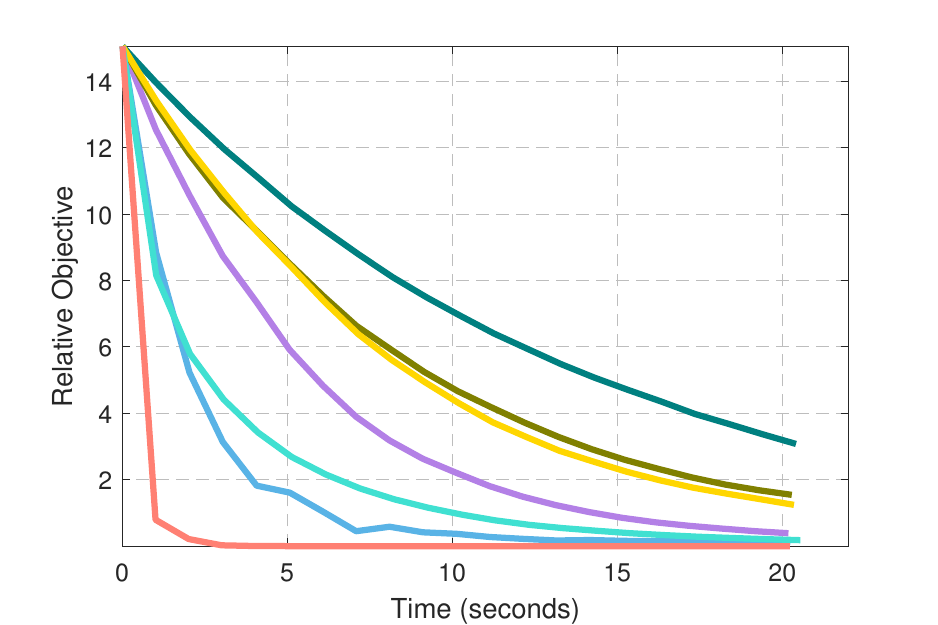}
		\caption{Gisette.}
	\end{subfigure}
	\begin{subfigure}{0.24\textwidth}
		\includegraphics[width=\textwidth]{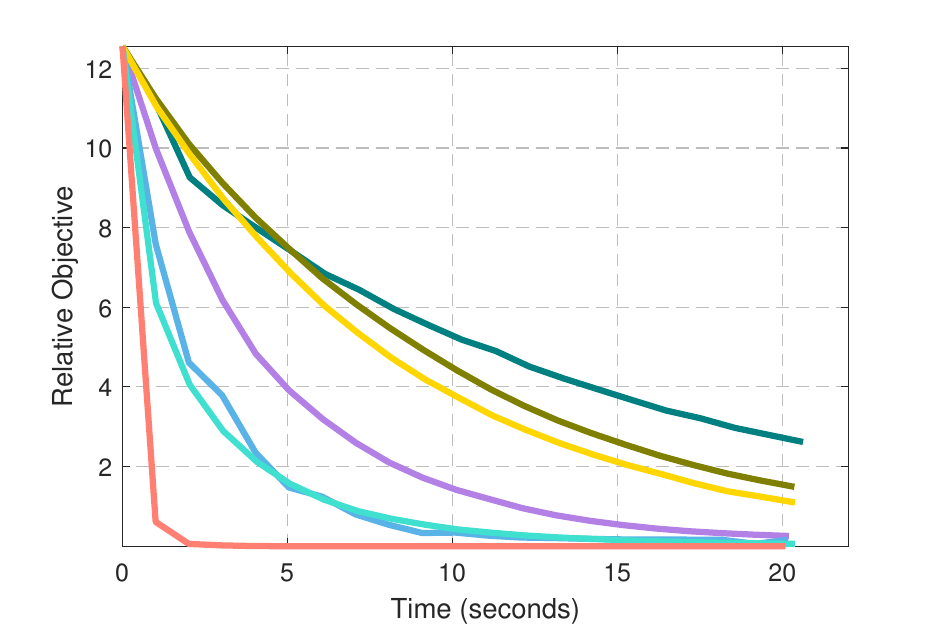}
		\caption{MNIST.}
	\end{subfigure}
	\begin{subfigure}{0.24\textwidth}
		\includegraphics[width=\textwidth]{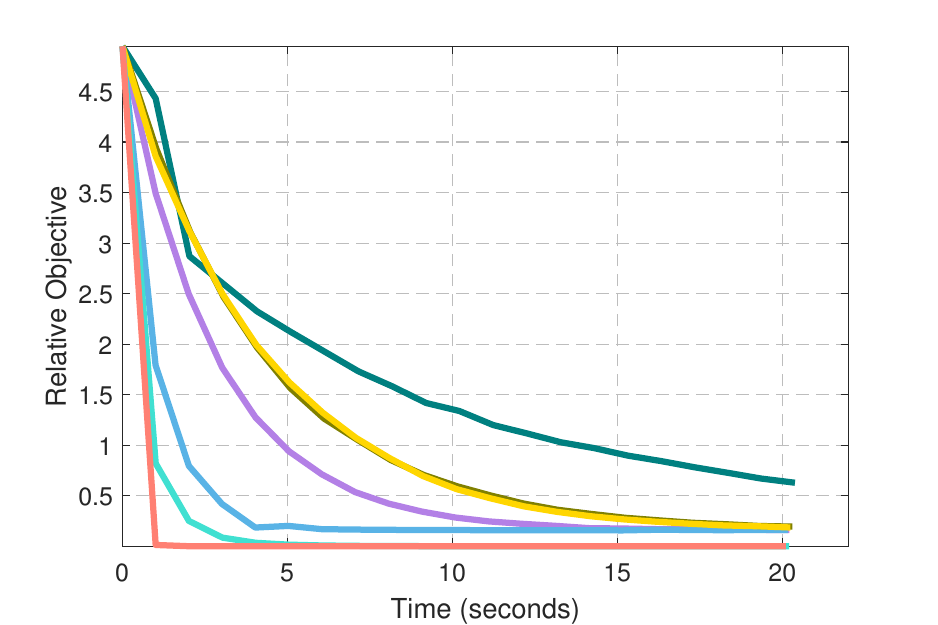}
		\caption{TEXT1.}
	\end{subfigure}
	\caption{Convergence curves of the compared methods for the LRMC problem in the finite-sum setting.}
	\label{fig:lrmcfinite}
\end{figure}

\begin{figure}[htbp]
	\centering
	\begin{subfigure}{0.7\textwidth}
		\includegraphics[width=\textwidth]{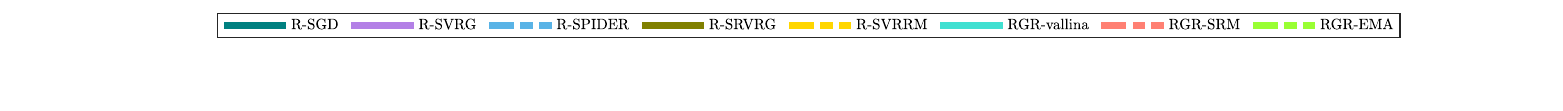}
	\end{subfigure}
	
	\begin{subfigure}{0.24\textwidth}
		\includegraphics[width=\textwidth]{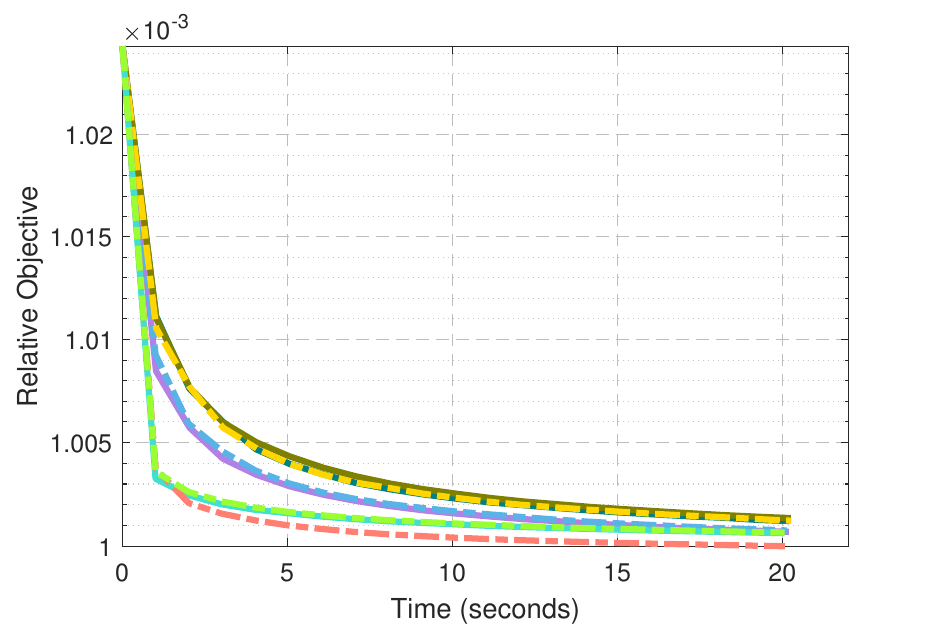}
		\caption{TDT2.}
	\end{subfigure}
	\begin{subfigure}{0.24\textwidth}
		\includegraphics[width=\textwidth]{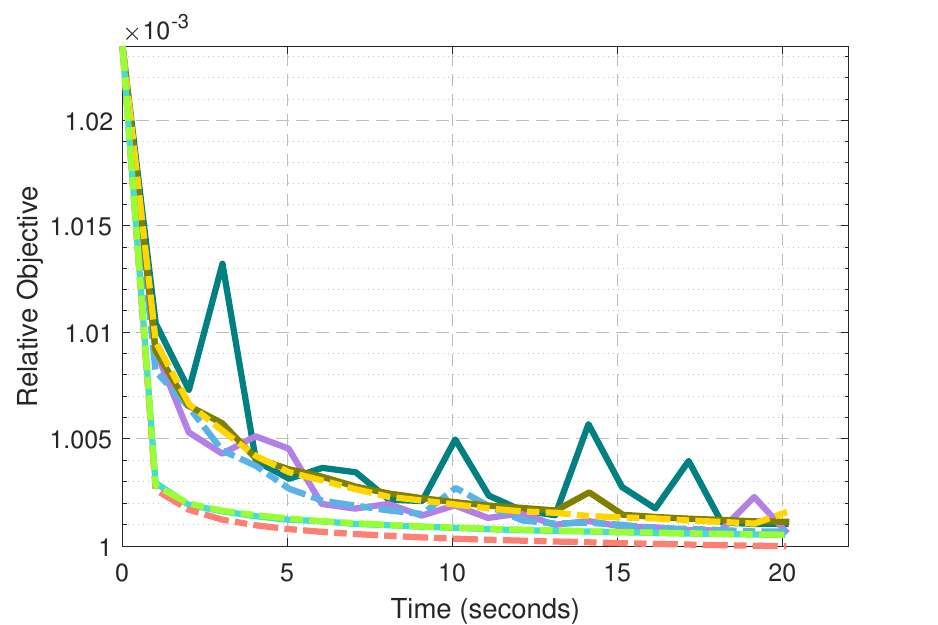}
		\caption{20News.}
	\end{subfigure}
	\begin{subfigure}{0.24\textwidth}
		\includegraphics[width=\textwidth]{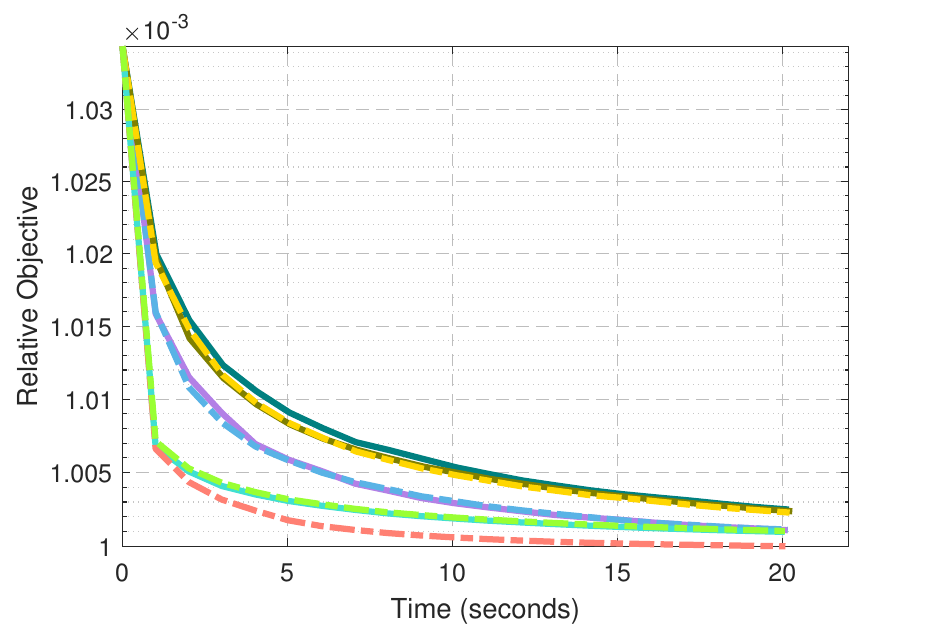}
		\caption{Sector.}
	\end{subfigure}
	\begin{subfigure}{0.24\textwidth}
		\includegraphics[width=\textwidth]{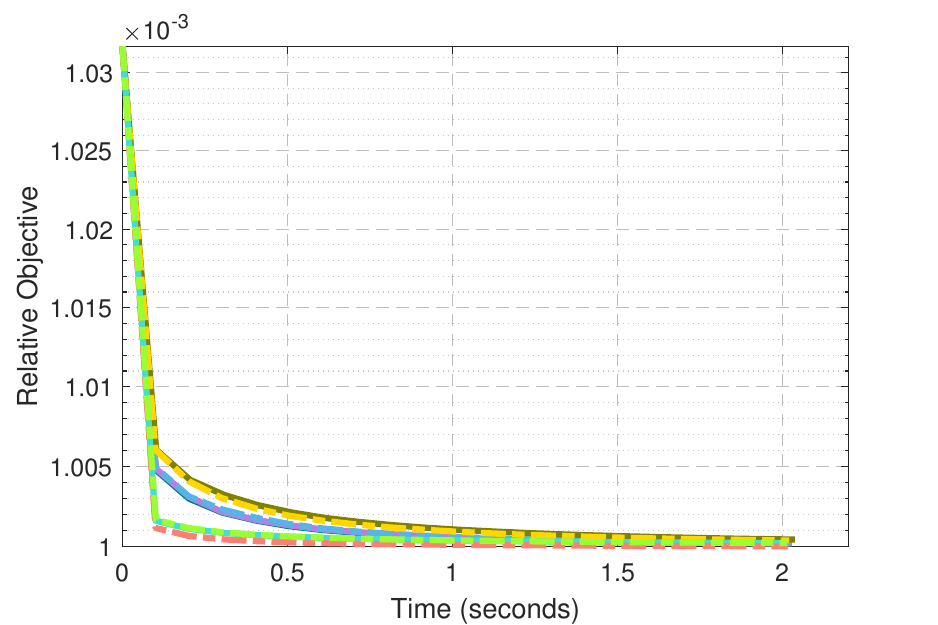}
		\caption{mushroom.}
	\end{subfigure}
	
	\begin{subfigure}{0.24\textwidth}
		\includegraphics[width=\textwidth]{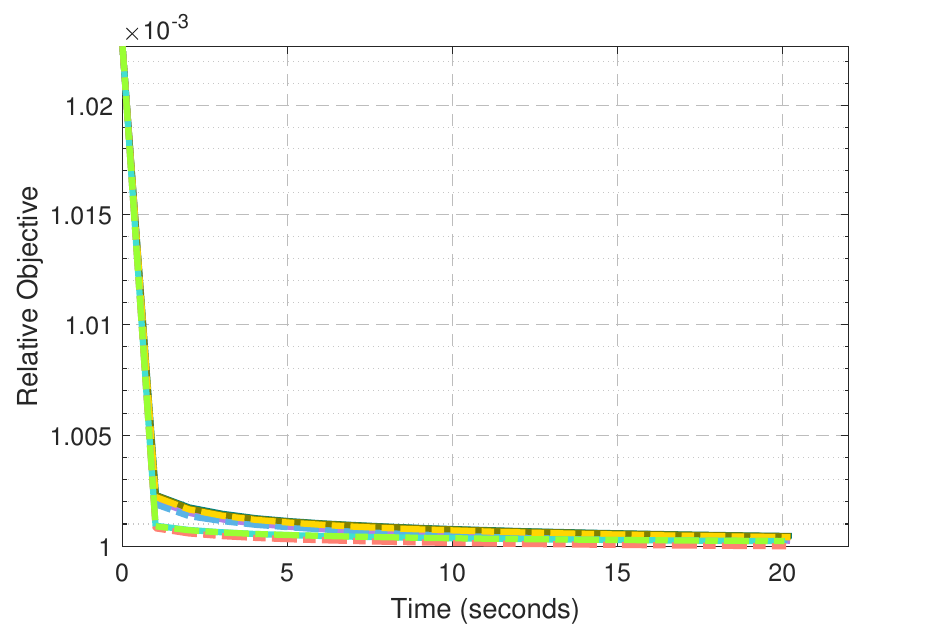}
		\caption{Cifar.}
	\end{subfigure}
	\begin{subfigure}{0.24\textwidth}
		\includegraphics[width=\textwidth]{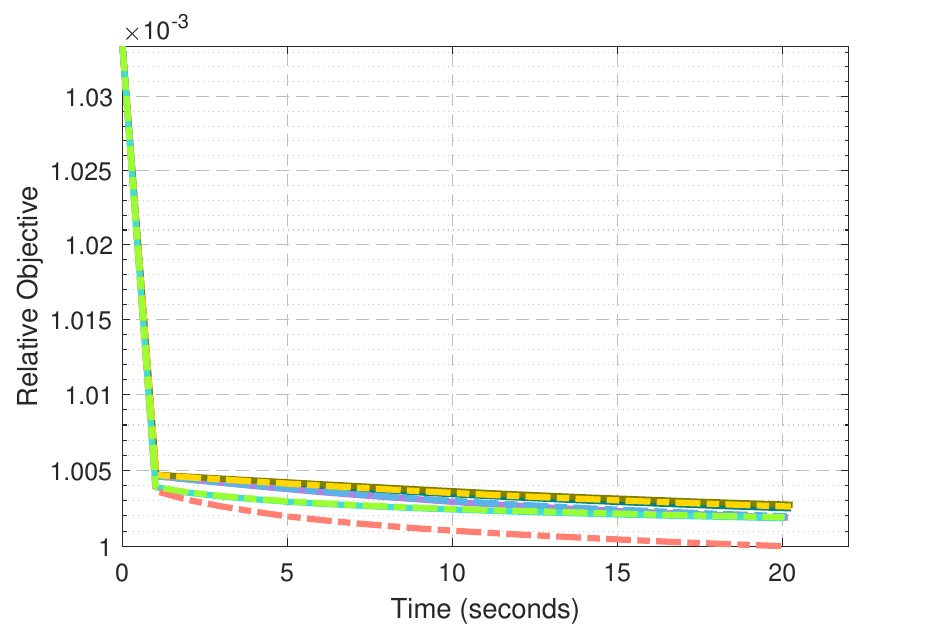}
		\caption{Gisette.}
	\end{subfigure}
	\begin{subfigure}{0.24\textwidth}
		\includegraphics[width=\textwidth]{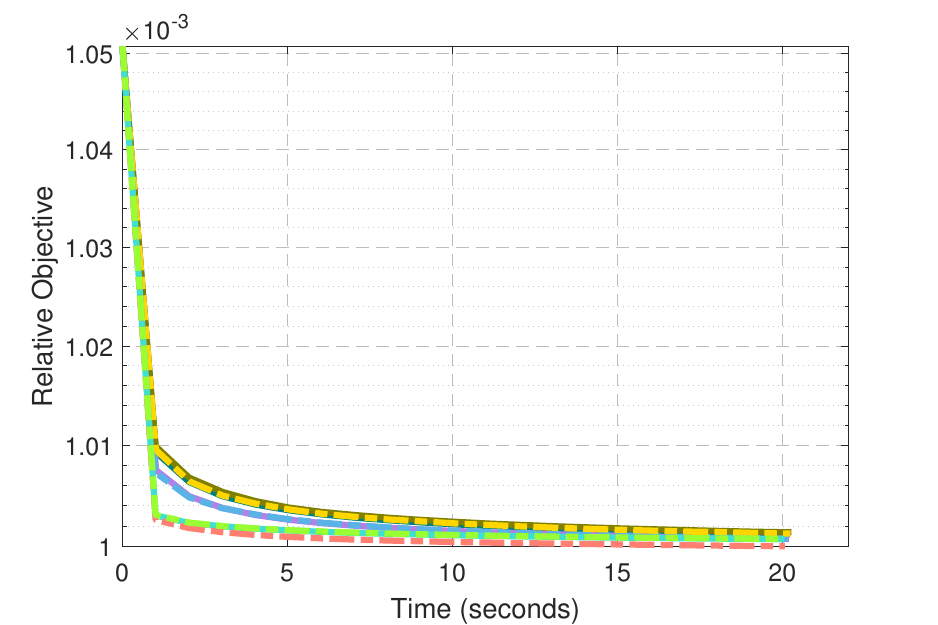}
		\caption{MNIST.}
	\end{subfigure}
	\begin{subfigure}{0.24\textwidth}
		\includegraphics[width=\textwidth]{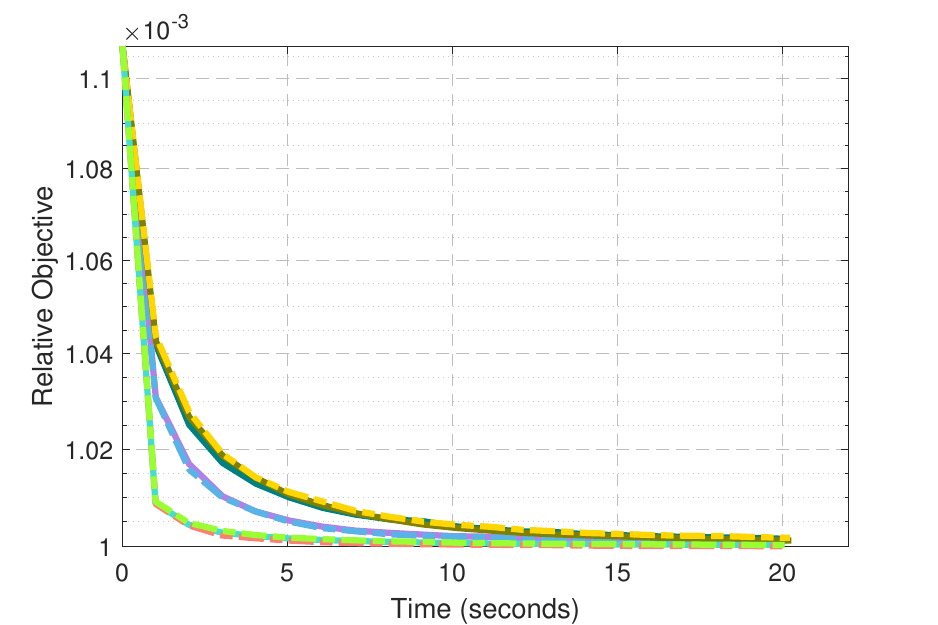}
		\caption{TEXT1.}
	\end{subfigure}
	\caption{Convergence curves of the compared methods for the PCA problem in the online setting.}
	\label{fig:pcaonline}
\end{figure}

\begin{figure}[htbp]
	\centering
	\centering
	\begin{subfigure}{0.7\textwidth}
		\includegraphics[width=\textwidth]{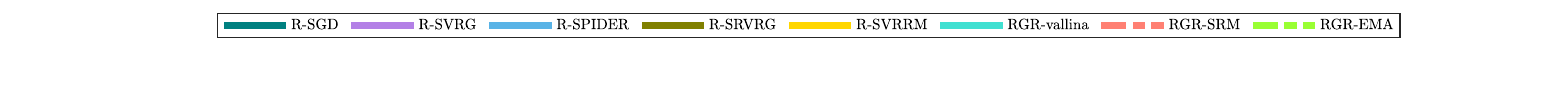}
	\end{subfigure}
	
	\begin{subfigure}{0.24\textwidth}
		\includegraphics[width=\textwidth]{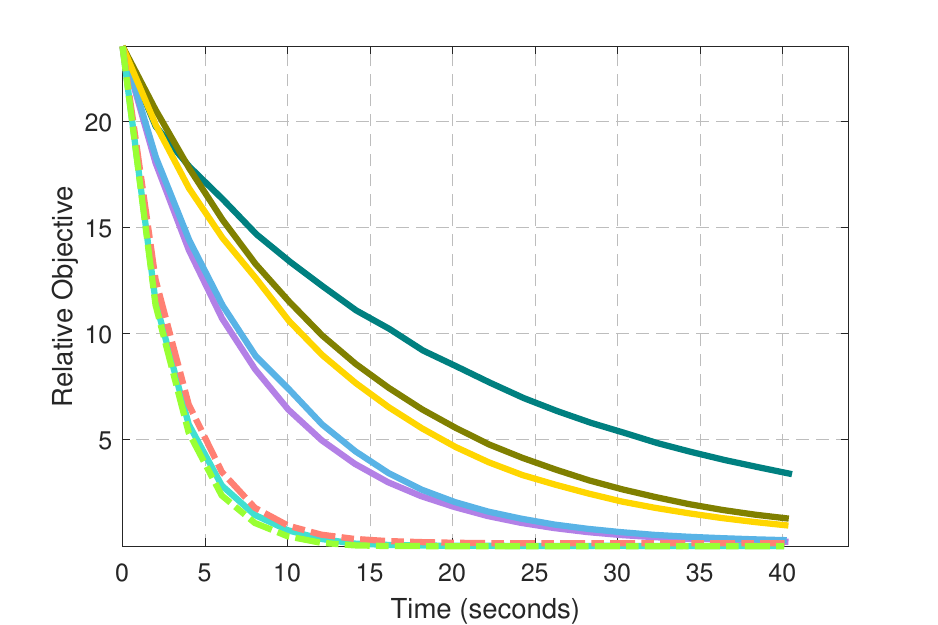}
		\caption{TDT2.}
	\end{subfigure}
	\begin{subfigure}{0.24\textwidth}
		\includegraphics[width=\textwidth]{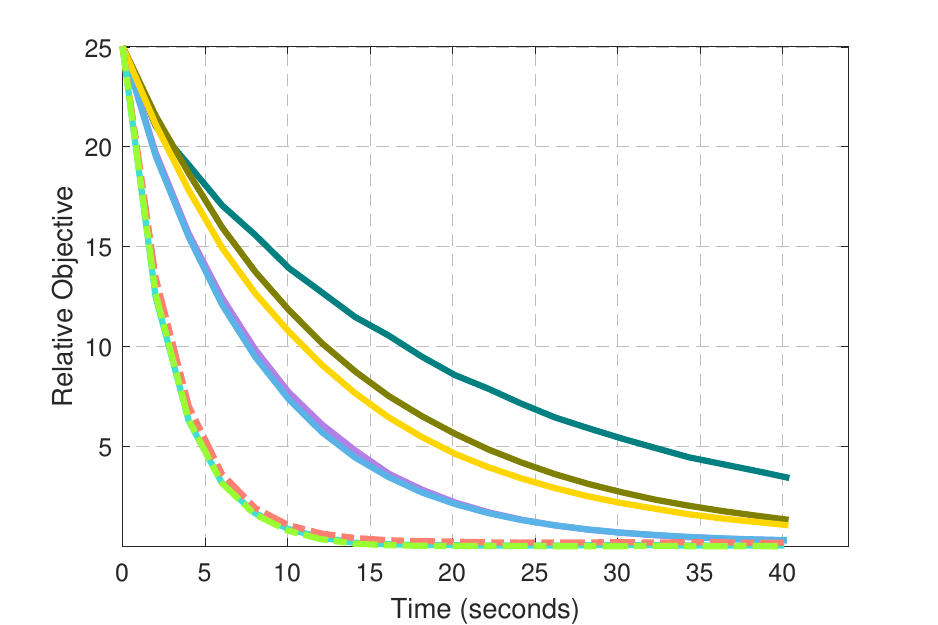}
		\caption{20News.}
	\end{subfigure}
	\begin{subfigure}{0.24\textwidth}
		\includegraphics[width=\textwidth]{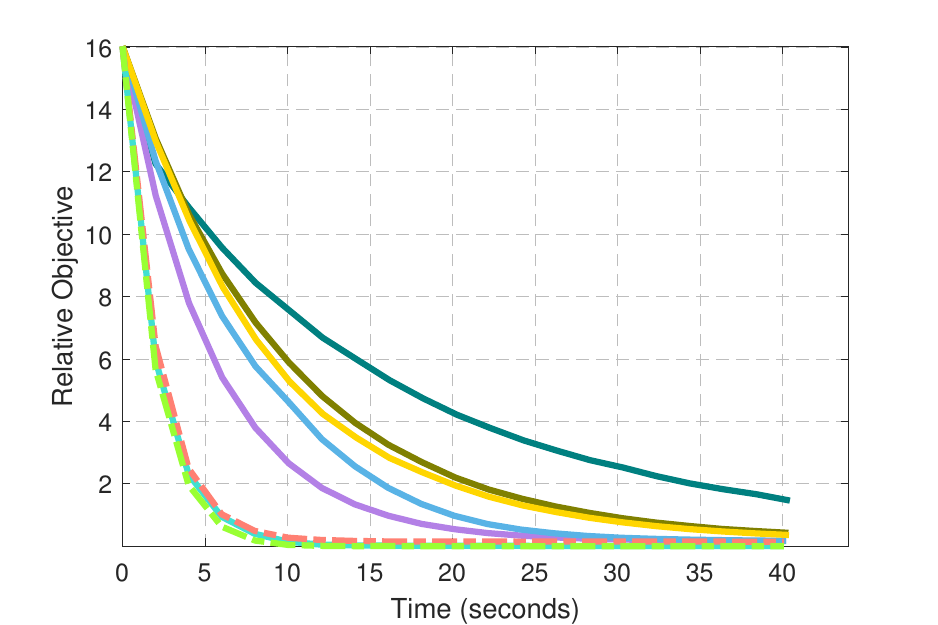}
		\caption{Sector.}
	\end{subfigure}
	\begin{subfigure}{0.24\textwidth}
		\includegraphics[width=\textwidth]{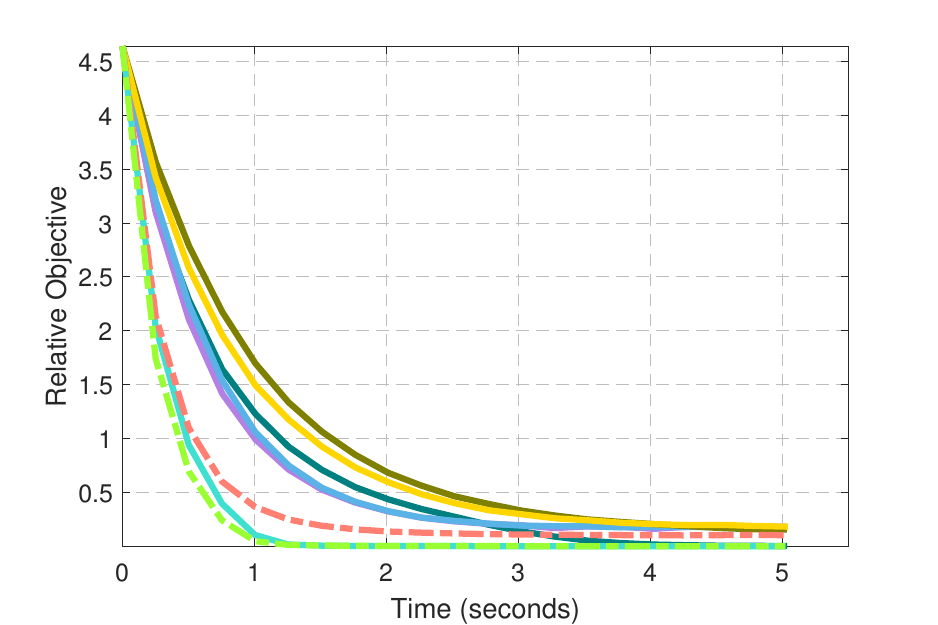}
		\caption{mushroom.}
	\end{subfigure}
	
	\begin{subfigure}{0.24\textwidth}
		\includegraphics[width=\textwidth]{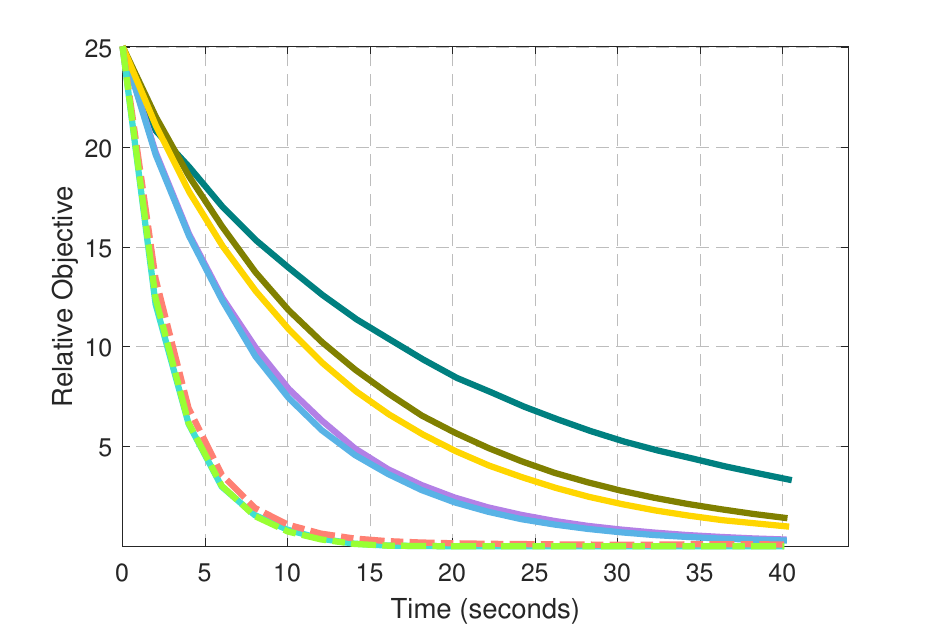}
		\caption{Cifar.}
	\end{subfigure}
	\begin{subfigure}{0.24\textwidth}
		\includegraphics[width=\textwidth]{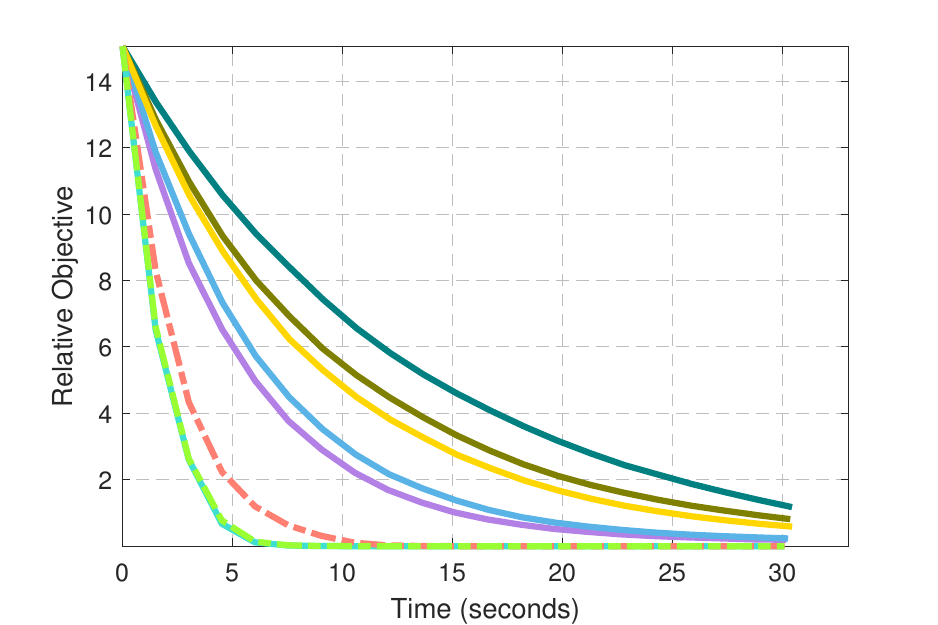}
		\caption{Gisette.}
	\end{subfigure}
	\begin{subfigure}{0.24\textwidth}
		\includegraphics[width=\textwidth]{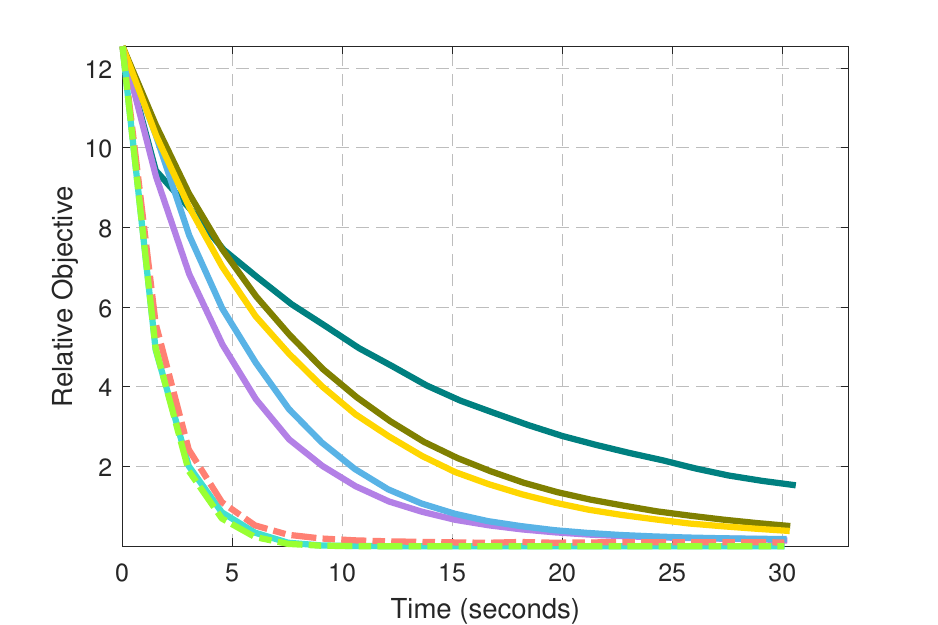}
		\caption{MNIST.}
	\end{subfigure}
	\begin{subfigure}{0.24\textwidth}
		\includegraphics[width=\textwidth]{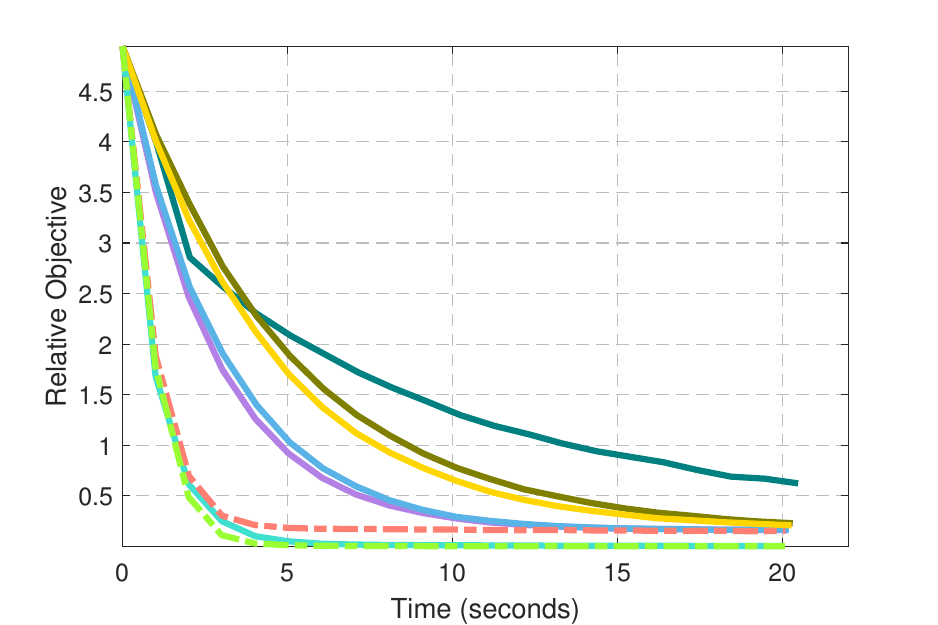}
		\caption{TEXT1.}
	\end{subfigure}
	\caption{Convergence curves of the compared methods for the LRMC problem in the online setting.}
	\label{fig:lrmconline}
\end{figure}

\begin{figure}[htbp]
	\centering
	\captionsetup[subfigure]{justification=centering, font=small}
	\begin{subfigure}{0.24\textwidth}
		\includegraphics[width=\textwidth]{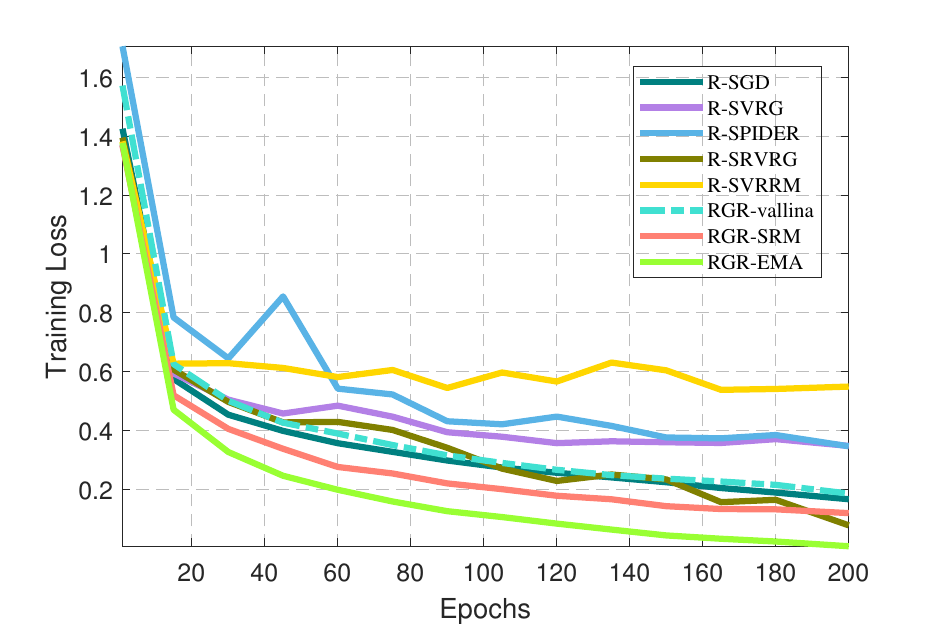}
		\caption{Train loss\\ on Cifar.}
	\end{subfigure}
	\centering
	\begin{subfigure}{0.24\textwidth}
		\includegraphics[width=\textwidth]{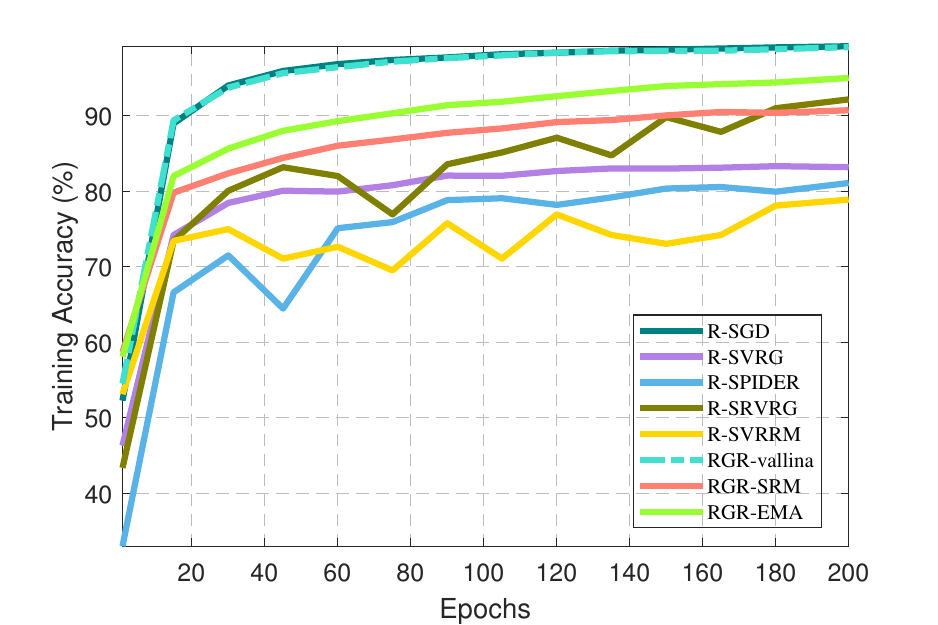}
		\caption{Train Acc\\ on Cifar.}
	\end{subfigure}
	\begin{subfigure}{0.24\textwidth}
		\includegraphics[width=\textwidth]{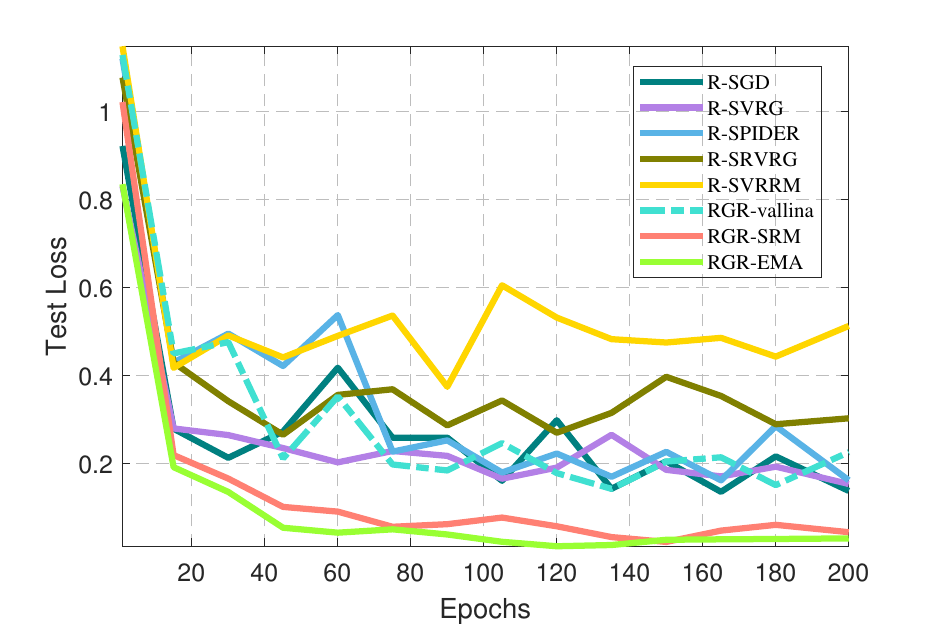}
		\caption{Test loss\\ on Cifar.}
	\end{subfigure}
	\begin{subfigure}{0.24\textwidth}
		\includegraphics[width=\textwidth]{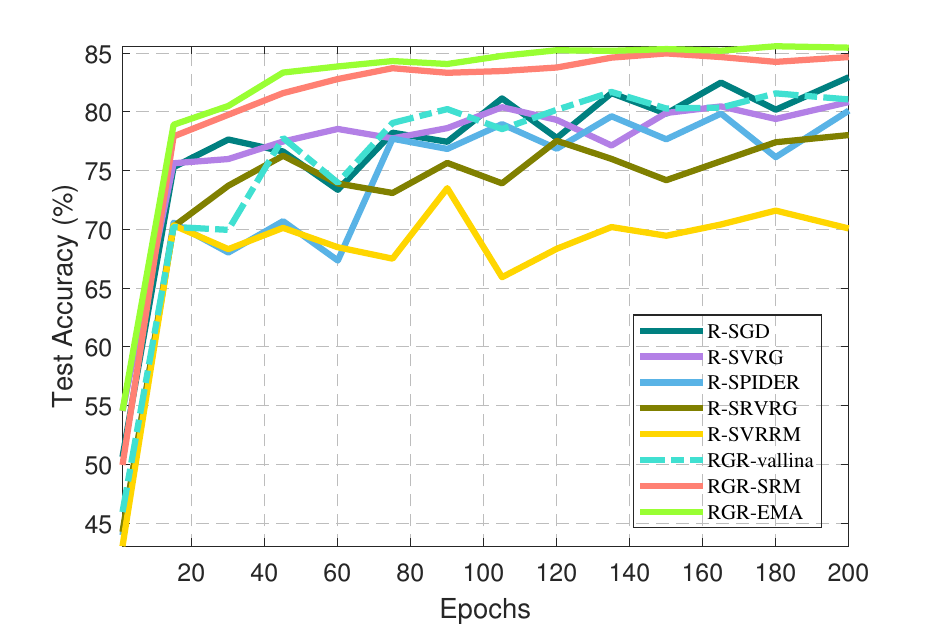}
		\caption{Test Acc\\ on Cifar.}
	\end{subfigure}

	\begin{subfigure}{0.24\textwidth}
		\includegraphics[width=\textwidth]{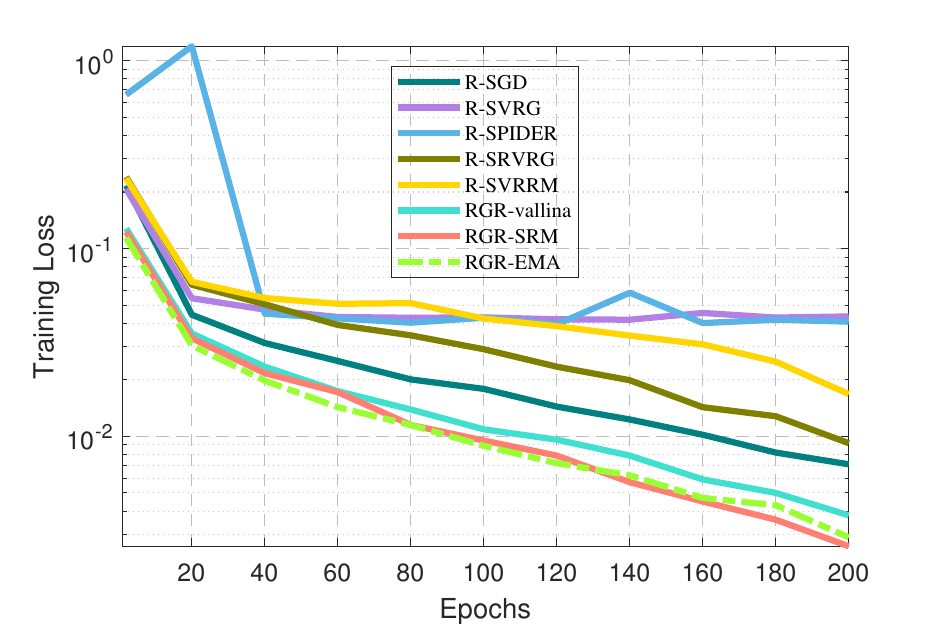}
		\caption{Train loss\\ on MNIST.}
	\end{subfigure}
	\begin{subfigure}{0.24\textwidth}
		\includegraphics[width=\textwidth]{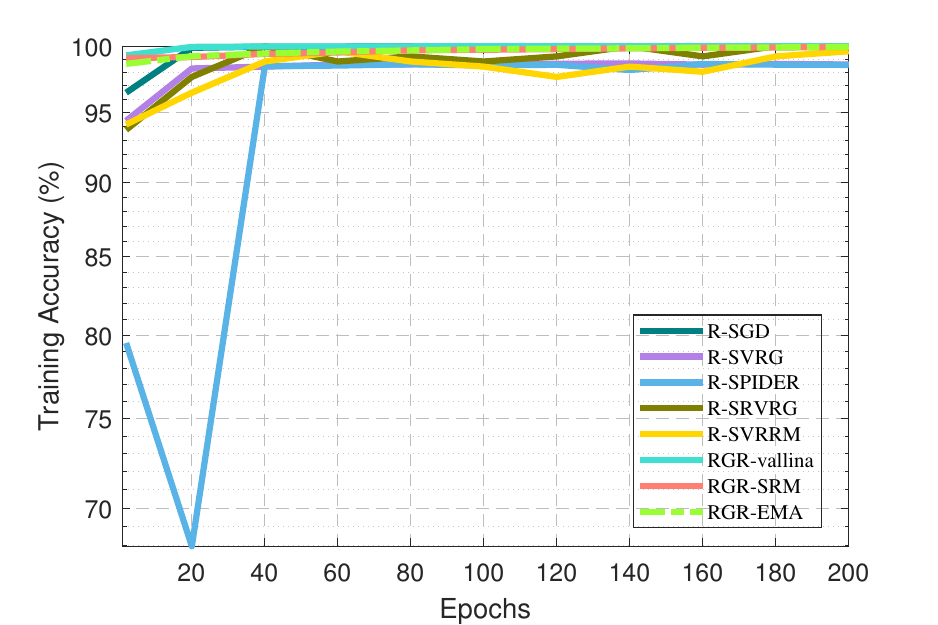}
		\caption{Train Acc\\ on MNIST.}
	\end{subfigure}
	\begin{subfigure}{0.24\textwidth}
		\includegraphics[width=\textwidth]{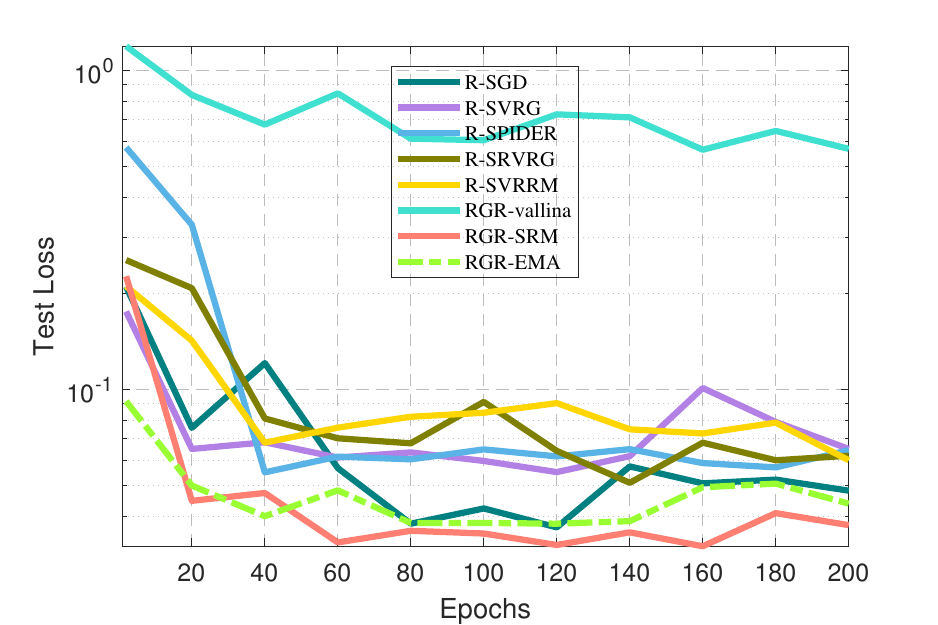}
		\caption{Test loss\\ on MNIST.}
	\end{subfigure}
	\begin{subfigure}{0.24\textwidth}
		\includegraphics[width=\textwidth]{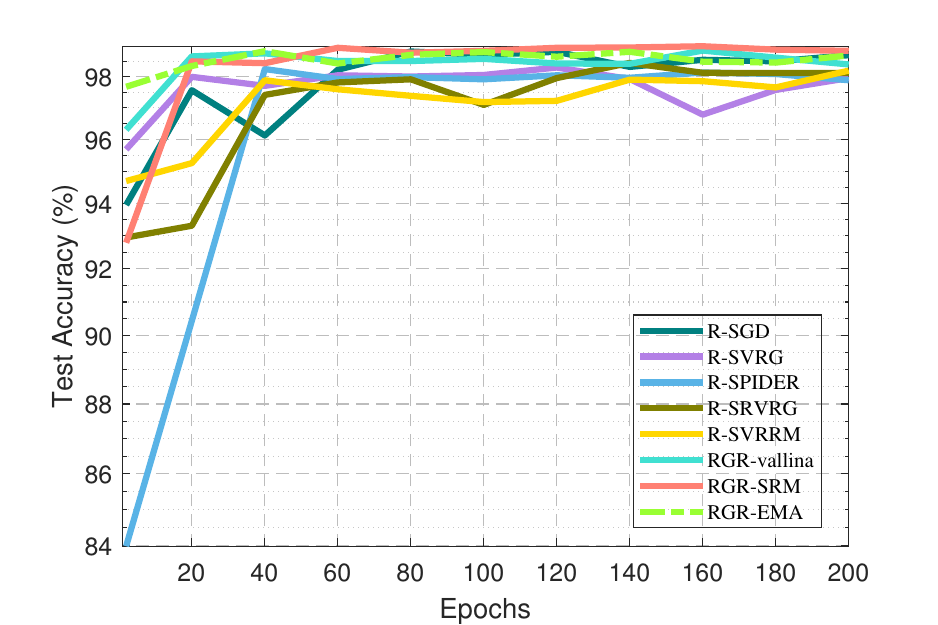}
		\caption{Test Acc\\ on MNIST.}
	\end{subfigure}
	\caption{Training and test performance of the compared methods for OCNNs in the online setting, including loss and accuracy.}
	\label{fig:ocnn}
\end{figure}

\begin{figure}[htbp]
	\centering
	\begin{subfigure}{0.24\textwidth}
		\includegraphics[width=\textwidth]{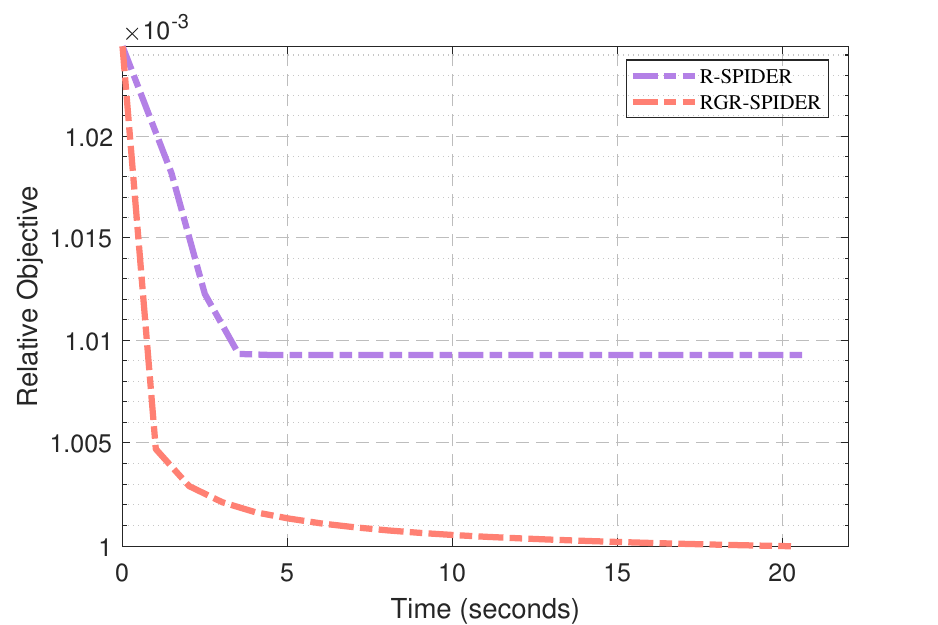}
		\caption{TDT2.}
	\end{subfigure}
	\begin{subfigure}{0.24\textwidth}
		\includegraphics[width=\textwidth]{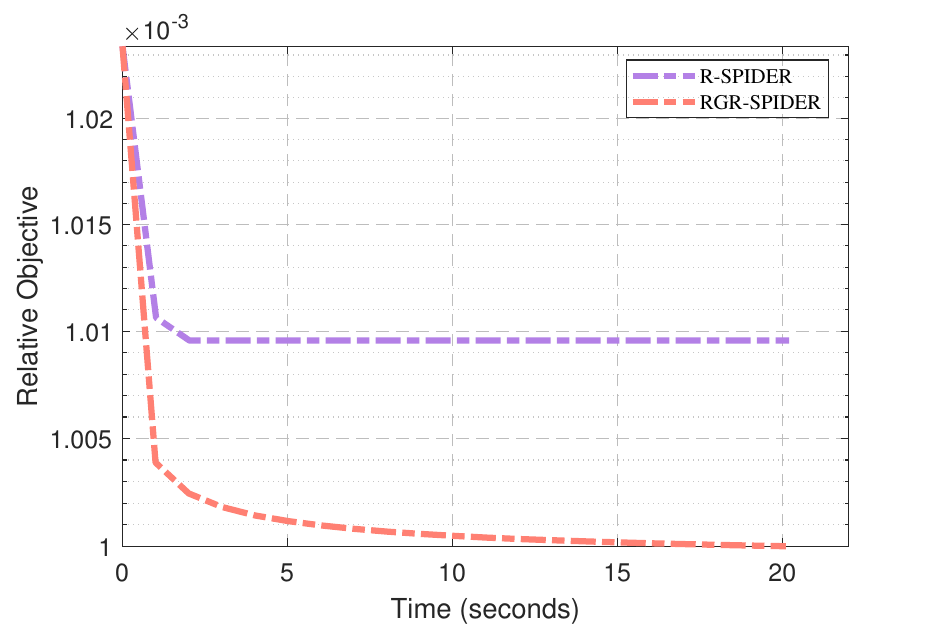}
		\caption{20News.}
	\end{subfigure}
	\begin{subfigure}{0.24\textwidth}
		\includegraphics[width=\textwidth]{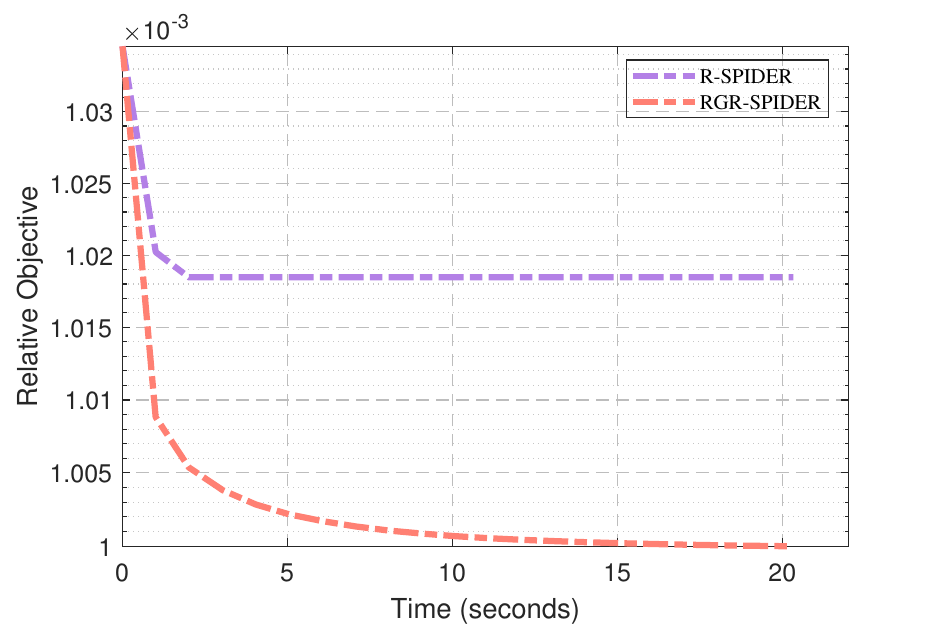}
		\caption{Sector.}
	\end{subfigure}
	\begin{subfigure}{0.24\textwidth}
		\includegraphics[width=\textwidth]{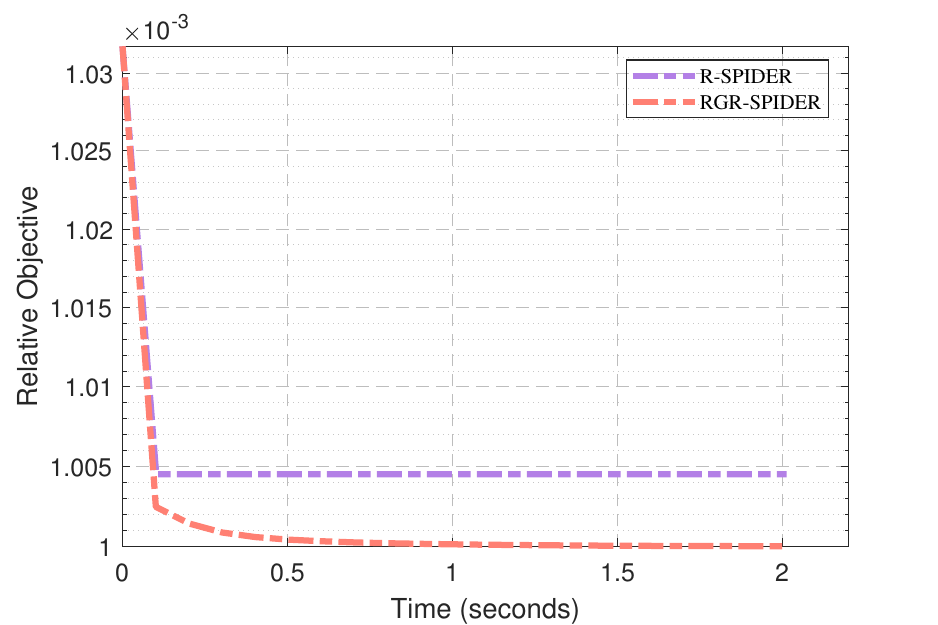}
		\caption{mushroom.}
	\end{subfigure}
	
	\begin{subfigure}{0.24\textwidth}
		\includegraphics[width=\textwidth]{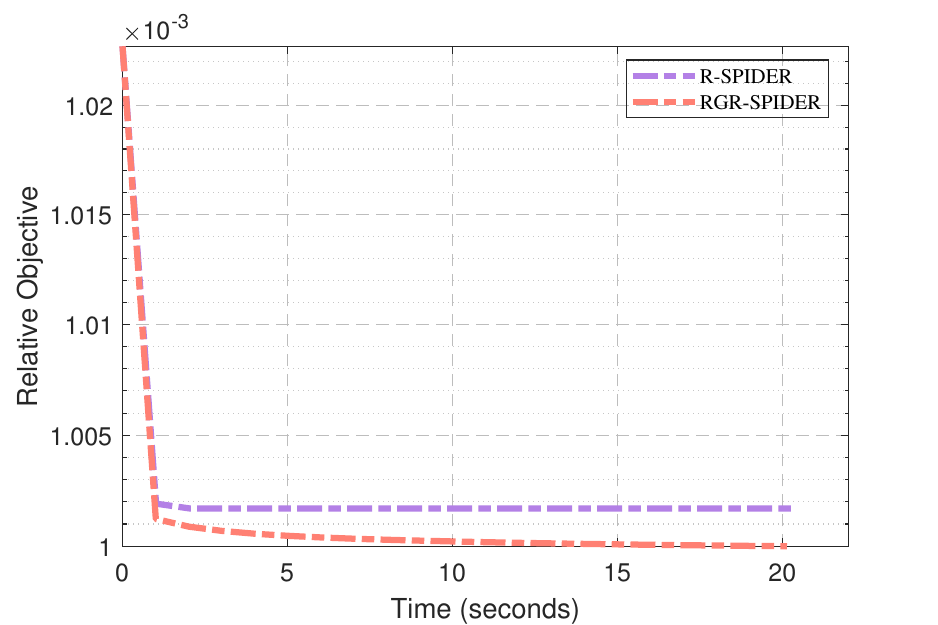}
		\caption{Cifar.}
	\end{subfigure}
	\begin{subfigure}{0.24\textwidth}
		\includegraphics[width=\textwidth]{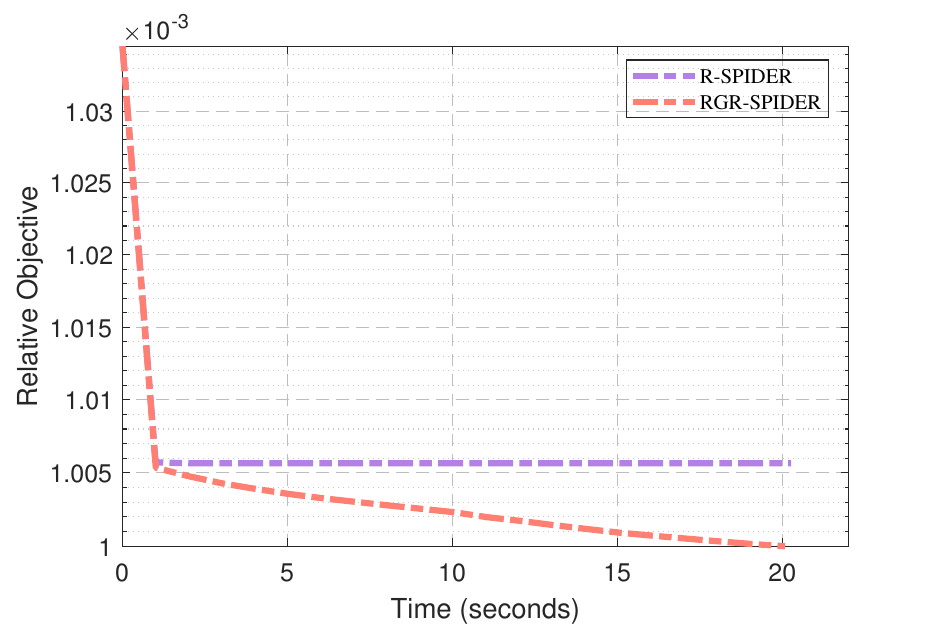}
		\caption{Gisette.}
	\end{subfigure}
	\begin{subfigure}{0.24\textwidth}
		\includegraphics[width=\textwidth]{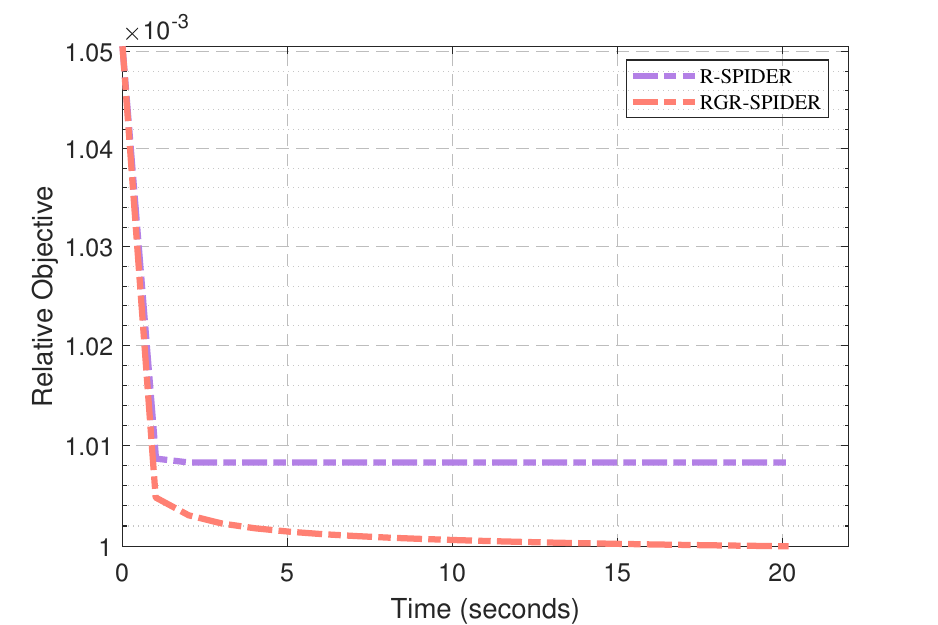}
		\caption{MNIST.}
	\end{subfigure}
	\begin{subfigure}{0.24\textwidth}
		\includegraphics[width=\textwidth]{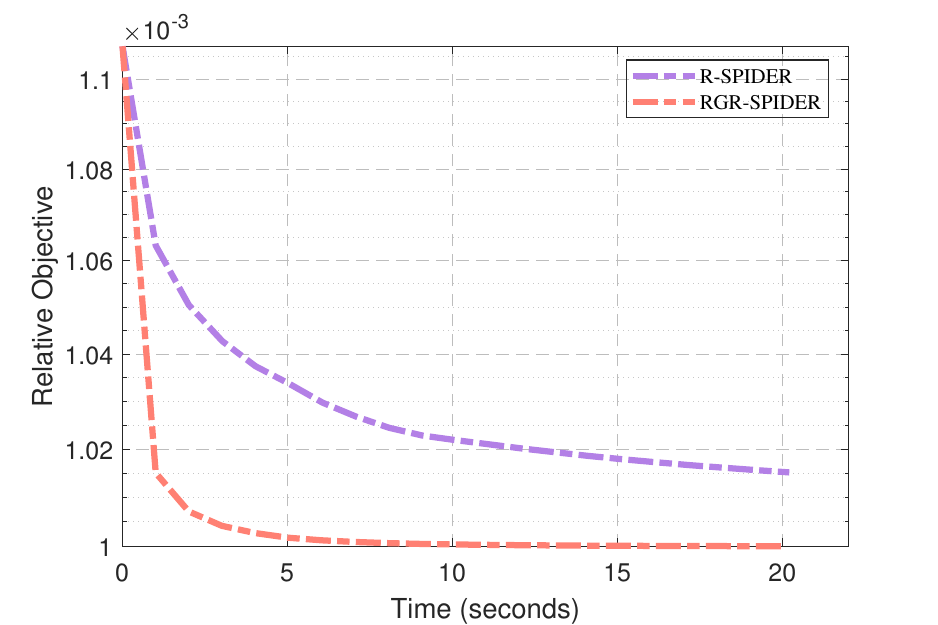}
		\caption{TEXT1.}
	\end{subfigure}
	\caption{Comparison between Euclidean-space and Riemannian tangent-space implementations of the {\sf SPIDER} estimator for finite-sum PCA.}\label{fig:comparespider}
\end{figure}

\begin{figure}[htbp]
	\centering
	\begin{subfigure}{0.24\textwidth}
		\includegraphics[width=\textwidth]{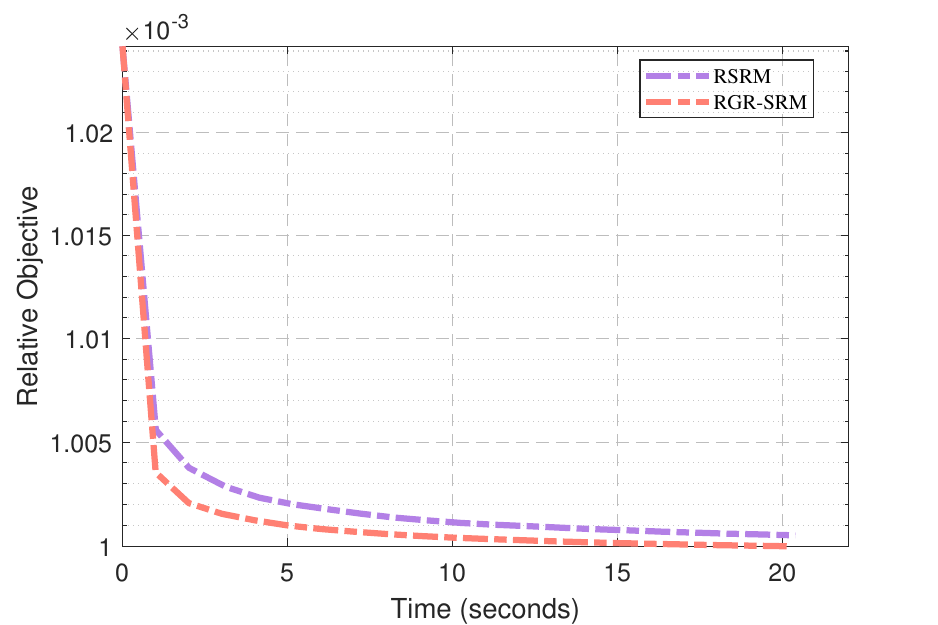}
		\caption{TDT2.}
	\end{subfigure}
	\begin{subfigure}{0.24\textwidth}
		\includegraphics[width=\textwidth]{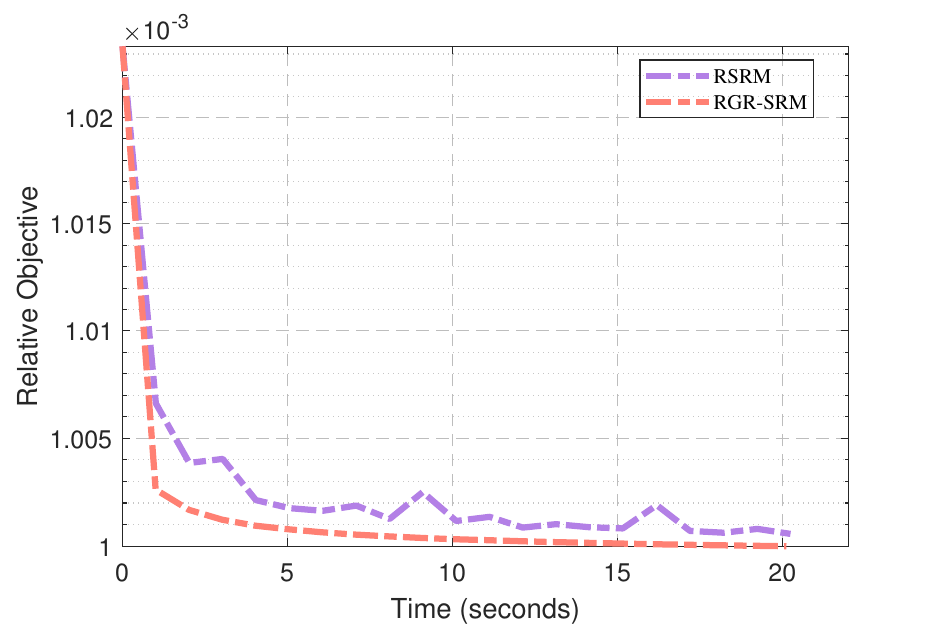}
		\caption{20News.}
	\end{subfigure}
	\begin{subfigure}{0.24\textwidth}
		\includegraphics[width=\textwidth]{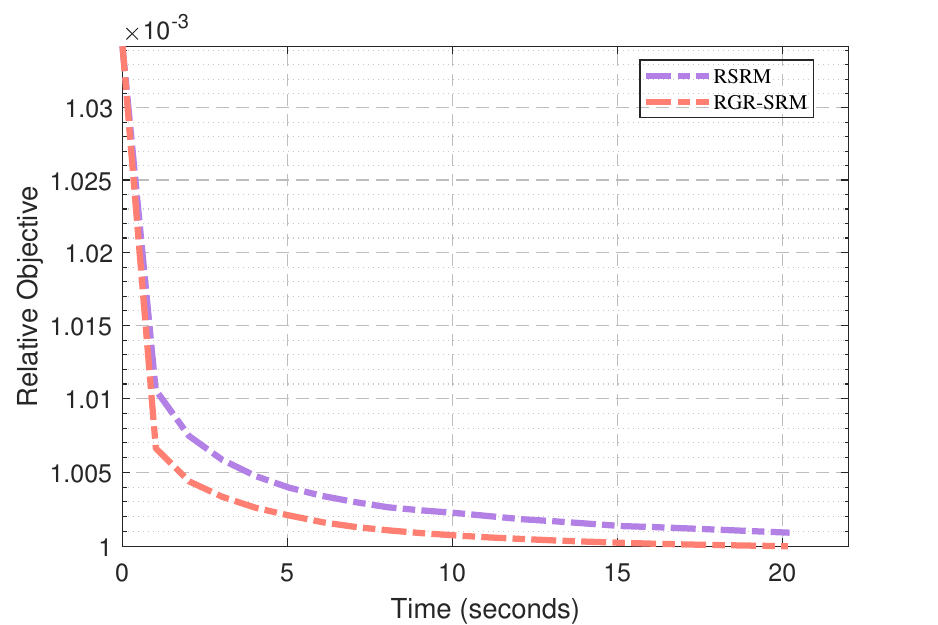}
		\caption{Sector.}
	\end{subfigure}
	\begin{subfigure}{0.24\textwidth}
		\includegraphics[width=\textwidth]{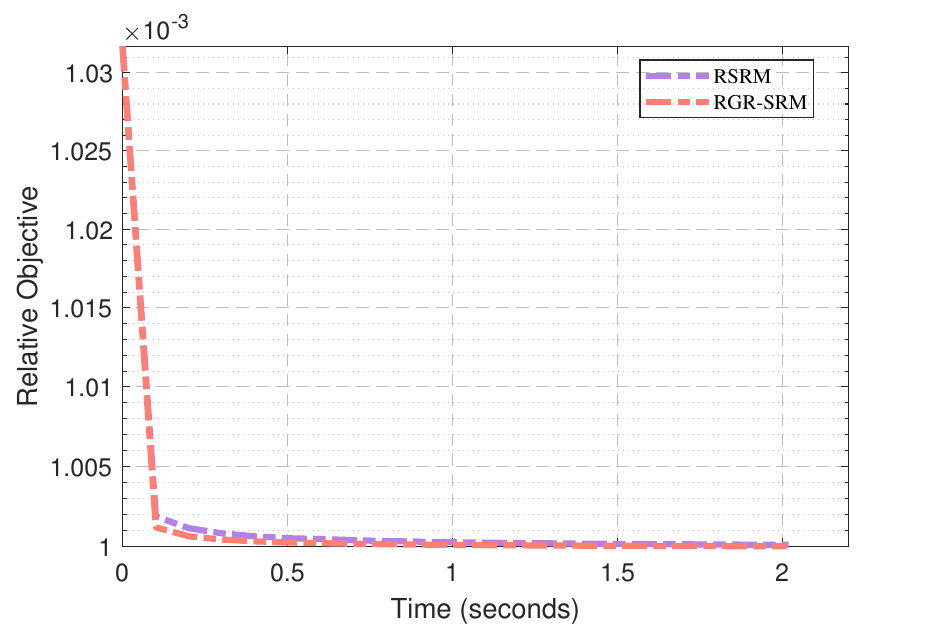}
		\caption{mushroom.}
	\end{subfigure}
	
	\begin{subfigure}{0.24\textwidth}
		\includegraphics[width=\textwidth]{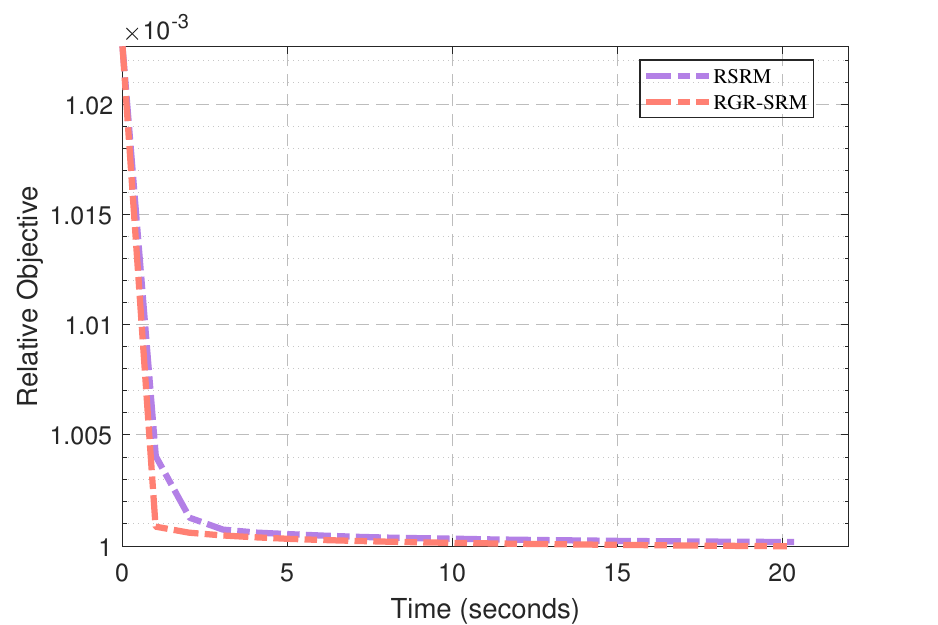}
		\caption{Cifar.}
	\end{subfigure}
	\begin{subfigure}{0.24\textwidth}
		\includegraphics[width=\textwidth]{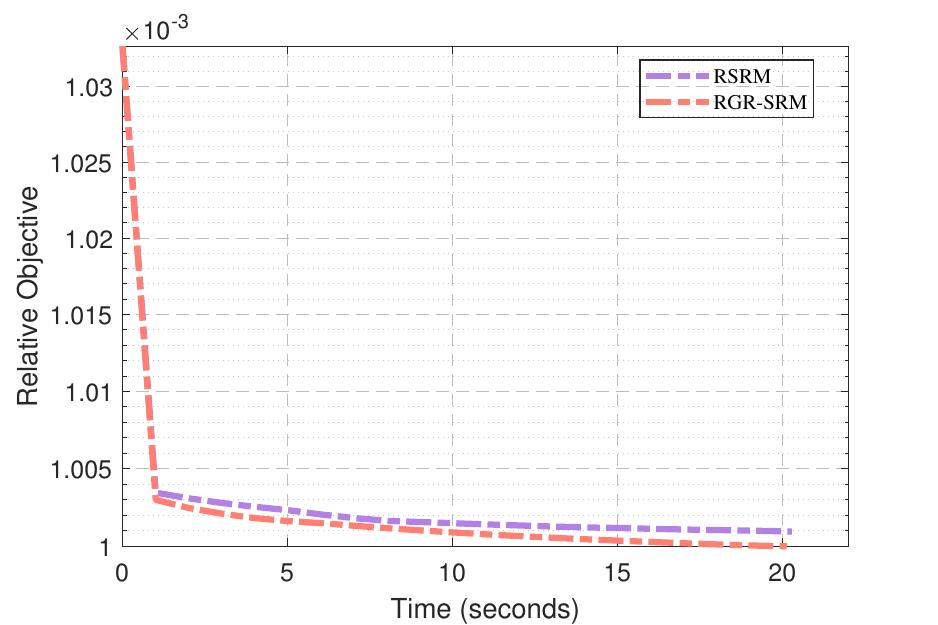}
		\caption{Gisette.}
	\end{subfigure}
	\begin{subfigure}{0.24\textwidth}
		\includegraphics[width=\textwidth]{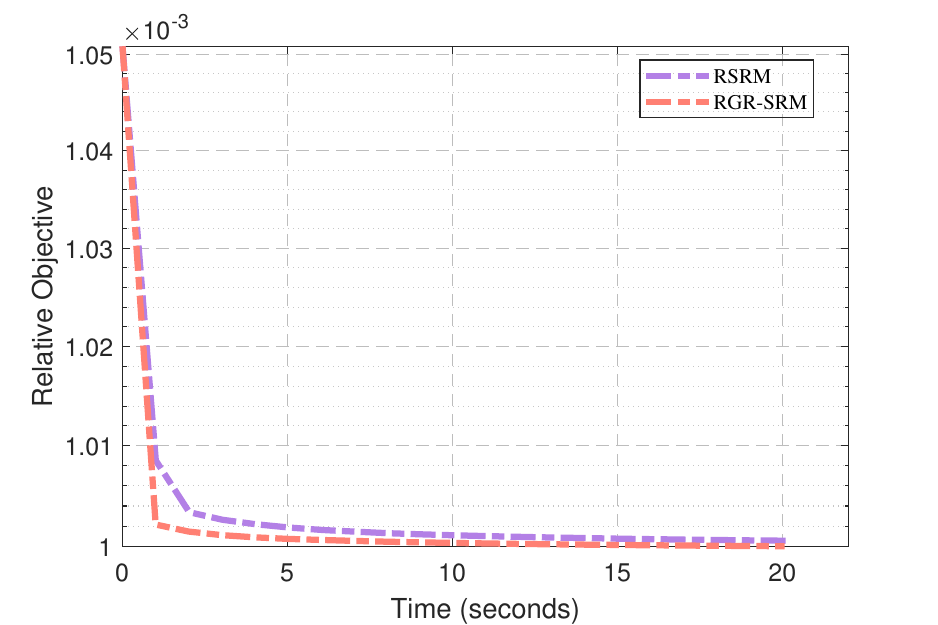}
		\caption{MNIST.}
	\end{subfigure}
	\begin{subfigure}{0.24\textwidth}
		\includegraphics[width=\textwidth]{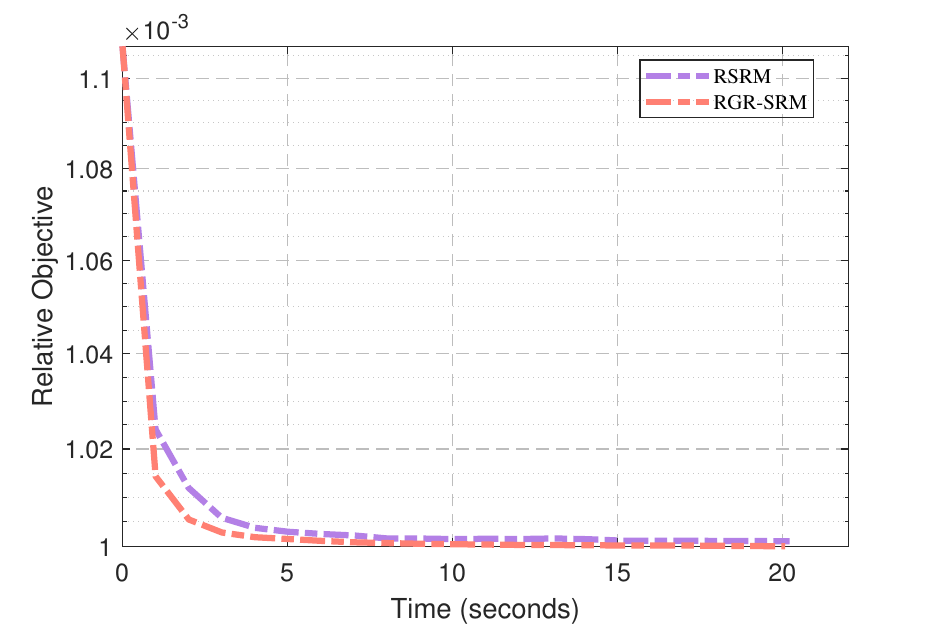}
		\caption{TEXT1.}
	\end{subfigure}
	\caption{Comparison between Euclidean-space and Riemannian tangent-space implementations of the {\sf SRM} estimator for online PCA.}\label{fig:comparesrm}
\end{figure}

\begin{figure}[htbp]
	\centering
	\begin{subfigure}{0.24\textwidth}
		\includegraphics[width=\textwidth]{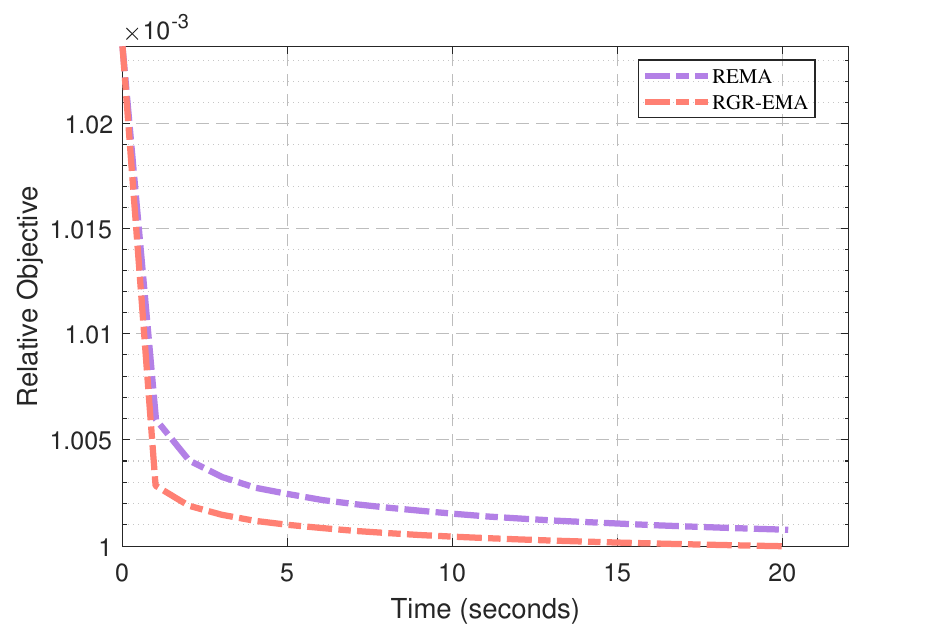}
		\caption{TDT2.}
	\end{subfigure}
	\begin{subfigure}{0.24\textwidth}
		\includegraphics[width=\textwidth]{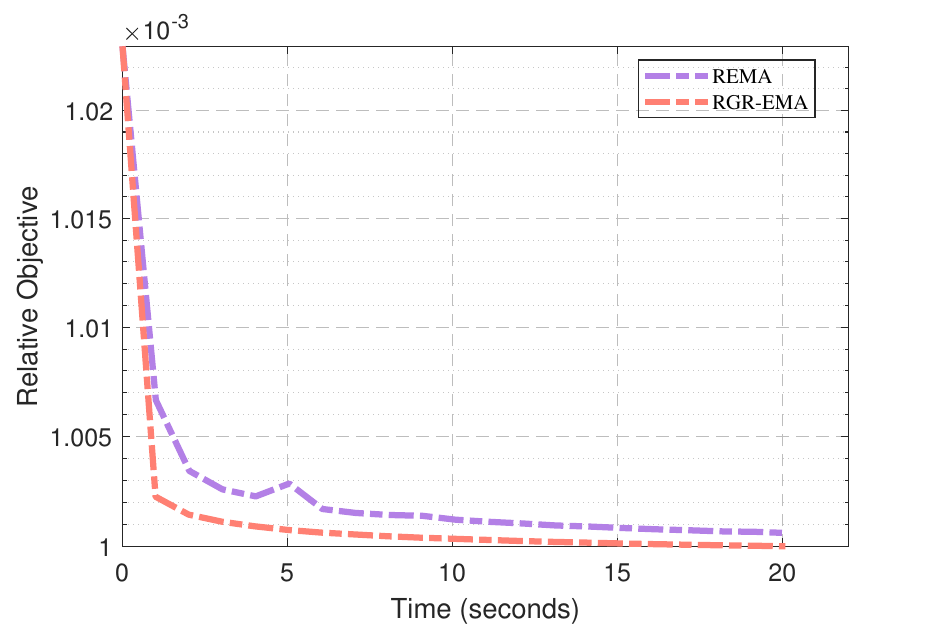}
		\caption{20News.}
	\end{subfigure}
	\begin{subfigure}{0.24\textwidth}
		\includegraphics[width=\textwidth]{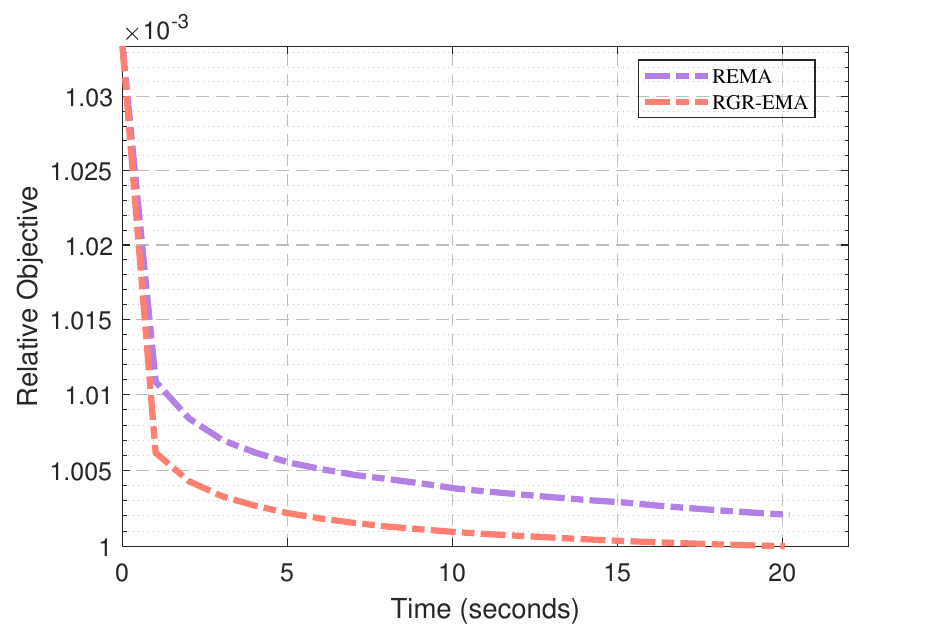}
		\caption{Sector.}
	\end{subfigure}
	\begin{subfigure}{0.24\textwidth}
		\includegraphics[width=\textwidth]{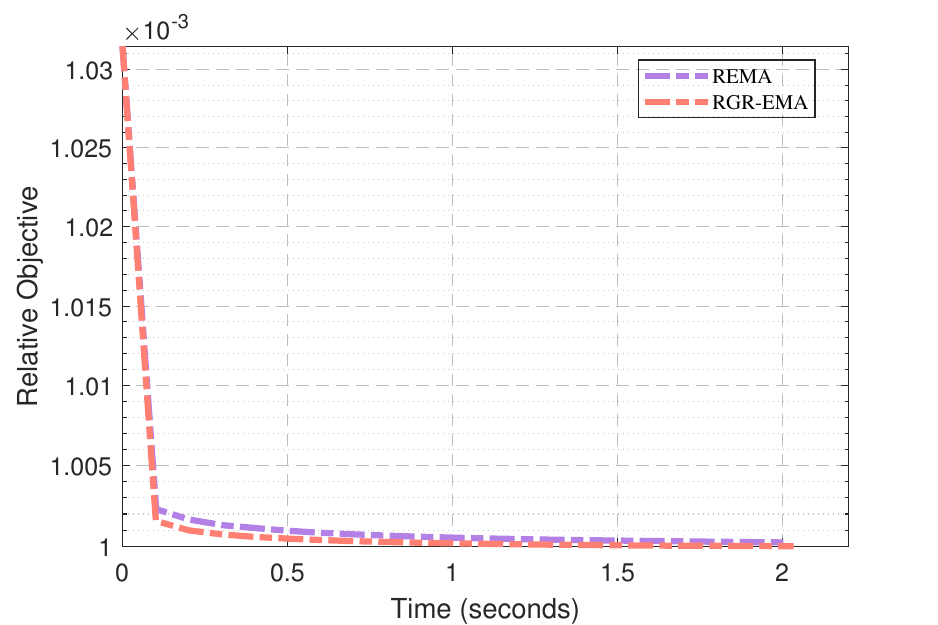}
		\caption{mushroom.}
	\end{subfigure}
	
	\begin{subfigure}{0.24\textwidth}
		\includegraphics[width=\textwidth]{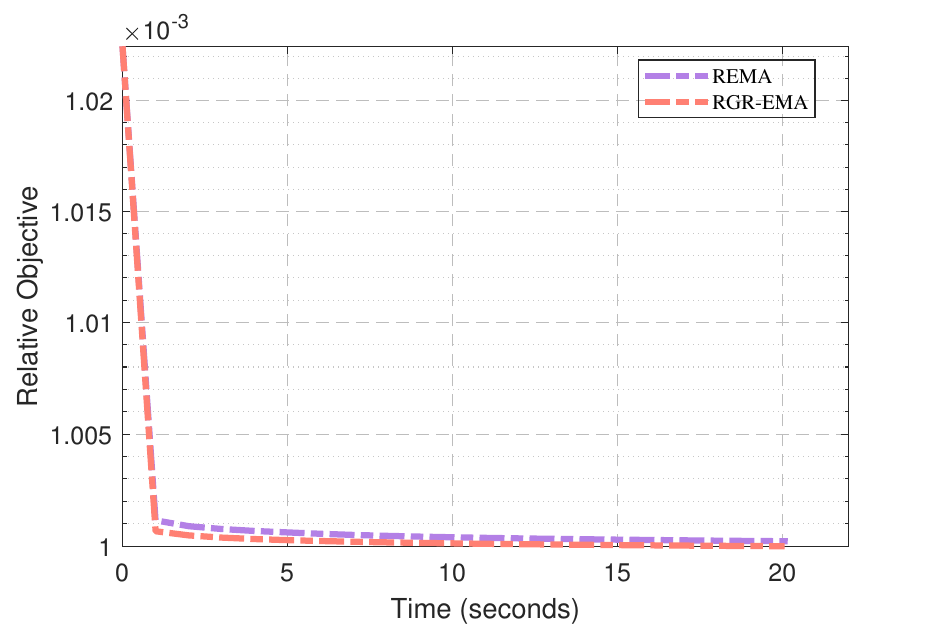}
		\caption{Cifar.}
	\end{subfigure}
	\begin{subfigure}{0.24\textwidth}
		\includegraphics[width=\textwidth]{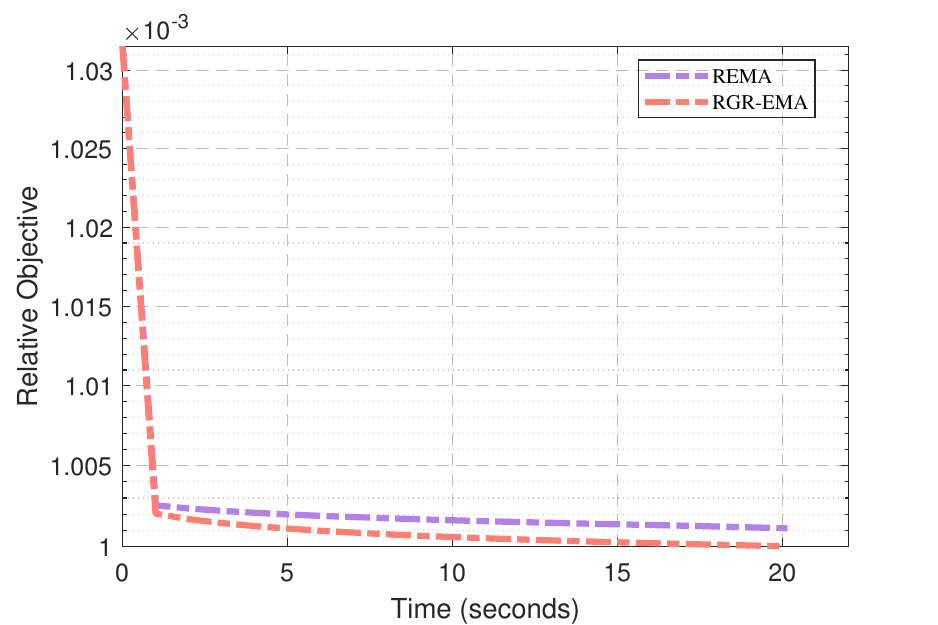}
		\caption{Gisette.}
	\end{subfigure}
	\begin{subfigure}{0.24\textwidth}
		\includegraphics[width=\textwidth]{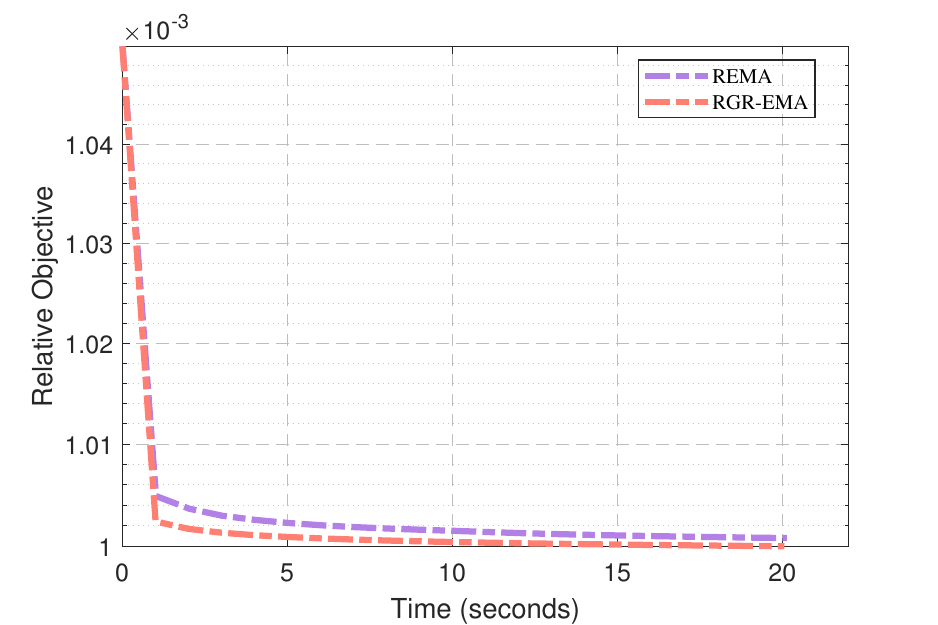}
		\caption{MNIST.}
	\end{subfigure}
	\begin{subfigure}{0.24\textwidth}
		\includegraphics[width=\textwidth]{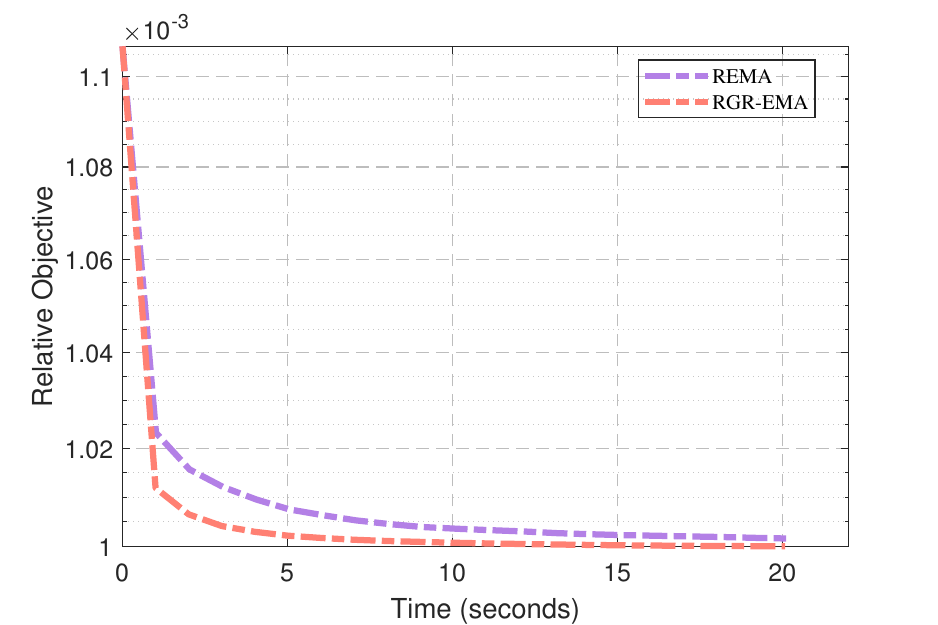}
		\caption{TEXT1.}
	\end{subfigure}
	\caption{Comparison between Euclidean-space and Riemannian tangent-space implementations of the {\sf EMA} estimator for online PCA.}\label{fig:compareema}
\end{figure}

\appendix

\newpage
\onecolumn
\appendix
\huge{Appendix}

\normalsize
The appendix is organized as follows.

Appendix~\ref{app:sect:notations} summarizes the notation used throughout the paper.

Appendix~\ref{appendixema} introduce the {\sf REMA} method discussed in
Section~\ref{sec:Euclidean-tangent}.

Appendix~\ref{app:sect:lemmas} presents several relevant lemmas.

Appendix~\ref{app:sect:proof} provides the proofs for the convergence analysis in Section~\ref{sec:com}.

\section{Notations}\label{app:sect:notations}

In this paper, bold lowercase letters represent vectors, and uppercase letters denote real-valued matrices. The following notations are used throughout this paper.

\begin{itemize}[leftmargin=12pt,itemsep=0.2ex]
	
	\item Vectors are denoted by boldface lowercase letters, and matrices by boldface uppercase letters.
	\item $\X^\top$: Transpose of matrix $\X$.
	
	\item $\mathbf{I}_r$: The $r\times r$ identity matrix.
	\item $\mathbf{0}_{n\times r}$: The zero matrix in $\mathbb{R}^{n\times r}$.
	
	\item $\A_{ij}$: The $(i,j)$-th entry of matrix $\A$, that is, the element in the $i$-th row and $j$-th column.

\item $\|\X\|_{\fro}$: The Frobenius norm of matrix $\X \in \mathbb{R}^{n\times r}$, defined as $\|\X\|_{\fro} = \sqrt{\langle \X,\X\rangle}.$

\item $\langle \A,\B\rangle$: The standard inner product on $\mathbb{R}^{n\times r}$, defined as $\langle \A,\B\rangle = \mathrm{tr}(\A\trans \B)
= \sum_{i=1}^{n}\sum_{j=1}^{r} \A_{ij}\B_{ij},$ where $\A,\B\in\mathbb{R}^{n\times r}$.

	\item $\{A_k\}_{k=0}^{\infty}$, $\{B_k\}_{k=0}^{\infty}$: Infinite sequences indexed by $k = 0, 1, 2, \ldots$.

	\item $[T]\triangleq\{1,\ldots,T\}$ denotes the set of the first $T$
	positive integers.
	
\end{itemize}

\section{Riemannian Exponential Moving Average with Parallel Transport}
\label{appendixema}

This section presents the {\sf REMA} method described in
Algorithm~\ref{alg:REMA}, which serves as the tangent-space counterpart
discussed in Section~\ref{sec:Euclidean-tangent}.

\begin{algorithm}[htbp]
	\caption{Riemannian Exponential Moving Average ({\sf REMA})
		with Parallel Transport}
	\label{alg:REMA}
	\begin{algorithmic}[1]
		
		\State \textbf{Input:}
		Initial point $\X_0\in\mathrm{St}(n,r)$,
		stepsize parameter $\theta>0$,
		EMA weight $\alpha\in(0,1)$,
		mini-batch size $b$, and iteration budget $T$.
		
		\State \textbf{Initialization:}
		Draw an initial mini-batch $\xi_0$ of size $b$ and set
		$
		\G_0=\nabla f(\X_0;\xi_0),
		\
		\V_0
		=
		\G_0-\X_0\G_0^\top\X_0
		\in T_{\X_0}\mathrm{St}(n,r).
		$
		
		\For{$t=1$ \textbf{to} $T$}
		
		\State (S1) Set the stepsize
		$\mu_t=\theta t^{-1/2}$.
		
		\State (S2) Update the iterate via retraction:
		\Statex \hspace{3em}
		$\displaystyle
		\X_t
		=
		\Retr_{\X_{t-1}}(-\mu_t\V_{t-1}).
		$
		
		\State (S3) Draw a random mini-batch $\xi_t$ of size $b$.
		
		\State (S4) Compute the Euclidean stochastic gradient:
		\Statex \hspace{3em}
		$\displaystyle
		\G_t=\nabla f(\X_t;\xi_t).
		$
		
		\State (S5) Map $\G_t$ to a tangent direction:
		\Statex \hspace{3em}
		$\displaystyle
		\widehat{\G}_t
		=
		\G_t-\X_t\G_t^\top\X_t
		\in T_{\X_t}\mathrm{St}(n,r).
		$
		
		\State (S6) Update the Riemannian estimator via
		parallel transport:
		\Statex \hspace{3em}
		$\displaystyle
		\V_t
		=
		\alpha\widehat{\G}_t
		+
		(1-\alpha)
		\mathcal{T}_{\X_{t-1}}^{\X_t}(\V_{t-1}).
		$
		
		\EndFor
		
		\State Draw $\tau$ uniformly from $\{1,\ldots,T\}$.
		\State \textbf{Output:} $\X_\tau$.
		
	\end{algorithmic}
\end{algorithm}

\section{Relevant Lemmas}\label{app:sect:lemmas}

We provide a set of useful lemmas, each independent of context and specific methodologies.

\begin{lemma}[{\cite[Lemma A.5]{yuan2026adaptive}}] \label{lemma:online}
	We let $c>0$ and $S_t\triangleq \sum_{i=0}^t A_i$, where $\{A_i\}_{i=0}^{\infty}$ and $\{S_i\}_{i=0}^{\infty}$ are two non-negative sequences. We have: $\sum_{i=0}^t {A_i}/{\sqrt{c +  S_i}} \leq 2 \sqrt{ c +  S_t  }$. 
\end{lemma}

\begin{lemma}[{\cite[Lemma A.4]{yuan2026adaptive}}] \label{lemma:remove:log}
	Assume that $\{A_i\}_{i=0}^n$ and $\{B_i\}_{i=0}^{n+1}$ are two non-negative sequences with $A_0\leq A_1\leq \ldots \leq A_n$. Then, we have: $$\ts \sum_{t=0}^n A_t (B_t - B_{t+1})\leq [\max_{i=0}^n A_i]\cdot[\max_{j=0}^n B_j].$$ 
\end{lemma}

%


\begin{lemma}[{\cite[Lemma A.6]{yuan2026adaptive}}] \label{lemma:online:pp}
	Let $c>0$ and $S_t\triangleq \sum_{i=0}^t A_i$, where $\{A_i\}_{i=0}^{\infty}$ and $\{S_i\}_{i=0}^{\infty}$ are two non-negative sequences. We have $\ts \sum_{i=0}^t \tfrac{A_i}{ c + S_i} \leq \log( 1 + S_t/c) \leq \tfrac{(c+ S_t)^{p}}{p c^p} - 1$, where $p\in(0,1]$.
\end{lemma}

\begin{lemma}[{\cite[Lemma B.7(a)]{yuan2026adaptive}}] \label{lemma:pq}
		 The {\sf SPIDER} estimator produces a stochastic gradient $\G_t$ that, for all $t$ with $\underline{r_t} \leq t \leq \overline{r_t}$, we have:  $\EEE[\|\G_t- \nabla f(\X_t)\|_{\fro}^2] - \|\G_{t-1} - \nabla f(\X_{t-1})\|_{\fro}^2	\leq \tfrac{L^2}{b} \EEE [  \|\X_{i} - \X_{i-1}\|_{\fro}^2 ]$. Then $\EEE[\|\nabla f(\X_t) - \G_t\|^2_{\fro}] \leq \tfrac{L^2}{b} \sum_{i = (r_t-1)q}^{t-1} \EEE[\|\X_{i+1} - \X_i\|^2_{\fro}]$.
\end{lemma}

\begin{lemma}[{\cite[Lemma B.9]{yuan2026adaptive}}] \label{lemma:idk}
We define $q' \triangleq q\hat{\eta}^{q-1}$, where $\hat{\eta}$ is given in Lemma \ref{lemma:etat}. For any $t$ satisfying $(r_t-1)q \leq t \leq r_tq-1$, the following results hold:
\begin{enumerate}[label=\textbf{(\alph*)}, leftmargin=18pt, itemsep=2pt, topsep=1pt, parsep=0pt, partopsep=0pt]
		
		\item $\sum_{j=(r_t-1)q}^t [\tfrac{1}{\eta_j}\sum_{i=(r_j-1)q}^{j-1}\|\X_{i+1} - \X_i\|_{\fro}^2] \leq  q'\sum_{i=(r_t-1)q}^{t-1}\tfrac{1}{\eta_i}\|\X_{i+1} - \X_i\|_{\fro}^2$.	

        \item $\sum_{j=(r_t-1)q}^t [\sum_{i=(r_j-1)q}^{j-1}\|\X_{i+1} - \X_i\|_{\fro}^2] \leq  (q-1) \sum_{i=(r_t-1)q}^{t-1}\|\X_{i+1} - \X_i\|_{\fro}^2$.	

         \item $\sum_{t=0}^{T-1} [\sum_{i=(r_t-1)q}^{t-1}\|\X_{i+1} - \X_i\|_{\fro}^2] \leq  (q-1) \sum_{t=0}^{T-1} \|\X_{t+1} - \X_t\|_{\fro}^2$.	
		
\end{enumerate}
\end{lemma}

\begin{lemma}\label{lemma:etat}
We define $\tfrac{1}{\eta^2_{t}} \triangleq \tfrac{1}{\eta^2_{t-1}} + \|\Q_{t-1}\|_{\fro}^2$ and $\tilde{\Deltas}_t \triangleq \nabla f(\X_t; \xi_t) - \nabla f(\X_{t-1}; \xi_t)$. For all $t$ with $(r_t-1)q \leq t \leq r_tq-1$, then we have: $\tfrac{1 / \eta_{t+1}}{1 / \eta_{t}} \leq \hat{\eta}$ almost surely, where $\hat{\eta} \triangleq \sqrt{	1+	4q\etaup^2(C_f^2+4(q-1)L^2r)}$.

\begin{proof} 
For the term $\|\tilde{\Deltas}_i\|_{\fro}^2$, we exploit the $L$-smoothness of $f$ and the orthogonality of $\mathbf{X}_i \in \mathrm{St}(n,r)$ (where $\|\mathbf{X}_i\|_{\fro}^2 = r$) to derive the following inequalities almost surely:
\beq
\|\tilde{\Deltas}_i\|_{\fro}^2 &=& \|\nabla f(\X_i; \xi_i) - \nabla f(\X_{i-1}; \xi_i)\|_{\fro}^2 \overset{}{\leq} L^2\|\X_i -\X_{i-1}\|_{\fro}^2 \nn\\
&\overset{}{\leq}&2L^2\|\X_i\|_{\fro}^2 + 2L^2\|\X_{i-1}\|_{\fro}^2 \overset{}{\leq} 4L^2r.\label{eq:tildedelta}
\eeq
\noi Then, we have almost surely:
\beq
(\tfrac{1 / \eta_{t+1}}{1 / \eta_{t}})^2 &=& \tfrac{1 / \eta^2_{t}+ \|\Q_t\|_{\fro}^2}{1 / \eta^2_{t}}} = \sqrt{1 + \tfrac{\|\Q_t\|_{\fro}^2}{1 / \eta^2_{t}} \nn\\
&\overset{\step{1}}{\leq}& 1 + 4\etaup^2\|\nabla f(\X_{(r_t-1)q}) + \ts \sum_{i = (r_t-1)q+1}^t \tilde{\Deltas}_i\|_{\fro}^2\nn\\
&\overset{\step{2}}{\leq} &1 + 4\etaup^2 (t - (r_t-1)q+1) \left(\|\nabla f(\X_{(r_t-1)q})\|_{\fro}^2 + \ts \sum_{i = (r_t-1)q+1}^t \|\tilde{\Deltas}_i\|_{\fro}^2\right) \nn\\
&\overset{\step{3}}{\leq} &\underbrace{
		1+
		4q\etaup^2
		\left(
		C_f^2
		+
		4(q-1)L^2r
		\right)
}_{\triangleq \hat{\eta}^2}, \nn
\eeq
where step \step{1} uses $\eta_t \leq \etaup$ and the definition of $\Q_t$ and $ \|\A - \X\A\trans\X\|_{\fro}^2 \leq  4\|\A\|_{\fro}^2, \forall \X \in \mathrm{St}(n,r)$; step \step{2} uses $\ts\|\sum_{i=1}^t \A_i\|_{\fro}^2 \leq t \cdot \sum_{i=1}^t \|\A_i\|_{\fro}^2 $; step \step{3} follows from the fact that $t - (r_t-1)q +1\leq q$, the assumption $\|\nabla f(\X)\|_{\fro} \leq C_f$, and  $\|\tilde{\Deltas}_i\|_{\fro}^2 \leq 4L^2r$ as shown in \eqref{eq:tildedelta}.
\end{proof} 
\end{lemma}

\begin{lemma}\label{lemma:A3}
	Assume $E_{t+1} \le (1 - \tau_{t+1})E_t + A_{t+1}$, where $\tau_t = t^{-p}$ with $0 < p \le 1$, and $\{E_i, A_i\}$ are two nonnegative sequences. Then we have:
	\begin{equation}
		\tau_t E_t \le 2 (E_t - E_{t+1} + A_{t+1}).\nn
	\end{equation}
\end{lemma}

\begin{proof}
	Rearranging the terms in the given assumption $E_{t+1} \le (1 - \tau_{t+1})E_t + A_{t+1}$, we first obtain an upper bound for $\tau_{t+1}E_t$:
	\begin{equation} \label{eq:tau_next_E}
		\tau_{t+1}E_t \le E_t - E_{t+1} + A_{t+1}.
	\end{equation}
	Since the step sizes $\tau_t$ and $\tau_{t+1}$ are strictly positive, we can express the target term $\tau_t E_t$ by introducing the ratio $\tfrac{\tau_t}{\tau_{t+1}}$:
	\begin{equation} 
		\tau_t E_t = \tfrac{\tau_t}{\tau_{t+1}} (\tau_{t+1} E_t) \overset{\step{1}}{\leq} \tfrac{\tau_t}{\tau_{t+1}} (E_t - E_{t+1} + A_{t+1}) \overset{\step{2}}{\leq} 2(E_t - E_{t+1} + A_{t+1}),\nn
	\end{equation}
	where the step \step{1} follows directly from \eqref{eq:tau_next_E}; step \step{2} uses $	\tfrac{\tau_t}{\tau_{t+1}} = \left( 1 + \tfrac{1}{t} \right)^p \le 2^p \leq 2$.
\end{proof}

\begin{lemma}\label{lemma:A5}
	Assume $E_{t+1} \le (1 - \tau_{t+1})E_t + A_{t+1}$, where $\tau_t = t^{-p}$ with $p \in (0, 1]$, and $\{E_i, A_i\}$ are two nonnegative sequences. Then we have:
	\begin{equation*}
		\sqrt{\tau_t} E_t \le 4 \left( \tfrac{E_t}{\sqrt{\tau_t}} - \tfrac{E_{t+1}}{\sqrt{\tau_{t+1}}} + \tfrac{A_{t+1}}{\sqrt{\tau_{t+1}}} \right).\nn
	\end{equation*}
\end{lemma}

\begin{proof}
	Dividing both sides of the given assumption $E_{t+1} \le (1 - \tau_{t+1})E_t + A_{t+1}$ by $\sqrt{\tau_{t+1}}$, we obtain:
	\beq
		\ts\sqrt{\tau_{t+1}}E_t &\leq&	\ts\tfrac{E_t}{\sqrt{\tau_{t+1}}} - \tfrac{E_{t+1}}{\sqrt{\tau_{t+1}}} + \tfrac{A_{t+1}}{\sqrt{\tau_{t+1}}}\nn\\
		\Rightarrow \ts E_t \underbrace{\left( \tfrac{1}{\sqrt{\tau_t}} - \tfrac{1}{\sqrt{\tau_{t+1}}} + 	\ts \sqrt{\tau_{t+1}} \right)}_{\triangleq C_t} &\leq& 	\ts\tfrac{E_t}{\sqrt{\tau_t}} - \tfrac{E_{t+1}}{\sqrt{\tau_{t+1}}} + \tfrac{A_{t+1}}{\sqrt{\tau_{t+1}}}.\nn
	\eeq

	Our next goal is to establish a lower bound for $C_t$ that $ \tfrac{1}{4} \sqrt{\tau_t}\le C_t $. Substituting $\tau_t = t^{-p}$, the lower bound is equivalent to:
	\begin{equation} \label{eq:algebraic_claim}
		(t+1)^{p/2} - t^{p/2} \le (t+1)^{-p/2} - \tfrac{1}{4} t^{-p/2}.
	\end{equation}
	Using $x - y = \tfrac{x^2 - y^2}{x + y}$, the left-hand side (LHS) of \eqref{eq:algebraic_claim} becomes: $\text{LHS} = \tfrac{(t+1)^p - t^p}{(t+1)^{p/2} + t^{p/2}}$. Since $p \in (0, 1]$, the function $f(x) = x^p$ is concave. By the Mean Value Theorem, $(t+1)^p - t^p \le p t^{p-1} \le t^{p-1}$. Thus, we can bound the LHS as follows:
	\begin{equation}
		\text{LHS} \le \tfrac{t^{p-1}}{(t+1)^{p/2} + t^{p/2}} = t^{-p/2} \tfrac{t^{p-1+p/2}}{(t+1)^{p/2} + t^{p/2}} = t^{-p/2} \tfrac{t^{p-1}}{(1+1/t)^{p/2} + 1}.\nn
	\end{equation}
	Given $t \ge 1$ and $p \le 1$, we have $t^{p-1} \le 1$.  To prove \eqref{eq:algebraic_claim}, we just need to verify:
	\beq \label{eq:v_ineq}
		\tfrac{t^{-p/2}}{(1+1/t)^{p/2} + 1}	&\le& (t+1)^{-p/2} - \tfrac{1}{4} t^{-p/2}\nn\\
		\Rightarrow\tfrac{1}{(1+1/t)^{p/2} + 1} &\le& (1+1/t)^{-p/2} - \tfrac{1}{4}.
	\eeq
	Let $v = (1+1/t)^{-p/2}$. Since $t \ge 1$ and $p \le 1$, then $1+1/t \leq 2$,  $p/2 \leq 1/2$, and $v \ge 2^{-1/2} = 1/\sqrt{2}$. 
	Substituting $v$ into \eqref{eq:v_ineq} (noting $(1+1/t)^{p/2} = 1/v$), we get:
	\beq
		\tfrac{1}{1/v + 1} = \tfrac{v}{v+1} \le v - \tfrac{1}{4}.\nn
	\eeq
	Since $v > 0$, multiplying both sides by $(v+1)$ yields:
	\beq
		v \le v^2 + \tfrac{3}{4}v - \tfrac{1}{4} \implies v^2 - \tfrac{1}{4}v - \tfrac{1}{4} \ge 0.\nn
	\eeq
	The positive root of the quadratic equation $x^2 - 0.25x - 0.25 = 0$ is $\tfrac{1+\sqrt{17}}{8}$. Since our variable $v \geq 1/\sqrt{2} > \tfrac{1+\sqrt{17}}{8}$, the quadratic inequality strictly holds for all valid $t$ and $p$.	Finally, having proved $\tfrac{1}{4} \sqrt{\tau_t}\leq C_t$, we conclude that	$$
		\tfrac{1}{4} \sqrt{\tau_t} E_t \le \tfrac{E_t}{\sqrt{\tau_t}} - \tfrac{E_{t+1}}{\sqrt{\tau_{t+1}}} + \tfrac{A_{t+1}}{\sqrt{\tau_{t+1}}}.	$$
\end{proof}

\noi The following two lemmas concerning the retraction operator is useful for our subsequent analysis.

\section{Proof of Section \ref{sec:com}}\label{app:sect:proof}

\subsection{Proof of Lemma \ref{prop:VR}}\label{app:prop:VR}
Before the proof, we define $\XX_t \triangleq f(\X_t)- f(\bar{\X})$, $\AA_t \triangleq \eta_t\|\Q_t\|_{\fro}^2$, $\BB_t \triangleq \eta_t^2\|\Q_t\|_{\fro}^2$, and $\SS_t \triangleq \|\G_t - \nabla f(\X_t)\|_{\fro}^2$.
\begin{proof}
	
	We define: $\Deltas_t \triangleq -\eta_t \Q_t$. Using the $L$-smooth Lipschitz property, we have the following inequalities almost surely:
	\beq
	&&f(\X_{t+1}) - f(\X_t) \nn\\
	&\leq&   \la \X_{t+1} - \X_t, \nabla f(\X_t) \ra + \tfrac{L}{2}\|\X_{t+1} - \X_t\|_{\fro}^2 \nn\\
	& \overset{\step{1}}{=} &  \la \Retr_{\X_t}(\Deltas_t) - \X_t, \nabla f(\X_t) \ra + \tfrac{L}{2}\|\Retr_{\X_t}(\Deltas_t) - \X_t\|_{\fro}^2\nn\\
	& = &  \la \Retr_{\X_t}(\Deltas_t) - \X_t - \Deltas_t, \nabla f(\X_t) \ra + \la  \Deltas_t, \nabla f(\X_t) \ra +  \tfrac{L}{2} \|\Retr_{\X_t} (\Deltas_t) - \X_t\|_{\fro}^2 \nn\\
	&\overset{\step{2}}{\leq}&  \la \Retr_{\X_t}(\Deltas_t) - \X_t - \Deltas_t, \nabla f(\X_t) \ra + \la  \Deltas_t, \nabla f(\X_t) \ra + \tfrac{\dot{k}^2L}{2}\|\Deltas_t\|_{\fro}^2  \nn\\
	&\overset{\step{3}}{\leq}&  \|\nabla f(\X_t)\|_{\fro} \|\Retr_{\X_t}(\Deltas_t) - \X_t - \Deltas_t\|_{\fro} + \la  \Deltas_t, \nabla f(\X_t) \ra + \tfrac{\dot{k}^2L}{2}\|\Deltas_t\|_{\fro}^2  \nn\\
	&\overset{\step{4}}{\leq}& \tfrac{C_f\ddot{k} + \dot{k}^2L}{2}\|\Deltas_t\|_{\fro}^2 + \la  \Deltas_t, \nabla f(\X_t) \ra,\label{eq:lproperty}
	\eeq
	where step \step{1} uses $\X_{t+1} = \Retr_{\X_t}(\Deltas_t)$; step \step{2} uses Lemma \ref{lemma:M12}; step \step{3} uses $\la \A,\B\ra \leq \|\A\|_{\fro} \|\B\|_{\fro}$; step \step{4} uses Lemma \ref{lemma:M12}.

	We now bound the term $\la  \Deltas_t, \nabla f(\X_t) \ra$. We derive:
	\beq
	&& \la  \Deltas_t, \nabla f(\X_t) \ra \nn\\
	&=&  \la  \Deltas_t, \G_t \ra + \la  \Deltas_t, \nabla f(\X_t) -\G_t \ra  \nn\\
	&=&   -\eta_t \la \Q_t, \G_t \ra + \eta_t \la\G_t - \X_t \G_t \trans \X_t, \G_t- \nabla f(\X_t)  \ra  \nn\\
	& \overset{\step{1}}{\leq} &   -\tfrac{\eta_t}{2} \|\Q_t\|_{\fro}^2  + \eta_t \la\G_t - \X_t \G_t \trans \X_t, \G_t- \nabla f(\X_t)  \ra  \nn\\
	& = &   -\tfrac{\eta_t}{4} \|\Q_t\|_{\fro}^2  +  \eta_t \left(  \la\G_t - \X_t \G_t \trans \X_t, \G_t- \nabla f(\X_t)  \ra - \tfrac{1}{4} \|\G_t - \X_t \G_t \trans \X_t\|_{\fro}^2\right)  \nn\\
	& \overset{\step{2}}{\leq}   &   - \tfrac{\eta_t}{4} \|\Q_t\|_{\fro}^2  +  \eta_t \left(  \|\G_t - \X_t \G_t \trans \X_t\|_{\fro} \cdot \|\G_t- \nabla f(\X_t)\|_{\fro} - \tfrac{1}{4} \|\G_t - \X_t \G_t \trans \X_t\|_{\fro}^2\right)  \nn\\
	& \overset{\step{3}}{\leq}  &   - \tfrac{\eta_t}{4} \|\Q_t\|_{\fro}^2  +  \etaup \|\G_t- \nabla f(\X_t)\|_{\fro}^2,\label{eq:delta}
	\eeq
	where step \step{1} uses Lemma \ref{lemma:GD:bound}; step \step{2} uses $\la \A,\B\ra \leq \|\A\|_{\fro} \|\B\|_{\fro}$; step \step{3} uses $ax - \tfrac{1}{4}x^2\leq a^2$ and $\eta_t \leq \etaup$. Combining \eqref{eq:lproperty} and \eqref{eq:delta}, we have almost surely:
	\beq
	f(\X_{t+1}) - f(\X_t) &\overset{}{\leq}& \tfrac{C_f\ddot{k} + \dot{k}^2L}{2}\|\Deltas_t\|_{\fro}^2  - \tfrac{\eta_t}{4} \|\Q_t\|_{\fro}^2  +  \etaup \|\G_t- \nabla f(\X_t)\|_{\fro}^2 \nn\\
	&\overset{\step{1}}{=}& \tfrac{(C_f\ddot{k} + \dot{k}^2L)\eta_t^2}{2}\|\Q_t\|_{\fro}^2 - \tfrac{\eta_t}{4} \|\Q_t\|_{\fro}^2  +  \etaup \|\G_t- \nabla f(\X_t)\|_{\fro}^2,\label{eq1}
	\eeq
	where step \step{1} uses the definition: $\Deltas_t \triangleq -\eta_t \Q_t$. Then using the definition of $\XX_t$ gives almost surely:
	
	\beq
	\XX_{t+1} - \XX_t \overset{}{\leq} \underbrace{\tfrac{(C_f\ddot{k} + \dot{k}^2L)}{2}}_{c_1} \cdot \underbrace{\eta_t^2\|\Q_t\|_{\fro}^2 }_{\BB_t}  - \tfrac{1}{4} \cdot \underbrace{\eta_t\|\Q_t\|_{\fro}^2}_{\AA_t} +   \etaup  \underbrace{\|\G_t- \nabla f(\X_t)\|_{\fro}^2}_{\SS_t}. \nn
	\eeq
\end{proof}

\subsection{Proof of Lemma \ref{lemma:AB}}\label{app:lemma:AB}

We define $\hat{\eta} \triangleq \sqrt{	1+	4q\etaup^2(C_f^2+4(q-1)L^2r)}$.

\begin{proof}

	\textbf{Part (a)}. We define $\eta_T \triangleq \tfrac{1}{ \sqrt{1 / \etaup^2 + \sum_{t=0}^{T-1} \|\Q_t\|_{\fro}^2 }}$. Then we derive the following results almost surely:
	\beq\label{eq:AA}
	\ts \sum_{t=0}^{T-1} \eta_t\|\Q_t\|^2_{\fro} &\overset{\step{1}}{\leq}& \sum_{t=0}^{T-1} \hat{\eta} \eta_{t+1}\|\Q_t\|^2_{\fro}\overset{}{=} \ts  \sum_{t=0}^{T-1}  \tfrac{ \hat{\eta}\|\Q_t\|^2_{\fro} }{ \sqrt{1 / \etaup^2 + \sum^t_{i=0}\|\Q_i\|^2_{\fro}} }  \nn\\
	&\overset{\step{1}}{\leq} & \ts 2 \hat{\eta} \underbrace{\sqrt{1 / \etaup^2 + \sum_{t=0}^{T-1} \|\Q_t\|_{\fro}^2 }}_{1 / \eta_T},\nn
	\eeq
	\noi where step \step{1} uses Lemma \ref{lemma:etat} and step \step{2} uses Lemma \ref{lemma:online}.

	\noi\textbf{Part (b)}. We have the following results almost surely:
	\beq\label{eq:ZZ:bound:pre::D}
	\ts \sum_{t=0}^{T-1} \eta_t^2\|\Q_t\|^2_{\fro} & \overset{\step{1}}{\leq} & \ts \sum_{t=0}^{T-1} \hat{\eta}^2\eta_{t+1}^2\|\Q_t\|^2_{\fro}  \overset{}{=}  \ts  \sum_{t=0}^{T-1} \tfrac{\hat{\eta}^2\|\Q_t\|_{\fro}^2}{1 / \etaup^2 + \sum^t_{i=0}\|\Q_i\|^2_{\fro}} \nn\\
	&\overset{\step{2}}{\leq} & \ts  \tfrac{\hat{\eta}^2(1 / \etaup^2 + \sum_{t=0}^{T-1} \|\Q_t\|_{\fro}^2)^{1/4}}{{1/4} \cdot (1 / \etaup^2)^{1/4}} 
	\overset{}{=}  \ts  4\hat{\eta}^2\sqrt{\etaup}\cdot \underbrace{(1 / \etaup^2 + \ts \sum_{t=0}^{T-1} \|\Q_t\|_{\fro}^2)^{1/4}}_{1 / \sqrt{\eta_{T}}},\nn
	\eeq
	\noi where step \step{1} uses Lemma \ref{lemma:etat}; step \step{2} uses Lemma \ref{lemma:online:pp} with $p=1/4$. 
\end{proof}

\subsection{Proof of Lemma \ref{prop:V}}\label{app:prop:V}
Before the proof, we define $\XX_t \triangleq f(\X_t)- f(\bar{\X})$, $\AA_t \triangleq \eta_t\|\Q_t\|_{\fro}^2$, $\BB_t \triangleq \eta_t^2\|\Q_t\|_{\fro}^2$, $\SS_t \triangleq \|\G_t- \nabla f(\X_t)\|_{\fro}^2$, and $\eta_T \triangleq \tfrac{1}{ \sqrt{1 / \etaup^2 + \sum_{t=0}^{T-1} \|\Q_t\|_{\fro}^2 }}$.

\begin{proof}
    From Lemma \ref{prop:VR}, we have the following results almost surely:
	\beq
	\XX_{t+1} - \XX_t \overset{}{\leq} \underbrace{\tfrac{(C_f\ddot{k} + \dot{k}^2L)}{2}}_{c_1} \cdot \underbrace{\eta_t^2\|\Q_t\|_{\fro}^2 }_{\BB_t}  - \tfrac{1}{4} \cdot \underbrace{\eta_t\|\Q_t\|_{\fro}^2}_{\AA_t} + \etaup  \underbrace{\|\G_t- \nabla f(\X_t)\|_{\fro}^2}_{\SS_t}. \label{adad}
	\eeq

\noi Then we derive
\beq
\ts \EEE[\sum_{t=0}^{T-1} \|\G_t- \nabla f(\X_t)\|^2_{\fro}] &\overset{\step{1}}{\leq}& \ts\tfrac{L^2}{b}\sum_{t=0}^{T-1} \sum_{i=(r_t-1)q}^{t-1} \EEE[\|\X_{i+1} - \X_i\|^2_{\fro}] \nn\\
&\overset{\step{2}}{\leq}& \tfrac{L^2(q-1)}{b} \ts \sum_{t=0}^{T-1}  \EEE[\|\X_{t+1} - \X_t\|^2_{\fro}] \nn\\
&\overset{\step{3}}{\leq}& \tfrac{\dot{k}^2L^2(q-1)}{b} \ts \sum_{t=0}^{T-1}  \EEE[\BB_t]  \nn\\
&\overset{\step{4}}{\leq}& \tfrac{4\hat{\eta}^2\sqrt{\bar{\eta}}\dot{k}^2L^2(q-1)}{b}  \EEE[\tfrac{1}{\sqrt{\eta_{T}}}],   \label{acac}
\eeq
where step \step{1} uses Lemma \ref{lemma:pq}; step \step{2} uses Lemma \ref{lemma:idk} \textbf{(c)}; step \step{3} uses Lemma \ref{lemma:M12} that $\|\Retr_{\X}(\Deltas)-\X\|_{\fro}\leq \dot{k} \|\Deltas\|_{\fro}$ and the definition of $\BB_t$; step \step{4} uses Lemma \ref{lemma:AB} \textbf{(b)}. Summing \ref{adad} over $t = 0,\cdots,T-1$ and taking expectations yields:
    \beq
    \EEE[\XX_T]  &\leq&   \EEE[\XX_0 +  \ts \sum_{t=0}^{T-1} \left( c_1 \BB_t -  \tfrac{1}{4} \AA_t + \etaup \SS_t \right)]\nn\\
    &\overset{\step{1}}{\leq}&  \EEE[\XX_0 +  \ts \sum_{t=0}^{T-1} \left(  -  \tfrac{1}{4} \AA_t + \etaup \SS_t \right) + 4\hat{\eta}^2\sqrt{\bar{\eta}}c_1 \tfrac{1}{\sqrt{\eta_{T}}}]  \nn\\
    &\overset{\step{2}}{\leq}& \EEE[\XX_0] - \tfrac{1}{2}(\EEE[\tfrac1{\eta_T}-\tfrac1{\etaup}]) + \EEE[\ts \sum_{t=0}^{T-1}   \etaup \SS_t  + 4\hat{\eta}^2\sqrt{\bar{\eta}}c_1 \tfrac{1}{\sqrt{\eta_{T}}}]  \nn\\
    &\overset{\step{3}}{\leq}& \EEE[\XX_0] - \tfrac{1}{2}(\EEE[\tfrac1{\eta_T}-\tfrac1{\etaup}]) + \underbrace{4\hat{\eta}^2\sqrt{\bar{\eta}}(\tfrac{\bar{\eta}\dot{k}^2L^2(q-1)}{b} + c_1)}_{c_1'}\EEE[\tfrac{1}{\sqrt{\eta_{T}}}],
    \eeq
where step \step{1} uses Lemma \ref{lemma:AB} \textbf{(b)}; step \step{2} uses the facts that $\AA_t = \eta_t\|\Q_t\|_{\fro}^2 = \eta_t(\tfrac{1}{\eta_{t+1}^2} - \tfrac{1}{\eta_t^2})\geq \tfrac{2}{\eta_{t+1}} - \tfrac{2}{\eta_t}$ and $ \sum_{t=0}^{T-1}\AA_t\geq 2\left(
\tfrac1{\eta_T}-\tfrac1{\etaup} \right) $; step \step{3} uses Inequality \eqref{acac}.

    Set $w_1 \triangleq 2\EEE[\XX_0]+\tfrac1{\etaup}$ and $w_2 \triangleq 2c_1'$, and use the Jensen's inequality $\EEE[\tfrac{1}{\sqrt{\eta_{T}}}]  \leq \sqrt{\EEE[\tfrac{1}{\eta_{T}}]}$ then we have:
    \beq
        \EEE[\tfrac{1}{\eta_{T}}] + 2\EEE[\XX_T] \leq w_1 + w_2 \sqrt{\EEE[\tfrac{1}{\eta_{T}}]} \label{aeae}
    \eeq
    Dropping the non-negative term $2\EEE[\XX_T]$ from  \eqref{aeae} and using $w_2\sqrt{\EEE[\tfrac{1}{\eta_{T}}]} \leq (\EEE[\tfrac{1}{\eta_{T}}] + w_2^2) /2$, we obtain:
    \beq
        \EEE[\tfrac{1}{\eta_{T}}] \leq 2w_1 + w_2^2 \triangleq \tfrac{1}{\underline{\eta}}.  \nn
    \eeq
    Dropping the non-negative term $\EEE[\tfrac{1}{\eta_{T}}]$ from  \eqref{aeae} and using $\EEE[\tfrac{1}{\eta_{T}}] \leq \tfrac{1}{\eta}$, we have
    \begin{align}
    \EEE[\XX_T]
    &\leq \tfrac{w_1}{2}  + \tfrac{w_2}{2}
    \sqrt{\EEE\!\left[\tfrac1{\eta_T}\right]} \leq \tfrac{w_1}{2} +  \tfrac{w_2}{2\sqrt{\underline{\eta}}} \triangleq \overline{\XX}.\nn
    \end{align}
\end{proof}

\subsection{Proof of Theorem \ref{the:it:VR}}\label{app:the:it:VR}
Before the proof, we define $\XX_t \triangleq f(\X_t)- f(\bar{\X})$, $\AA_t \triangleq \eta_t\|\Q_t\|_{\fro}^2$, $\BB_t \triangleq \eta_t^2\|\Q_t\|_{\fro}^2$, $\SS_t \triangleq \|\G_t- \nabla f(\X_t)\|_{\fro}^2$, and $\eta_T \triangleq \tfrac{1}{ \sqrt{1 / \etaup^2 + \sum_{t=0}^{T-1} \|\Q_t\|_{\fro}^2 }}$.
\begin{proof}

	\noi \textbf{Part (a).}  First, we derive $\sum_{t=0}^{T-1}	\EEE\!\left[\|\Q_t\|_{\fro}^2\right]$ is bounded by a positive constant. Before that, from Lemma \ref{lemma:idk} \textbf{(a)}, we have $\sum_{j=(r_t-1)q}^t [\tfrac{1}{\eta_j}\sum_{i=(r_j-1)q}^{j-1}\|\X_{i+1} - \X_i\|_{\fro}^2] \leq  q'\sum_{i=(r_t-1)q}^{t-1}\tfrac{1}{\eta_i}\|\X_{i+1} - \X_i\|_{\fro}^2$. Let $K\triangleq\lceil T/q\rceil$ and define $t_k\triangleq \min\{kq-1,T-1\},
    \ k=1,\ldots,K.$ Then $\{0,\ldots,T-1\}$ can be partitioned into $K$ epochs, where the last epoch may be incomplete. Applying part \textbf{(a)} in Lemma \ref{lemma:idk} to each epoch and summing the resulting inequalities yield:
    \begin{align}
        &\ts\sum_{j=0}^{T-1}
        \tfrac{1}{\eta_j}
        \ts\sum_{i=(r_j-1)q}^{j-1}
        \|\X_{i+1}-\X_i\|_{\fro}^{2}
        \nonumber\\
        &=
        \ts\sum_{k=1}^{K}
        \ts\sum_{j=(k-1)q}^{t_k}
        \ts\tfrac{1}{\eta_j}
        \ts\sum_{i=(k-1)q}^{j-1}
        \|\X_{i+1}-\X_i\|_{\fro}^{2}
        \nonumber\\
        &\leq
        q'
        \ts\sum_{k=1}^{K}
        \ts\sum_{i=(k-1)q}^{t_k-1}
        \ts\tfrac{1}{\eta_i}
        \|\X_{i+1}-\X_i\|_{\fro}^{2}
        \nonumber\\
        &\leq
        q'
        \ts\sum_{i=0}^{T-1}
        \ts\tfrac{1}{\eta_i}
        \|\X_{i+1}-\X_i\|_{\fro}^{2}. \label{aaaeee}
    \end{align}
    Then, we derive the following results from \eqref{aaaeee}
    \beq
    \ts \EEE[\sum_{t=0}^{T-1} \tfrac{1}{\eta_t}\SS_t] &\overset{\step{1}}{\leq}& \EEE[\ts\tfrac{L^2}{b}\sum_{t=0}^{T-1}\left( \tfrac{1}{\eta_t} \cdot \sum_{i=(r_t-1)q}^{t-1} \|\X_{i+1} - \X_i\|^2_{\fro}\right)] \nn\\
    &\overset{\step{2}}{\leq}& \tfrac{L^2q'}{b} \ts \sum_{t=0}^{T-1}  \EEE[\tfrac{1}{\eta_t}\|\X_{t+1} - \X_t\|^2_{\fro}] \nn\\
    &\overset{\step{3}}{\leq}& \tfrac{\dot{k}^2L^2q'}{b} \ts \sum_{t=0}^{T-1}  \EEE[\AA_t],
    \eeq
    where step \step{1} uses Lemma \ref{lemma:pq}; step \step{2} uses Lemma \ref{lemma:idk} \textbf{(b)}; step \step{3} uses Lemma \ref{lemma:M12} that $\|\Retr_{\X}(\Deltas)-\X\|_{\fro}\leq \dot{k} \|\Deltas\|_{\fro}$ and the definition of $\AA_t$. 
	
	Multiplying both sides of Inequality~\eqref{adad} by $1/\eta_t$,
taking the total expectation, and summing over $t=0,\ldots,T-1$,
we obtain
	\beq\label{eq:second-moment-start}
	\ts\tfrac{1}{4}\sum_{t=0}^{T-1}	\EEE\!\left[\|\Q_t\|_{\fro}^2\right] 
	&\leq& \ts\EEE\!\left[\sum_{t=0}^{T-1}\tfrac{f(\X_t)-f(\X_{t+1})}{\eta_t} +  \sum_{t=0}^{T-1} \left(c_1\AA_t + \tfrac{\etaup}{\eta_t} \SS_t \right)\right] \nn\\	
    &\leq&	\ts\EEE\!\left[\sum_{t=0}^{T-1}\tfrac{f(\X_t)-f(\X_{t+1})}{\eta_t}+ (c_1+\tfrac{q'\etaup L^2\dot{k}^2}{b}) \cdot \sum_{t=0}^{T-1}\AA_t\right]\nn\\
	&\leq&\ts\EEE\!\left[\sum_{t=0}^{T-1}\tfrac{f(\X_t)-f(\X_{t+1})}{\eta_t}\right]+ 2\hat{\eta}(c_1+\tfrac{q'\etaup L^2\dot{k}^2}{b}) \EEE\!\left[\tfrac{1}{\eta_T}\right].
	\eeq
	
	\noi Since $\mathrm{St}(n,r)$ is compact and $f$ is
	$C_f$-Lipschitz continuous, for every $\X\in\mathrm{St}(n,r)$, $
	0
	\leq
	f(\X)-f(\bar{\X})
	\leq
	C_f\|\X-\bar{\X}\|_{\fro}
	\leq
	2C_f\sqrt{r}.
	$ Since $\{1/\eta_t\}_{t\geq0}$ is nondecreasing, applying
	Lemma~\ref{lemma:remove:log} pathwise with $
	A_t\triangleq\tfrac{1}{\eta_t},
	\
	B_t\triangleq f(\X_t)-f(\bar{\X})
	$
	yields
	\begin{equation}\label{eq:weighted-function-bound}
		\ts \EEE\!\left[\sum_{t=0}^{T-1}\tfrac{f(\X_t)-f(\X_{t+1})}{\eta_t}\right] 	\leq \EEE\!\left[ \tfrac{\max_{0\leq t\leq T-1}f(\X_t)-f(\bar{\X})}{\eta_{T-1}}\right] 
		\leq \EEE\!\left[\tfrac{2C_f\sqrt{r}}{\eta_{T-1}}\right] \leq
		\EEE\!\left[\tfrac{2C_f\sqrt{r}}{\eta_T}\right].
	\end{equation}

	\noi Combining \eqref{eq:second-moment-start} and \eqref{eq:weighted-function-bound}, we obtain
	\begin{align}
		\ts \sum_{t=0}^{T-1}	\EEE\!\left[\|\Q_t\|_{\fro}^2\right] \leq 4\underbrace{\left(2\hat{\eta}(c_1+\tfrac{q'\etaup L^2\dot{k}^2}{b}) + 2C_f\sqrt{r}  \right)}_{\triangleq C_Q} \EEE\!\left[\tfrac{1}{\eta_T}\right] \leq \tfrac{4C_Q}{\underline{\eta}}.\label{ss1}
	\end{align}


	\noi Then, we derive:
	\beq
	&& \tfrac{1}{T}  \ts \sum_{t=0}^{T-1}  \EEE[\|\operatorname{grad}f(\X_t)\|_{\fro}^2]\nn\\
	&\overset{}{\leq}&\ts \ts \tfrac{1}{T} \sum_{t=0}^{T-1} 2\left( \EEE[\|\operatorname{grad}f(\X_t) - \Q_t\|_{\fro}^2] + \EEE[\|\Q_t\|_{\fro}^2] \right)\nn\\
	&\overset{\step{1}}{\leq}& \ts \tfrac{1}{T}  \sum_{t=0}^{T-1} 8\EEE[\|\nabla f(\X_t) - \G_t\|_{\fro}^2] +  \tfrac{1}{T} \ts \sum_{t=0}^{T-1} 2\EEE[\|\Q_t\|_{\fro}^2]\nn\\
	&\overset{\step{2}}{\leq}& \ts \tfrac{1}{T}  \cdot \tfrac{32L^2\dot{k}^2\hat{\eta}^2\sqrt{\etaup}(q-1) }{b} \EEE[\tfrac{1}{\sqrt{\eta_{T}}}] +  \tfrac{1}{T} \ts \sum_{t=0}^{T-1} 2\EEE[\|\Q_t\|_{\fro}^2]\nn\\
	&\overset{\step{3}}{\leq}& \ts \tfrac{1}{T}  \cdot \tfrac{32L^2\dot{k}^2\hat{\eta}^2\sqrt{\etaup}(q-1) }{b} \EEE[\tfrac{1}{\sqrt{\eta_{T}}}]  +  \tfrac{1}{T} \ts (\tfrac{8C_Q}{\underline{\eta}})\nn\\
	&\overset{\step{4}}{\leq}& \ts \tfrac{1}{T} \cdot (\tfrac{32L^2\dot{k}^2\hat{\eta}^2\sqrt{\etaup}(q-1) }{b\sqrt{\underline{\eta}}} + \tfrac{8C_Q}{\underline{\eta}})   = \OO(\tfrac{1}{T} ),\label{uiuiui}
	\eeq
	where step \step{1} uses $\|\A - \X\A\trans\X\|_{\fro}^2 \leq 4\|\A\|_{\fro}^2,\forall \X \in \mathrm{St}(n,r)$; step \step{2} uses Inequality \eqref{acac}; step \step{3} uses Inequality \eqref{ss1}; step \step{4} uses Lemma \ref{prop:V} and Jensen's inequality for the concave function
	$\phi(x)=\sqrt{x}$, which gives	$
	\EEE\!\left[\tfrac{1}{\sqrt{\eta_T}}\right]
	=
	\EEE\!\left[\sqrt{\tfrac{1}{\eta_T}}\right]
	\leq
	\sqrt{\EEE\!\left[\tfrac{1}{\eta_T}\right]}
	\leq
	\tfrac{1}{\sqrt{\underline{\eta}}}
	$. Let $\tau$ is sampled uniformly from $\{1,\ldots,T\}$. By the Cauchy--Schwarz inequality,
\[
\begin{aligned}
	\EEE\left[
	\|\operatorname{grad}f(\X_{\tau})\|_{\fro}
	\right]
	&=
	\ts\tfrac{1}{T}\sum_{t=1}^{T}
	\EEE\left[
	\|\operatorname{grad}f(\X_t)\|_{\fro}
	\right] \leq
	\left(
	\ts\tfrac{1}{T}\sum_{t=1}^{T}
	\EEE\left[
	\ts \|\operatorname{grad}f(\X_t)\|_{\fro}^2
	\right]
	\right)^{1/2} \\
	&=
	\ts \OO\left(\tfrac{1}{\sqrt{T}}\right).
\end{aligned}
\]

	\noi \textbf{Part (b).} Let $b$ denote the mini-batch size and $q$ the update frequency for the {\sf RGR-SPIDER} algorithm. Suppose the method requires $T = \mathcal{O}(\varepsilon^{-2})$ total iterations to converge. During execution, the full-batch gradient $\nabla f(\mathbf{y}^t)$ is evaluated whenever $t \equiv 0 \pmod q$, which incurs a cost of $N$ and occurs exactly $\lceil T/q \rceil$ times. For the remaining $(T - \lceil T/q \rceil)$ iterations where $t \not\equiv 0 \pmod q$, only a stochastic mini-batch gradient is evaluated at a cost of $2b$. Consequently, the total computational complexity is:
	\begin{align*}
		\ts N \cdot \left\lceil \tfrac{T}{q} \right\rceil + 2b \cdot \left( T - \left\lceil \tfrac{T}{q} \right\rceil \right) 
		&\ts\le N \cdot \tfrac{T+q}{q} + 2b  T\ts\overset{\step{1}}{\le} N \cdot \tfrac{T+\lceil \sqrt N \rceil}{\lceil \sqrt N \rceil} + 2\lceil \sqrt N \rceil \cdot T \\
		&\ts\overset{\step{2}}{\leq} \sqrt{N} \cdot \mathcal{O}\left(\varepsilon^{-2}\right) + N + (2\sqrt{N}+1) \cdot \mathcal{O}\left(\varepsilon^{-2}\right)\\
        &\ts\overset{}{=}  \mathcal{O}(N + \sqrt{N}\varepsilon^{-2})
	\end{align*}
	where step \step{1} follows from the specific parameter choices $q = b = \lceil \sqrt N \rceil$, and step \step{2} substitutes the iteration bound $T = \mathcal{O}(\varepsilon^{-2})$. 
\end{proof}

\subsection{Proof of Lemma \ref{prop:C}}\label{app:prop:C}
\noi We first define $\YY_t \triangleq \EEE[f(\X_{t}) - f(\bar{\X})]$, $\CC_t \triangleq \EEE[ \|\M_t- \nabla f(\X_t)\|^2_{\fro}]$, and $\DD_t \triangleq \EEE[\|\M_t - \X_t \M_t \trans \X_t\|^2_{\fro}]$.
\begin{proof}
	We define $\Deltas_t \triangleq -\gamma_t (\M_t - \X_t \M_t \trans \X_t) $ and derive the following results almost surely:
	\beq
	&&f(\X_{t+1}) - f(\X_t) \nn\\
	&\leq&   \la \X_{t+1} - \X_t, \nabla f(\X_t) \ra + \tfrac{L}{2}\|\X_{t+1} - \X_t\|_{\fro}^2 \nn\\
	&\overset{\step{1}}{=}&  \la \Retr_{\X_t}(\Deltas_t) - \X_t, \nabla f(\X_t) \ra + \tfrac{L}{2}\|\Retr_{\X_t}(\Deltas_t) - \X_t\|_{\fro}^2\nn\\
	& = &  \la \Retr_{\X_t}(\Deltas_t) - \X_t - \Deltas_t, \nabla f(\X_t) \ra + \la  \Deltas_t, \nabla f(\X_t) \ra +  \tfrac{L}{2} \|\Retr_{\X_t} (\Deltas_t) - \X_t\|_{\fro}^2 \nn\\
	&\overset{\step{2}}{\leq}&  \la \Retr_{\X_t}(\Deltas_t) - \X_t - \Deltas_t, \nabla f(\X_t) \ra + \la  \Deltas_t, \nabla f(\X_t) \ra + \tfrac{\dot{k}^2L}{2}\|\Deltas_t\|_{\fro}^2  \nn\\
	&\overset{\step{3}}{\leq}&  \|\nabla f(\X_t)\|_{\fro} \|\Retr_{\X_t}(\Deltas_t) - \X_t - \Deltas_t\|_{\fro} + \la  \Deltas_t, \nabla f(\X_t) \ra + \tfrac{\dot{k}^2L}{2}\|\Deltas_t\|_{\fro}^2  \nn\\
	&\overset{\step{4}}{\leq}& \tfrac{C_f\ddot{k} + \dot{k}^2L}{2}\|\Deltas_t\|_{\fro}^2 + \la  \Deltas_t, \nabla f(\X_t) \ra,\label{eqa1}
	\eeq
	where step \step{1} uses the update rule $\X_{t+1} = \Retr_{\X_t} (\Deltas_t)$; step \step{2} uses Lemma \ref{lemma:M12} $(i)$; step \step{3} uses the fact that $\forall \A,\B, \la \A,\B\ra \leq \|\A\|_{\fro}\|\B\|_{\fro}$; step \step{4} uses Lemma \ref{lemma:M12} $(ii)$ and $f(\cdot)$ is $C_f$-Lipschitz continue.

	We now bound the term $\la  \Deltas_t, \nabla f(\X_t) \ra$. We obtain:
	\beq
	&& \la  \Deltas_t, \nabla f(\X_t) \ra \nn\\
	&=&  \la  \Deltas_t, \M_t \ra + \la  \Deltas_t, \nabla f(\X_t) -\M_t \ra  \nn\\
	&\overset{\step{1}}{=}&   -\gamma_t \cdot \la \M_t - \X_t \M_t \trans \X_t, \M_t \ra + \gamma_t \cdot \la\M_t - \X_t \M_t \trans \X_t, \M_t- \nabla f(\X_t)  \ra  \nn\\
	&\overset{\step{2}}{\leq}&   -\tfrac{\gamma_t}{2}  \|\M_t - \X_t \M_t \trans \X_t\|_{\fro}^2  + \gamma_t \cdot \la\M_t - \X_t \M_t \trans \X_t, \M_t- \nabla f(\X_t)  \ra  \nn\\
	&\overset{\step{3}}{\leq}&   - \tfrac{\gamma_t}{4} \|\M_t - \X_t \M_t \trans \X_t\|_{\fro}^2  +  \gamma_t \cdot  \|\M_t - \X_t \M_t \trans \X_t\|_{\fro} \cdot \|\M_t- \nabla f(\X_t)\|_{\fro} \nn\\
    &&- \tfrac{\gamma_t}{4}  \|\M_t - \X_t \M_t \trans \X_t\|_{\fro}^2\nn\\
	&\overset{\step{4}}{\leq}&   - \tfrac{\gamma_t}{4} \|\M_t - \X_t \M_t \trans \X_t\|_{\fro}^2  +   \gamma_t \|\M_t- \nabla f(\X_t)\|_{\fro}^2,\label{eqa2}
	\eeq
	where step \step{1} uses the fact that $\Deltas_t = -\gamma_t (\M_t - \X_t \M_t \trans \X_t) $; step \step{2} uses Lemma \ref{lemma:GD:bound}; step \step{3} uses the fact that $\forall \A,\B, \la \A,\B\ra \leq \|\A\|_{\fro}\|\B\|_{\fro}$; step \step{4} uses $\forall x, ax-\tfrac{1}{4}x^2\leq a^2$.

	Combining \eqref{eqa1}  and \eqref{eqa2}, we have almost surely:
	\beq
	&&f(\X_{t+1}) - f(\X_t) \nn\\
	&\overset{}{\leq}& \tfrac{C_f\ddot{k} + \dot{k}^2L}{2}\|\Deltas_t\|_{\fro}^2 - \tfrac{\gamma_t}{4} \|\M_t - \X_t \M_t \trans \X_t\|_{\fro}^2  +   \gamma_t \|\M_t- \nabla f(\X_t)\|_{\fro}^2 \nn\\
	&\overset{\step{1}}{=}& \tfrac{(C_f\ddot{k} + \dot{k}^2L)\gamma_t^2}{2}\|\M_t - \X_t \M_t \trans \X_t\|_{\fro}^2 - \tfrac{\gamma_t}{4} \|\M_t - \X_t \M_t \trans \X_t\|_{\fro}^2 +   \gamma_t \|\M_t- \nabla f(\X_t)\|_{\fro}^2 \nn\\
	&\overset{\step{2}}{\leq}& (\tfrac{C_f\ddot{k}\theta + \dot{k}^2L\theta}{2}) \cdot \gamma_t\|\M_t - \X_t \M_t \trans \X_t\|_{\fro}^2 - \tfrac{\gamma_t}{4} \|\M_t - \X_t \M_t \trans \X_t\|_{\fro}^2 +    \gamma_t \|\M_t- \nabla f(\X_t)\|_{\fro}^2 \nn\\
	&\overset{}{\leq}& (\underbrace{  \tfrac{C_f\ddot{k}\theta + \dot{k}^2L\theta - \tfrac{1}{2}}{2}}_{\triangleq c_2}) \cdot \gamma_t\|\M_t - \X_t \M_t \trans \X_t\|_{\fro}^2  +  \gamma_t  \|\M_t- \nabla f(\X_t)\|_{\fro}^2,\nn
	\eeq
	where step \step{1} uses the fact that $\Deltas_t = -\gamma_t (\M_t - \X_t \M_t \trans \X_t) $; step \step{2} uses $\gamma_t = \theta t^{-1/3}$ and $t^{-1/3} \leq 1, \forall t \geq 1$.	Finally taking the expectation of above inequality gives for $t\geq1$:
	\beq
	\YY_{t+1}  - \YY_t  \leq  c_2  \gamma_t \cdot \underbrace{\EEE[\|\M_t - \X_t \M_t \trans \X_t\|_{\fro}^2]}_{\DD_t}  +  \gamma_t \cdot  \underbrace{\EEE[ \|\M_t- \nabla f(\X_t)\|_{\fro}^2 ]}_{\CC_t}.\nn
	\eeq
\end{proof}

\subsection{Proof of Theorem \ref{prop:D}}\label{app:prop:D}

\noi We first define $\M_t\triangleq \nabla f(\X_t;\xi_t) + (1-\alpha_t) \left( \M_{t-1} - \nabla f(\X_{t-1};\xi_t) \right)$, $\H_t \triangleq \M_t- \nabla f(\X_t)$, $\CC_t \triangleq \EEE[ \|\M_t- \nabla f(\X_t)\|^2_{\fro}]$, and $\DD_t \triangleq \EEE[\|\M_t - \X_t \M_t \trans \X_t\|^2_{\fro}]$.

\begin{proof}
	We have the following results:
	\beq
	&&\H_t \nn\\ &=& \M_t- \nabla f(\X_t) \nn\\
	&=& (1-\alpha_{t}) (\M_{t-1} - \nabla f(\X_{t-1};\xi_t)) + \nabla f(\X_t;\xi_t) - \nabla f(\X_t)   \nn\\
	&=& (1-\alpha_{t}) (\M_{t-1} - \nabla f(\X_{t-1})) + (1-\alpha_{t}) ( \nabla f(\X_{t-1})- \nabla  f(\X_{t-1};\xi_t)) \nn\\
    && + \nabla f(\X_t;\xi_t) - \nabla f(\X_t)   \nn\\
	&\overset{\step{1}}{=} & \underbrace{(1-\alpha_{t}) ( \nabla f(\X_{t-1})- \nabla  f(\X_{t-1};\xi_t)) + \nabla f(\X_t;\xi_t) - \nabla f(\X_t)}_{\Y_t} \nn\\
    && +  (1-\alpha_{t})\H_{t-1}, \label{eq:gf2}
	\eeq
	where step \step{1} uses the definition of $\H_t$.   We now bound the term $\|\Y_t\|_{\fro}^2$ and give the following inequalities almost surely:
	\beq
	&& \|\Y_t\|_{\fro}^2\nn\\
	& = & \|(1-\alpha_{t}) ( \nabla f(\X_{t-1}) - \nabla  f(\X_{t-1};\xi_t)) + \nabla f(\X_t;\xi_t) - \nabla f(\X_t)\|_{\fro}^2 \nn\\
	& \overset{}{=} & \|\alpha_{t} ( \nabla  f(\X_{t};\xi_t) - \nabla f(\X_{t}) ) + (1-\alpha_{t}) (\nabla  f(\X_{t};\xi_t) - \nabla  f(\X_{t-1};\xi_t) \nn\\
    &&+  \nabla  f(\X_{t-1}) - \nabla  f(\X_t) ) \|_{\fro}^2 \nn\\ 
	& \overset{\step{1}}{\leq} &  3(1-\alpha_{t})^2 \left(\|\nabla  f(\X_{t};\xi_t) - \nabla  f(\X_{t-1};\xi_t)\|_{\fro}^2 + \| \nabla  f(\X_{t-1}) - \nabla  f(\X_t)\|_{\fro}^2 \right)\nn\\
    &&+ 3\alpha_{t}^2\|  \nabla f(\X_{t}) - \nabla  f(\X_{t};\xi_t)\|_{\fro}^2 \nn\\
	& \overset{\step{2}}{\leq} &  6L^2 \|\X_t - \X_{t-1}\|_{\fro}^2 + 3\alpha_{t}^2\| \nabla f(\X_{t}) - \nabla  f(\X_{t};\xi_t)\|_{\fro}^2,\label{eq:zz}
	\eeq
	where step \step{1} uses $\|\A + \B + \C\|_{\fro}^2 \leq 3\|\A\|_{\fro}^2 + 3\|\B\|_{\fro}^2 + 3\|\C\|_{\fro}^2$; step \step{2} uses \textbf{(b)} and \textbf{(c)} in Assumption \ref{ass:fx} and $(1-\alpha_{t})^2 \leq 1$.  Using Inequality \eqref{eq:gf2} and then taking the expectation gives:
	\beq
	\CC_t & \overset{}{=} &\EEE[\|(1-\alpha_{t})\H_{t-1} + \Y_t\|_{\fro}^2]\nn\\
	&\overset{\step{1}}{=} & \EEE[\| (1-\alpha_{t})\H_{t-1}\|_{\fro}^2 + \EEE[\|\Y_t \|_{\fro}^2] \nn\\
	&\overset{\step{2}}{\leq} & \EEE[\| (1-\alpha_{t})\H_{t-1}\|_{\fro}^2 + \EEE[ 6L^2 \|\X_t - \X_{t-1}\|_{\fro}^2 + 3\alpha_{t}^2\| ( \nabla f(\X_{t}) - \nabla  f(\X_{t};\xi_t))\|_{\fro}^2] \nn\\
	&\overset{\step{3}}{\leq} & \EEE[6L^2\| \Retr_{\X_{t-1}}(-\gamma_{t-1}(\M_{t-1} - \X_{t-1} \M_{t-1} \trans \X_{t-1})) - \X_{t-1} \|_{\fro}^2] \nn\\
	&& +   (1-\alpha_t)^2\CC_{t-1}  +3\alpha_t^2\sigma^2\nn\\	
	&\overset{\step{4}}{\leq}& (1-\alpha_t)^2\CC_{t-1}   +   6L^2\gamma_{t-1}^2\dot{k}^2\underbrace{\EEE[\|\M_{t-1} - \X_{t-1} \M_{t-1} \trans \X_{t-1}\|_{\fro}^2 ]}_{\DD_{t-1}} + 3\alpha_t^2\sigma^2\nn\\		
	&\overset{\step{5}}{\leq}&(1-\alpha_t)\CC_{t-1}   +  6L^2\gamma_{t-1}^2\dot{k}^2 \DD_{t-1} + 3 \alpha_t^2\sigma^2,\nn
	\eeq
	where step \step{1} uses $ \EEE[\nabla f(\X_t; \xi_t) - \nabla f(\X_t)] = 0 $ and $ \EEE[\nabla f(\X_{t-1}; \xi_t) - \nabla f(\X_{t-1})] = 0 $; step \step{2} uses \eqref{eq:zz}; step \step{3} uses $\X_t=\Retr_{\X_{t-1}}(-\gamma_{t-1}\R_{t-1})$, the definition of $\CC_t$, and Assumption \ref{ass:sigma}; step \step{4} uses Lemma \ref{lemma:M12} that $\|\Retr_{\X}(\Deltas)-\X\|_{\fro}\leq \dot{k} \|\Deltas\|_{\fro}$; step \step{5} uses $(1-\alpha_t) \leq 1$.
\end{proof}

\subsection{Proof of Theorem \ref{the:E}}\label{app:the:E}

\begin{proof}
	 We first define $\YY_t \triangleq \EEE[f(\X_{t}) - f(\bar{\X})]$ and $\CC_t \triangleq \EEE[\|\mathbf{M}_t - \nabla f(\mathbf{X}_t)\|_{\fro}^2]$. Then we have the following results from Lemma \ref{prop:D}:
	\begin{equation}
		\CC_{t+1}\leq (1-\alpha_{t+1}) \CC_t + \LLL_{t+1},\label{eq:S:recur}
	\end{equation}
	\noi where $\LLL_t\triangleq 6L^2\gamma_{t-1}^2\dot{k}^2 \DD_{t-1} + 3 \alpha_t^2\sigma^2$. Using Inequality (\ref{eq:S:recur}) and Lemma \ref{lemma:A5}, we derive:
	\beq
	\sqrt{\alpha_t} \CC_t \leq 4 \left( \tfrac{\CC_t}{\sqrt{\alpha_t}} - \tfrac{\CC_{t+1}}{\sqrt{\alpha_{t+1}}} + \tfrac{\LLL_{t+1}}{\sqrt{\alpha_{t+1}}}\right).\nn
	\eeq
	
	\noi Since $\gamma_t = \theta\sqrt{\alpha_{t}}$, we have:
	\beq
	\gamma_{t} \CC_t \leq  4\theta^2(\tfrac{\CC_t}{\gamma_{t}} - \tfrac{\CC_{t+1}}{\gamma_{t+1}} + \tfrac{\LLL_{t+1}}{\gamma_{t+1}}).\label{eq:gammas2t}
	\eeq

	\noi We have the following results from Lemma \ref{prop:C}:
	\beq
	&&\YY_{t+1} - \YY_t \leq c_2 \gamma_t \DD_t + \gamma_t \CC_t. \label{eq111}
	\eeq
	
	\noi Adding both sides by $\delta\gamma_t\CC_t$ in Inequality \eqref{eq111}, we have:
	\beq
	&&\YY_{t+1} - \YY_t  + \delta\gamma_t\CC_t\nn\\
	&\leq& c_2 \gamma_t \DD_t + (1+\delta) \gamma_t \CC_t\nn\\
	&\overset{\step{1}}{\leq}&  c_2 \gamma_t \DD_t + 4\theta^2 (1+\delta)(\tfrac{\CC_t}{\gamma_{t}} - \tfrac{\CC_{t+1}}{\gamma_{t+1}} + \tfrac{\LLL_{t+1}}{\gamma_{t+1}})\nn\\
	&\overset{\step{2}}{\leq}&  c_2 \gamma_t \DD_t +  4\theta^2 (1+\delta)(\tfrac{\CC_t}{\gamma_{t}} - \tfrac{\CC_{t+1}}{\gamma_{t+1}})+ \tfrac{4\theta^2 (1+\delta) (6L^2\gamma_t^2\dot{k}^2\DD_t + 3\alpha_{t+1}^2\sigma^2)}{\gamma_{t+1}}\nn\\
	&\overset{\step{3}}{\leq}&   -(\underbrace{-c_2 - 24 \cdot 2^{1/3}\theta^2 L^2\dot{k}^2(1+\delta)}_{\triangleq c_3})\gamma_t \DD_t + 4\theta^2 (1+\delta) (\tfrac{\CC_t}{\gamma_{t}} - \tfrac{\CC_{t+1}}{\gamma_{t+1}} )  + \tfrac{12\theta^2  \alpha_{t+1}^2\sigma^2(1+\delta)}{\gamma_{t+1}}\nn\\
	&\overset{\step{4}}{\leq}&   -c_3\gamma_t \DD_t + 4\theta^2 (1+\delta) (\tfrac{\CC_t}{\gamma_{t}} - \tfrac{\CC_{t+1}}{\gamma_{t+1}})  + 12\theta \sigma^2(1+\delta) t^{-1}, \nn
	\eeq
	where step \step{1} uses \eqref{eq:gammas2t};  step \step{2} uses the definition of $\LLL_t$; step \step{3} uses $\gamma_{t} \leq 2^{1/3}\gamma_{t+1}$;  step \step{4} uses the definitions of $\alpha_{t}$ and $\gamma_t$.

	Using the definition $\dot{\YY}_t \triangleq \YY_t + 4\theta^2 (1+\delta) \tfrac{\CC_t}{\gamma_{t}}$, we have:
	\beq
	\delta\gamma_t \CC_t + c_3 \gamma_t \DD_t  \leq \dot{\YY}_t  - \dot{\YY}_{t+1}   + 12\theta \sigma^2(1+\delta) (t+1)^{-1}.\nn
	\eeq

	\noi Summing the above inequality and using	$\underline{c}\triangleq\min\{\delta,c_3\}>0$, $t^{-1/3}\geq T^{-1/3}$ for $1\leq t\leq T$, and $\dot{\YY}_{T+1}\geq0$, we obtain
	\beq\label{erere}
	\ts \sum_{t=1}^T \theta\underline{c} \cdot t^{-1/3} \left( \CC_t +  \DD_t  \right)
	&\leq& \ts\sum_{t=1}^T \left( \dot{\YY}_t  - \dot{\YY}_{t+1}   + 12\theta \sigma^2(1+\delta) (t+1)^{-1}\right)\nn\\
	&\leq& \ts\dot{\YY}_1 + 12\theta \sigma^2(1+\delta) \log(T+1) \triangleq \Gamma\nn\\
	\Rightarrow \ts\sum_{t=1}^T  \left( \CC_t +  \DD_t  \right) &\leq& \tfrac{\Gamma}{\theta\underline{c}} \cdot T^{1/3} \nn\\
	&=&	\mathcal{O}\!\left(\log(T+1)\cdot T^{1/3}\right).
	\eeq
	
\noi Finally, we provide iteration complexity of {\sf RGR-SRM}.
\beq
&&\tfrac{1}{T}  \ts \sum_{t=1}^{T} \EEE[ \|\operatorname{grad}f(\X_t)\|_{\fro}^2] \nn\\
&\leq& \ts\tfrac{1}{T} \sum_{t=1}^{T}2\left(\EEE[\|\operatorname{grad}f(\X_t) - \R_t\|_{\fro}^2] + \EEE[\|\R_t\|_{\fro}^2] \right) \nn\\
&\overset{\step{1}}{\leq}& \ts\tfrac{1}{T} \sum_{t=1}^{T}(8\EEE[\|\nabla f(\X_t) - \M_t\|_{\fro}^2] + 2\EEE[\|\M_t - \X_t\M_t\trans\X_t\|_{\fro}^2]) \nn\\
&\leq&  \ts\tfrac{1}{T} \sum_{t=1}^{T}8(\CC_t + \DD_t) \nn\\
&\overset{\step{2}}{=} &\OO\left(\log (T+1)\cdot T^{-2/3}\right),\label{iiiiii}
\eeq
where step \step{1} uses $\|\A - \X\A\trans\X\|_{\fro}^2 \leq 4\|\A\|_{\fro}^2,\forall \X \in \mathrm{St}(n,r)$; step \step{2} uses \eqref{erere}.  
Let $\tau$ be uniformly distributed over $[T]$ and independent of the
algorithmic randomness. The Cauchy--Schwarz inequality then gives
\begin{align}
	\EEE\left[
	\|\operatorname{grad}f(\X_{\tau})\|_{\fro}
	\right]
	&\leq
	\left(
	\EEE\left[
	\|\operatorname{grad}f(\X_{\tau})\|_{\fro}^2
	\right]
	\right)^{1/2} \nn\\
	&=
	\left(
	\ts\tfrac{1}{T}\sum_{t=1}^{T}
	\EEE\left[
	\|\operatorname{grad}f(\X_t)\|_{\fro}^2
	\right]
	\right)^{1/2} \nn\\
	&=
	\OO\left(
	\sqrt{\log(T+1)}\,T^{-1/3}
	\right)
	=
	\tilde{\OO}(T^{-1/3}).
\end{align}
\end{proof}


\subsection{Proof of Lemma \ref{prop:F}}\label{app:prop:F}

\noi We first define $\ZZ_t \triangleq \EEE[f(\X_t)- f(\bar{\X})] $, $\EE_t \triangleq \EEE[ \|\N_t- \nabla f(\X_t)\|^2_{\fro}]$, and $\FF_t \triangleq \EEE[\|\N_t - \X_t \N_t \trans \X_t\|^2_{\fro}]$.

\begin{proof}
Using the same strategy as in Appendix~\ref{app:prop:C}, and replacing $\YY_{t+1}\YY_t$, $\gamma_t$, $\CC_t$, and $\DD_t$ with $\ZZ_{t+1}\ZZ_t$, $\mu_t$, $\EE_t$, and $\FF_t$, respectively, we obtain:
	\beq
	\ZZ_{t+1} - \ZZ_{t}  \leq \underbrace{(  \tfrac{C_f\ddot{k}\theta + \dot{k}^2L\theta - \tfrac{1}{2}}{2})}_{c_2} \cdot \mu_t \cdot \underbrace{\EEE[\|\N_t - \X_t \N_t \trans \X_t\|_{\fro}^2]}_{\FF_t}  +  \mu_t \cdot  \underbrace{\EEE[ \|\N_t- \nabla f(\X_t)\|_{\fro}^2 ]}_{\EE_t}.\nn
	\eeq
\end{proof}

\subsection{Proof of Lemma \ref{prop:G}}\label{app:prop:G}

\begin{proof}
	\noi First, we define $\K_t \triangleq \N_t- \nabla f(\X_t)$, $\N_t \triangleq \beta_t\nabla f(\X_t;\xi_t) + (1-\beta_t) \N_{t-1}$, $\EE_t \triangleq \EEE[ \|\N_t- \nabla f(\X_t)\|^2_{\fro}]$, and $\FF_t \triangleq \EEE[\|\N_t - \X_t \N_t \trans \X_t\|^2_{\fro}]$. Then we derive:
	\beq
	&&\K_t \nn\\ &=& \N_t- \nabla f(\X_t) \nn\\
	&=&  (1-\beta_t) \N_{t-1} + \beta_t\nabla f(\X_t;\xi_t)  - \nabla f(\X_t)   \nn\\
	&=&  (1-\beta_t) (\N_{t-1} - \nabla f(\X_{t-1})) + \beta_t\nabla f(\X_t;\xi_t) + (1-\beta_t)\nabla f(\X_{t-1})   - \nabla f(\X_t)   \nn\\
	&\overset{\step{1}}{=} &  \underbrace{(1-\beta_t)\K_{t-1} + (1-\beta_t)( \nabla f(\X_{t-1})- \nabla f(\X_t))}_{\Z_t} + \beta_t(\nabla f(\X_t;\xi_t) - \nabla f(\X_t)),  \nn\\\label{eq:kn}
	\eeq
    where step \step{1} uses the definition of $\K_t$. Then we have the following results:
    \beq
        &&\|\Z_t\|_{\fro}^2 \nn\\ &\triangleq &\|(1-\beta_t)\K_{t-1} + (1-\beta_t)( \nabla f(\X_{t-1})- \nabla f(\X_t))\|_{\fro}^2 \nn\\
        & \overset{\step{1}}{\leq}& (1 + \tau)\|(1 - \beta_t)\K_{t-1}\|_{\fro}^2 + \left(1 + \tfrac{1}{\tau}\right)\|(1 - \beta_t)(\nabla f(\X_{t-1}) - \nabla f(\X_t))\|_{\fro}^2,\ \forall \tau > 0 \nn\\
        & \overset{\step{2}}{=}& (1 + \beta_t)(1 - \beta_t)^2\|\K_{t-1}\|_{\fro}^2 + \left(1 + \tfrac{1}{\beta_t}\right)(1 - \beta_t)^2\|\nabla f(\X_{t-1}) - \nabla f(\X_t)\|_{\fro}^2 \nn\\
        & \overset{\step{3}}{\leq}& (1 - \beta_t)\|\K_{t-1}\|_{\fro}^2 + \left(1 + \tfrac{1}{\beta_t}\right)(1 - \beta_t)^2\|\nabla f(\X_t) - \nabla f(\X_{t-1})\|_{\fro}^2 \nn\\
        & \overset{\step{4}}{\leq}& (1 - \beta_t)\|\K_{t-1}\|_{\fro}^2 + \left(1 + \tfrac{1}{\beta_t}\right)(1 - \beta_t)^2L^2\|\X_t - \X_{t-1}\|_{\fro}^2 \nn\\
        & \overset{\step{5}}{\leq} &(1 - \beta_t)\|\K_{t-1}\|_{\fro}^2 + \tfrac{L^2}{\beta_t}\|\X_t - \X_{t-1}\|_{\fro}^2,\label{eq:ez}
    \eeq
    where step \step{1} uses \( \|a + b\|_{\fro}^2 \leq (1 + \tau)\|a\|_{\fro}^2 + (1 + \tfrac{1}{\tau})\|b\|_{\fro}^2 \) for all \( \tau > 0 \); step \step{2} uses the choice \( \tau = \beta_t \); step \step{3} uses \( (1 + \beta_t)(1 - \beta_t) \leq 1 \); step \step{4} uses \( \|\nabla f(\X_t) - \nabla f(\X_{t-1})\| \leq L\|\X_t - \X_{t-1}\| \); step \step{5} uses the following inequalities:
    \[
    \left(1 + \tfrac{1}{\beta_t}\right)(1 - \beta_t)^2 = \tfrac{1}{\beta_t}(\beta_t + 1)(1 - \beta_t)^2 = \tfrac{1}{\beta_t}(1 - \beta_t^2)(1 - \beta_t) \leq \tfrac{1}{\beta_t}.
    \]

	\noi  Taking the square of Equality \eqref{eq:kn} and then taking the expectation gives:
	\beq
	&& \EE_t \nn\\& \overset{}{=} &\EEE[\| \Z_t + \beta_t(\nabla f(\X_t;\xi_t) - \nabla f(\X_t))\|_{\fro}^2]\nn\\
	&\overset{\step{1}}{=} & \EEE[\|\Z_t \|_{\fro}^2] + \EEE[\| \beta_t(\nabla f(\X_t;\xi_t) - \nabla f(\X_t))\|_{\fro}^2]  \nn\\
	&\overset{\step{2}}{\leq} & \EEE[\|\Z_t \|_{\fro}^2] + \beta_t^2\sigma^2  \nn\\
	&\overset{\step{3}}{\leq} & (1 - \beta_t)\EE_{t-1} + \tfrac{L^2}{\beta_t} \EEE[\|\X_t - \X_{t-1}\|_{\fro}^2] + \beta_t^2\sigma^2  \nn\\
    &\overset{}{=} & (1 - \beta_t)\EE_{t-1}  + \beta_t^2\sigma^2  \nn\\
    &&+ \tfrac{L^2}{\beta_t}  \EEE[ \|\Retr_{\X_{t-1}}(-\mu_{t-1}(\N_{t-1} - \X_{t-1}\N_{t-1}\trans\X_{t-1})) - \X_{t-1}\|_{\fro}^2]\nn\\
	&\overset{\step{4}}{\leq}& (1 - \beta_t)\EE_{t-1}  +   \tfrac{L^2}{\beta_t}\mu_{t-1}^2\dot{k}^2\underbrace{\EEE[\|\N_{t-1} - \X_{t-1} \N_{t-1} \trans \X_{t-1}\|_{\fro}^2 ]}_{\FF_{t-1}} + \beta_t^2\sigma^2\nn\\		
	&\overset{}{=}& (1 - \beta_t)\EE_{t-1}   +  L^2\mu_{t-1}^2\dot{k}^2\tfrac{\FF_{t-1}}{\beta_t} + \beta_t^2\sigma^2,\nn
	\eeq
	\noi where step \step{1} uses $ \mathbb{E}[\nabla f(\X_t; \xi_t) - \nabla f(\X_t)] = 0 $; step \step{2} uses $\beta^t > 0$ and Assumption \ref{ass:sigma}; step \step{3} uses Inequality \eqref{eq:ez}; step \step{4} uses Lemma \ref{lemma:M12} that $\|\Retr_{\X}(\Deltas)-\X\|_{\fro}\leq \dot{k} \|\Deltas\|_{\fro}$.

\end{proof}

\subsection{Proof of Theorem \ref{the:H}}\label{app:the:H}

\begin{proof}
We define $\ZZ_t \triangleq \EEE[f(\X_t)- f(\bar{\X})] $ and $\EE_t =  \EEE[ \|\N_t- \nabla f(\X_t)\|^2_{\fro}]$ and have the following inequality from Lemma \ref{prop:G}:
    \begin{equation}
		\EE_{t+1}\leq (1-\beta_{t+1}) \EE_t + \KKK_{t+1},\label{eq:S:recur2}
	\end{equation}
	\noi where $\KKK_t\triangleq L^2\mu_{t-1}^2\dot{k}^2\tfrac{\FF_{t-1}}{\beta_t} + \beta_t^2\sigma^2$. Using Equation (\ref{eq:S:recur2}) and Lemma \ref{lemma:A3}, we derive:
	\beq
	\beta_t \EE_t &\leq& 2 \left( \EE_t - \EE_{t+1} + \KKK_{t+1}\right)\nn\\
    &\overset{\step{1}}{=}& 2 \left( \EE_t - \EE_{t+1} \right) + 2L^2\mu_t^2\dot{k}^2\tfrac{\FF_t}{\beta_{t+1}} + 2\beta_{t+1}^2\sigma^2\nn\\
    &\overset{\step{2}}{=}& 2 \left( \EE_t - \EE_{t+1} \right) + \tfrac{(t+1)^{1/2}}{t^{1/2}} \cdot \tfrac{2\theta^2L^2\dot{k}^2\FF_t}{t^{1/2}} + \tfrac{2\sigma^2}{t+1}\nn\\
    &\overset{\step{3}}{\leq}& 2 \left( \EE_t - \EE_{t+1} \right) +  \tfrac{4\theta^2L^2\dot{k}^2\FF_t}{t^{1/2}} + \tfrac{2\sigma^2}{t+1},\label{eq:betaet}
	\eeq
	where step \step{1} uses the definition of $\KKK_t$; step \step{2} uses $\beta_t = t^{-1/2}$ and $\mu_t = \theta t^{-1/2}$; step \step{3} uses $\tfrac{(t+1)^{1/2}}{t^{1/2}}  \leq 2$.

	Then, we have the following results from Lemma \ref{prop:F}:
	\beq
	\ZZ_{t+1} - \ZZ_t \leq c_2 \mu_t \FF_t + \mu_t  \EE_t. \nn
	\eeq
	
    \noi For any \(\delta>0\), adding \(\delta\mu_t\EE_t\) to both sides of the inequality yields:
	\beq
	&&\ZZ_{t+1} - \ZZ_t + \delta\mu_t\EE_t\nn\\
	&\leq& c_2 \mu_t\FF_t + (1+\delta) \mu_t\EE_t\nn\\
	&\overset{\step{1}}{\leq}&  c_2 \mu_t \FF_t +  2\theta(1+\delta)\left(\left( \EE_t - \EE_{t+1} \right) +  \tfrac{2\theta^2L^2\dot{k}^2\FF_t}{t^{1/2}} + \tfrac{\sigma^2}{t+1}\right)\nn\\
	&\overset{}{=}&   -(\underbrace{- c_2 - 4 \theta^2 L^2\dot{k}^2(1+\delta)}_{\triangleq c_4})\mu^t \FF_t + 2\theta(1+\delta)\left( \EE_t - \EE_{t+1} \right)  + \tfrac{2\theta(1+\delta)\sigma^2}{t+1}\nn\\
	&\overset{}{=}&   -c_4\mu^t \FF_t + 2\theta(1+\delta)\left( \EE_t - \EE_{t+1} \right)  + \tfrac{2\theta(1+\delta)\sigma^2}{t+1}, \nn
	\eeq
	where step \step{1} uses \eqref{eq:betaet}. Using the definition $\dot{\ZZ}_t \triangleq \ZZ_t + 2\theta(1+\delta)\EE_t$, we have:
	\beq
	\delta \mu_t\EE_t + c_4 \mu_t \FF_t  \leq \dot{\ZZ}_t  - \dot{\ZZ}_{t+1}  + \tfrac{2\theta(1+\delta)\sigma^2}{t+1}.\nn
	\eeq

	\noi  Summing the preceding inequality from \(t=1\) to \(T\), and using \(\underline{c}'\triangleq\min\{\delta,c_4\}\), \(t^{-1/2}\geq T^{-1/2}\), and \(\dot{\ZZ}_{T+1}\geq0\), we obtain:
	
	\beq \label{qqqqq}
	\ts \sum_{t=1}^T \theta \underline{c}' \cdot t^{-1/2} \left( \EE_t +  \FF_t  \right)
	&\leq& \ts\sum_{t=1}^T \left( \dot{\ZZ}_t  - \dot{\ZZ}_{t+1}   + 2\theta(1+\delta) \sigma^2 \cdot (t+1)^{-1}\right)\nn\\
	&\leq& \dot{\ZZ}_1 + 2\theta(1+\delta)\sigma^2\cdot\log(T+1) \triangleq \Gamma'\nn\\
	\Rightarrow \ts\sum_{t=1}^T  \left( \EE_t +  \FF_t  \right) &\leq& \tfrac{\Gamma'}{\theta\underline{c}'} \cdot T^{1/2}\nn\\
	&=&\mathcal{O}\!\left(\log(T+1)\cdot T^{1/2}\right).
	\eeq
    \noi We give the following results of iteration complexity.
    \beq
    &&\tfrac{1}{T}  \ts \sum_{t=1}^{T}  \EEE[\|\operatorname{grad}f(\X_t)\|_{\fro}^2] \nn\\
    &\leq& \ts\tfrac{1}{T} \sum_{t=1}^{T}(2\EEE[\|\operatorname{grad}f(\X_t) - \S_t\|_{\fro}^2] + 2\EEE[\|\S_t\|_{\fro}^2]) \nn\\
    &\overset{\step{1}}{\leq}&  \ts\tfrac{1}{T} \sum_{t=1}^{T}(8\EEE[\|\N_t - \nabla f(\X_t)\|_{\fro}^2] + 2\EEE[\|\N_t - \X_t\N_t\trans\X_t\|_{\fro}^2]) \nn\\
    &\leq&  \ts\tfrac{1}{T} \sum_{t=1}^{T}8(\EE_t + \FF_t) \nn\\
    &\overset{\step{2}}{=} &\OO\left(\log (T+1)\cdot T^{-1/2}\right),\label{111111111111}
    \eeq
	where step \step{1} uses $\|\A - \X\A\trans\X\|_{\fro}^2 \leq 4\|\A\|_{\fro}^2,\forall \X \in \mathrm{St}(n,r)$; step \step{2} uses \eqref{qqqqq}. 
	Let $\tau\sim\operatorname{Unif}[T]$ and independent of the
	algorithmic randomness. Finally, the Cauchy--Schwarz inequality yields
	\begin{align}
		\EEE\left[
		\|\operatorname{grad}f(\X_{\tau})\|_{\fro}
		\right]
		&\leq
		\left(
		\EEE\left[
		\|\operatorname{grad}f(\X_{\tau})\|_{\fro}^2
		\right]
		\right)^{1/2} \nn\\
		&=
		\left(
		\ts\tfrac{1}{T}\sum_{t=1}^{T}
		\EEE\left[
		\|\operatorname{grad}f(\X_t)\|_{\fro}^2
		\right]
		\right)^{1/2} \nn\\
		&=
		\OO\left(
		\sqrt{\log(T+1)}\,T^{-1/4}
		\right)
		=
		\tilde{\OO}(T^{-1/4}).
	\end{align}
\end{proof}


%
%

\bibliographystyle{plain}

\bibliography{mybib}

\end{document}